\documentclass[11pt]{amsart}
\usepackage{amssymb}
\usepackage{tikz}
\usepackage{tikz-cd}
\usetikzlibrary{matrix, arrows,  decorations.markings,decorations.pathreplacing,  patterns,  plotmarks}
\usepackage{amsthm}
\usepackage{sidecap}
\usepackage[backend=bibtex, style=alphabetic, maxcitenames=5, maxbibnames=9 ]{biblatex}
\renewbibmacro{in:}{} \usepackage{enumerate}
\usepackage{subcaption}

\usepackage{hyperref}
\usepackage{cleveref}

\newtheorem{theorem}{Theorem}[subsection]
\newtheorem*{theorem*}{Theorem}
\newtheorem{lemma}[theorem]{Lemma}
\newtheorem*{lemma*}{Lemma}
\newtheorem{prop}[theorem]{Proposition}
\newtheorem{corollary}[theorem]{Corollary}
\newtheorem{definition}[theorem]{Definition}

\newtheorem{claim}[theorem]{Claim}
\newtheorem*{claim*}{Claim}
\newtheorem{example}[theorem]{Example}
\newtheorem{conjecture}[theorem]{Conjecture}

\newtheorem{remark}[theorem]{Remark}

\crefname{claim}{claim}{claims}
\crefname{prop}{proposition}{propositions}
\crefname{prop}{proposition}{propositions}
 
\newcommand{\ZZ}{\mathbb{Z}}
\newcommand{\CC}{\mathbb{C}}
\newcommand{\RR}{\mathbb{R}}
\newcommand{\NN}{\mathbb{N}}

\newcommand{\into}{\hookrightarrow}
\newcommand{\tensor}{\otimes}

\newcommand{\li}{i}
\newcommand{\lj}{j}
\newcommand{\uu}[1]{\underline{\underline{#1}}}

\DeclareMathOperator{\Crit}{Crit}
\DeclareMathOperator{\ev}{ev}
\DeclareMathOperator{\ind}{ind}
\DeclareMathOperator{\id}{id}

\DeclareMathOperator{\Area}{Area}

\DeclareMathOperator{\grad}{grad}
\newcommand{\CP}{\mathbb{CP}}

\DeclareMathOperator{\val}{val}
\DeclareMathOperator{\Flux}{Flux}

\newcommand{\CF}{{CF^\bullet}}

\newcommand{\CM}{{CM^\bullet}}
\newcommand{\HF}{{HF^\bullet}} 
\bibliography{./preambles/References}

\numberwithin{equation}{section}
\begin{document}

\title{Lagrangian Cobordisms and Lagrangian Surgery}
\author{Jeff Hicks}
\begin{abstract}
    
Lagrangian $k$-surgery modifies an immersed Lagrangian submanifold by topological $k$-surgery while removing a self-intersection.
Associated to a $k$-surgery is a Lagrangian surgery trace cobordism. 
We prove that every Lagrangian cobordism is exactly homotopic to a concatenation of suspension cobordisms and Lagrangian surgery traces. This exact homotopy can be chosen with as small Hofer norm as desired.
Furthermore, we show that each Lagrangian surgery trace bounds a holomorphic teardrop pairing the Morse cochain associated with the handle attachment to the Floer cochain generated by the self-intersection. 
We give a sample computation for how these decompositions can be used to algorithmically construct bounding cochains for Lagrangian submanifolds.
In an appendix, we describe a 2-ended embedded monotone Lagrangian cobordism which is not the suspension of a Hamiltonian isotopy following a suggestion of Abouzaid and Auroux.
 \end{abstract}
\maketitle
\setcounter{tocdepth}{1}
\tableofcontents
\section{Introduction}
\subsection{Overview: Lagrangian Cobordisms}
A Lagrangian cobordism $K:L^+\rightsquigarrow L^-$ is a Lagrangian submanifold $K\subset X\times \CC$ whose projection to $\CC$ fibers over the real axis outside of a compact set. This determines a set of ``ends'' for the Lagrangian cobordism given by Lagrangians $L^+, L^-\subset X$ (see \cref{def:cobordism}).
\begin{figure}
    \centering
    \begin{tikzpicture}

\begin{scope}[shift={(6,0)}]

\fill[gray!20]  (-4.5,4) rectangle (1.5,2);
\draw[thick] (1.5,3.5) .. controls (1,3.5) and (0,3.5) .. (-0.5,3) .. controls (-1,2.5) and (-3,2.5) .. (-3,3) .. controls (-3,3.5) and (-1,3.5) .. (-0.5,3) .. controls (0,2.5) and (1,2.5) .. (1.5,2.5);

\end{scope}

\draw[pattern=north west lines, pattern color=orange] (5.5,3) .. controls (5,2.5) and (3,2.5) .. (3,3) .. controls (3,3.5) and (5,3.5) .. (5.5,3);
\node[above] at (5.5,3.4) {$(q_-\to q_+)$};
\node[left] at (3,3) {$e$};
\node (v1) at (3,1) {$e$};
\node (v2) at (5.5,0) {$(q_-\to q_+)$};
\draw  (v1) edge[->] node[fill=white]{$T^A$} (v2);
\node at (2.5,3.5) {$\mathbb C$};
\node[left] at (2.5,1) {$CF^0(K^{0,1})$};
\node[left] at (2.5,0) {$CF^1(K^{0,1})$};
\node[fill, circle, scale=.3] at (3,3) {};
\end{tikzpicture}     \caption{An immersed null cobordism $K^{0, 1}\subset \CC^0\times \CC$ of the $0$-dimensional Whitney sphere $S^0\subset \CC^0$. 
    The patterned region represents a holomorphic teardrop pairing Morse cochain with self-intersection.}
    \label{fig:basic}
\end{figure}
These Lagrangian submanifolds were first discussed by Arnold in \cite{arnol1980lagrange}. 
Lagrangian cobordisms give an equivalence relation on Lagrangian submanifolds which is coarser than Hamiltonian isotopy (in the sense that if $L^+, L^-$ are Hamiltonian isotopic, they are also Lagrangian cobordant, with $K$ having topology $L^+\times \RR$). 

The largest collection of constructions for Lagrangian cobordisms come from \cite{haug2015lagrangian}, which introduced a $k$-surgery operation on Lagrangian submanifolds and an associated $k$-trace cobordism for $0 \leq k \leq n-1$, where $n=\dim(X)/2$.
When $k=0$, this is the Polterovich surgery introduced in \cite{polterovich1991surgery}.
Crucially, Haug provides a geometric meaning to the surgery model described by \cite{audin1994symplectic}, which writes down a $k$-trace cobordism for $0\leq k \leq n$. 
When $k=n$, the cobordism constructed in \cite{audin1994symplectic} is a null-cobordism for the Whitney sphere. 
For the simplest example, $k=n=0$, the Whitney sphere is an immersion 
\[\li^{0,0,+}:S^0=\{q^+, q^-\}\to  \CC^0=\{\text{pt}\},\]
and the Polterovich surgery trace/Whitney null-cobordism is the curve drawn in \cref{fig:basic}.
  
These surgery traces have some nice properties:
\begin{itemize}
    \item As manifolds, the ends $L^-, L^+$ of a Lagrangian surgery trace cobordism $K^{k, n-k+1}: L^+\rightsquigarrow L^-$ differ by topological surgery. The ``height'' on the Lagrangian trace cobordism given by projection to the real coordinate $\pi_\RR: K^{k, n-k+1}\to \RR$ has a single critical point of index $n-k$. 
    \item As an embedding, $L^-$ has one fewer self-intersection than $L^+$. If $K^{k, n-k+1}$ is graded, one Floer generator associated with this self-intersection lives in degree $n-k-1$. 
\end{itemize}

\subsection{Overview: Floer Cohomology and surgery}

The Floer complex of an embedded Lagrangian submanifold $\CF(L)$ is a deformation of the cochain group $\CM(L)$ by incorporating the counts of holomorphic disks with boundary on $L$ into the differential and product structures of $\CM(L)$. 
The underlying cochain group $\CM(L)$ can be singular cochains, Morse cochains, or differential forms.
The deformation enhances $\CF(L)$ with a filtered $A_\infty$ structure counting configurations of holomorphic polygons with boundary on $L$.
The algebra structure of $\CF(L)$ contains an abundance of data about the Lagrangian $L$.
Most importantly, when $\CF(L)$ is a tautologically unobstructed or weakly unobstructed-$A_\infty$ algebra it has homology groups $HF^\bullet(L)$ which are invariant under Hamiltonian isotopy. 
Similarly, whenever  $ L^+, L^-$ embedded monotone Lagrangian cobordant, Biran and Cornea \cite{biran2013lagrangian} show that  $\CF(L^+)$ and $\CF(L^-)$ are homotopy equivalent.

For monotone Lagrangian cobordisms with multiple ends, there is a known relation between Polterovich surgery, Lagrangian cobordisms, and exact sequences of Floer cohomology.
By using a neck stretching argument to compare holomorphic polygons on the summands of the connect sum to holomorphic polygons with boundary on the surgery,  \cite{fukaya2007lagrangian} shows that when Lagrangians $L^+=L^0\cup L^1$ and $L^-$ are related by Polterovich surgery that $L^-$ arises as a mapping cone on $L^0\to L^1$ in the Fukaya category. 
Because there exists a surgery trace cobordism between $L^+$ and $L^-$, we can obtain the same result from \cite{biran2013lagrangian} without appealing to neck-stretching.
Related work of \cite{nadler2020stable,tanaka2018surgery} shows that a surgery exact triangle in the Fukaya category arises from the Lagrangian Polterovich surgery trace.
In  the immersed setting, \cite{palmer2019invariance}, \citeauthor{palmer2019invariance} showed that Polterovich surgery leaves Floer cohomology invariant upon incorporating a bounding cochain.

\subsection{Overview: Monotone Lagrangian Cobordisms}
Without placing additional restrictions on the kinds of Lagrangian cobordisms we consider, the equivalence relation given by Lagrangian cobordisms is far too flexible: for example, whenever $L^+$ and $L^-$ are Lagrangian isotopic, they are embedded Lagrangian cobordant (although the Lagrangian cobordism may be non-orientable).
One way to re-impose some rigidity on our Lagrangian cobordisms is to ask that they have well-defined Floer cohomology. In particular:
\begin{itemize}
    \item if we require $K$ to be exact and embedded with $\dim (K)>6$, then $K$ has the topology of $L\times \RR$ for some $L\subset X$   (\cite{suarez2017exact}); or
    \item  by asking our Lagrangian cobordisms to be monotone and embedded  we learn that the ends $L^+, L^-$ are equivalent in the Fukaya category (\cite{biran2013lagrangian}).
\end{itemize}
These conditions are so rigid as to make it difficult to find any Lagrangian cobordisms at all! In the first setting, the only known examples of Lagrangian cobordisms are those arising from Hamiltonian isotopy; in the second case, we provide (to our knowledge) the first known example of a 2-ended monotone Lagrangian cobordism which is not a suspension in \cref{app:wideornarrow}. 
Previous work of the author (\cite{hicks2019wall}) shows that embedded Lagrangian cobordisms which are unobstructed by bounding cochain have enough flexibility to make their construction feasible, yet enough rigidity to provide meaningful Floer theoretic results. 
 \subsection{Outline and results}
This paper contributes two observations to the theory of Lagrangian cobordisms.
\begin{itemize}
    \item  The first (\Cref{par:one}) is that Lagrangian cobordisms are flexible enough that they admit a handle decomposition into standard pieces given by $k$-surgeries. 
    This result only uses standard techniques in symplectic geometry and does not contain any Floer-theoretic computations. 
    \item The second provides rigidity. In \Cref{sec:teardrops} we relate the geometry of Lagrangian cobordisms to Floer theory. For each standard $k$-surgery, we exhibit a holomorphic teardrop with boundary on our Lagrangian cobordism, which pairs the surgery handles of the cobordism with its self-intersections (\cref{fig:basic}). Obstructions arising from the holomorphic teardrop can possibly occur in the local model for $0$-surgery, and unobstructedness of this curvature term provides a rigidity criterion for Lagrangian cobordism.
\end{itemize}
If a 0-surgery is the Lagrangian connect sum at $q$ of two Lagrangians $L_0, L_1$ which intersect at several points, then the bounding cochain making the surgery trace unobstructed restricts to a bounding cochain on the immersed Lagrangian submanifold $L_0\cup L_1$. This Lagrangian equipped with bounding cochain is isomorphic to the mapping cone at the intersection point; thus the holomorphic teardrops which appear on the Lagrangian surgery trace give a construction of the \cite{fukaya2007lagrangian} exact surgery triangle.
We also provide numerous examples of Lagrangian cobordisms.
\begin{itemize} 
    \item \Cref{sec:applications} contains some examples where we use the Floer cohomology of surgery traces to compute obstructedness/unobstructedness of Lagrangian cobordisms.
    \item In \cref{app:wideornarrow} we construct a 2-ended monotone Lagrangian cobordism which is not given by a Hamiltonian isotopy. The construction comes from a proposal due to Abouzaid and Auroux to construct a neither ``wide-nor-narrow'' monotone Lagrangian submanifold, which we also include. 
\end{itemize}
We now give a more detailed outline of the paper.

\Cref{subsec:background} reviews known constructions of Lagrangian cobordisms and Floer theory of Lagrangian cobordisms. 
We focus on Biran and Cornea's theorem that ``monotone Lagrangian cobordisms provide equivalences in the Fukaya category,'' and explore the limitations that monotonicity places on this theorem. 
We give a simple example of an oriented obstructed Lagrangian submanifold whose ends are non-isomorphic in the Fukaya category.

In \cref{subsec:decompositions}, we provide some standard tools for decomposing Lagrangian cobordisms. 
In short: when decomposing a Lagrangian cobordism $K\subset X\times \CC$, we can consider decompositions that have boundaries fibering over the $X$ or $\CC$ coordinate.
In \cref{prop:tdecomposition} we show that Lagrangian cobordisms are decomposable along the $\CC$-coordinate, and in \cref{prop:splittingX} we give a method for decomposing Lagrangian cobordisms along the $X$-coordinate. 

These decompositions are used in \cref{subsec:highersurgery} to construct the standard surgery handle.
We review the parameterization of the Whitney sphere and the Lagrangian null-cobordism for the Whitney sphere in \cref{subsubsec:whitneysphere}.
The parameterization is compared to the Lagrangian surgery handle from  \cite{audin1994symplectic}.
In addition to the parameterization of the surgery handle, \cref{fig:whitneysphere,fig:handle11,fig:handle02,fig:handle03,fig:handle12} provide plots of these handles as projections to the $\CC$ coordinate and as multisections of the $T^*\RR^n$; we hope that these examples provide the reader with intuition on the construction and geometry of surgery handles.
We use this particular parameterization of the surgery handle in \cref{subsec:cobordismsaresurgery}, where we prove the main result of this section.
\begin{theorem*}[Restatement of \cref{thm:cobordismsaresurgery}]
    Let $K: L^+\rightsquigarrow L^-$ be a Lagrangian cobordism.
    $K$ is exactly homotopic to the concatenation of surgery trace cobordisms and suspensions of exact homotopies. Furthermore, the Hofer norm of this exact homotopy can be made as small as desired.
\end{theorem*}
The proof of \cref{thm:cobordismsaresurgery} shows Lagrangian cobordism can be placed into a good position by an exact homotopy.
We additionally comment on the relation between Lagrangian surgery and anti-surgery.
We acknowledge that exact homotopy is a high price to pay in order to place a Lagrangian cobordism in standard position. However, it is the strongest equivalence relation we can hope to place as the standard form must be immersed. 
We conjecture that unobstructedness of a Lagrangian submanifold $K$ is preserved by exact isotopies whose Hofer norm is smaller 
than the valuation of the bounding cochain on $L$. In particular, if $K$ is embedded and tautologically unobstructed, the conjecture states that the standard decomposition (with transverse intersections, given by \cref{thm:cobordismsarebottleneckedsurgery}) is unobstructed.

\Cref{sec:teardrops} investigates the relation between self-intersections of Lagrangian submanifolds and topology of the Lagrangian surgery trace.
In \cref{subsubsec:observations}, we show that the ends of a surgery trace cobordism can be equipped with a Morse function so that the critical points of the negative end are in index-preserving bijection with the self-intersections of the positive end. 
It follows that whenever $K: L^-\rightsquigarrow L^+$ is embedded and graded that $\chi(L^-)=\chi(L^+)$.
The remainder of the section extends this relation to Floer cohomology.
Given a two-ended \emph{monotone} Lagrangian cobordism $K:L^+\rightsquigarrow L^-$ and any other monotone Lagrangian $L'$, \cite{biran2013lagrangian} construct a homotopy equivalence $\CF(L^+, L')\to \CF(L^-, L')$, from which the above equality of Euler characteristics follows (\cref{thm:birancornea}). As a result, $L^+$ and $L^-$ are quasi-isomorphic objects in the Fukaya category.  We call this isomorphism the ``continuation map'' associated to a Lagrangian cobordism $K$.
We adopt the name continuation map from the homotopy equivalence on Floer cohomology groups induced by a Hamiltonian isotopy (which is also called the continuation map).

\Cref{subsec:doublebottleneck} lays the groundwork by defining Lagrangian cobordisms with double bottlenecks, which allow us to discuss Lagrangian cobordisms whose ends are immersed. We have a version of \cref{thm:cobordismsaresurgery} where  the pieces in the decomposition are Lagrangian cobordisms with double bottlenecks. When $(K, t^-, t^+): (L^+, H^+)\to (L^-, H^-)$ is a Lagrangian cobordism with double bottlenecks, we prove that for appropriate definition of Floer cochains,
\[\chi^{si}(L^-)=\chi^{si}(L^+)=\chi^{bot}(K,t^-, t^+).\]
We then provide a short review of immersed Floer cohomology in \cref{subsec:immersedoverview}.
In \cref{subsec:teardrops}, we conjecture that the homotopy equivalence from \cref{thm:birancornea} can be extended to the non-monotone immersed setting via a pairing of Floer cochains witnessed by a holomorphic teardrop.
\begin{theorem*}[Restatement of \cref{thm:teardropexistence}]
    The standard Lagrangian surgery trace bounds a holomorphic teardrop pairing the critical point of the surgery trace with the self-intersection in Floer cohomology. 
\end{theorem*}

\Cref{sec:applications} discusses when we can and cannot find an extension of \cref{thm:birancornea} to a (non-monotone) Lagrangian cobordism.
\Cref{subsec:initialcomputation} uses the pairing from \cref{thm:teardropexistence} to justify why the example Lagrangian cobordism given in \cref{fig:obstructedLagrangian} does not construct a continuation map in the Fukaya category. 
\Cref{subsec:samplecomputation} applies this conjectural framework to a computation yielding a continuation map associated with a Lagrangian surgery trace cobordism in specific examples (\cref{fig:obstructedpants,fig:bigdiagram}). In these examples, the holomorphic teardrop contributes to a curvature term $m^0: \Lambda\to \CF(K^{A, B})$ in Floer cohomology.
The existence of a bounding cochain or obstruction of $\CF(K^{A, B})$ either yields or precludes the construction of a continuation map on Floer cohomology between $L_{E_b}$ and $S^1_E\cup S^1_{-E}$.
In the case of \cref{fig:obstructedpants}, the Lagrangians $L_{E_b}$ and $S^1_E\cup S^1_{-E}$ are disjoint, so the result is obvious; however the comparison to the setting of \cref{fig:bigdiagram} which only differs in terms of the areas $A$ and $B$ is illustrative.

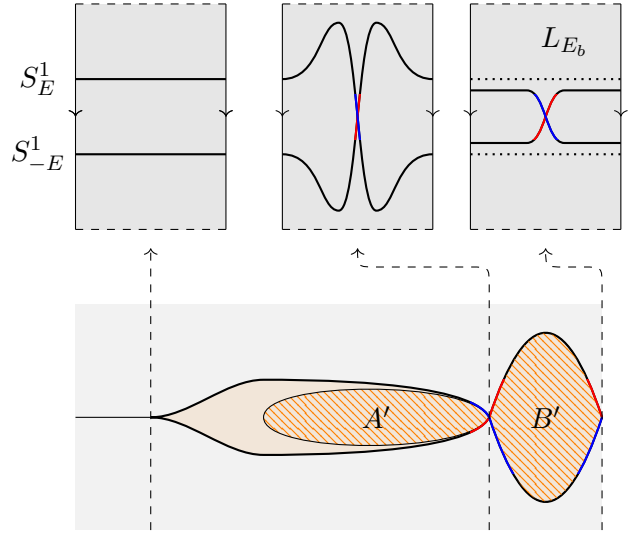
\begin{figure}
    \centering
    
\begin{tikzpicture}

\begin{scope}[shift={(-9,-12.5)},decoration={    markings,   , mark=at position 0.5 with {\arrow{>}}}]
\begin{scope}[decoration={    markings,   , mark=at position 0.5 with {\arrow{>}}}]

\fill[gray!20]  (-4,-3) rectangle (-2,0);
\draw[postaction={decorate}] (-4,0) -- (-4,-3);
\draw[postaction={decorate}](-2,0)--(-2,-3);
\draw[dashed] (-4,0) -- (-2,0)  (-2,-3) -- (-4,-3);
\end{scope}
\begin{scope}[shift={(5.25,0)}]

\fill[gray!20]  (-4,-3) rectangle (-2,0);
\draw[postaction={decorate}] (-4,0) -- (-4,-3);
\draw[postaction={decorate}](-2,0)--(-2,-3);
\draw[dashed] (-4,0) -- (-2,0)  (-2,-3) -- (-4,-3);

\end{scope}
\begin{scope}[shift={(2.75,0)}]

\fill[gray!20]  (-4,-3) rectangle (-2,0);
\draw[postaction={decorate}] (-4,0) -- (-4,-3);
\draw[postaction={decorate}](-2,0)--(-2,-3);
\draw[dashed] (-4,0) -- (-2,0)  (-2,-3) -- (-4,-3);

\end{scope}

\draw[thick] (-4,-2) -- (-2,-2);
\draw[thick] (-4,-1) -- (-2,-1);
\fill[gray!10]  (-4,-4) rectangle (3,-7);

\draw[fill=brown!20,thick] (-3,-5.5) .. controls (-2.5,-5.5) and (-2,-6) .. (-1.5,-6) .. controls (-1,-6) and (1.25,-6) .. (1.5,-5.5) .. controls (1.25,-5) and (-1,-5) .. (-1.5,-5) .. controls (-2,-5) and (-2.5,-5.5) .. (-3,-5.5);
\draw[fill=brown!20,thick] (3,-5.5) .. controls (2.5,-4) and (2,-4) .. (1.5,-5.5) .. controls (2,-7) and (2.5,-7) .. (3,-5.5);
\draw[pattern=north west lines, pattern color=orange] (3,-5.5) .. controls (2.5,-4) and (2,-4) .. (1.5,-5.5) .. controls (2,-7) and (2.5,-7) .. (3,-5.5);
\draw (-4,-5.5) -- (-3,-5.5);
\draw[pattern=north west lines, pattern color=orange] (1.5,-5.5) .. controls (1.25,-6) and (-1.5,-6) .. (-1.5,-5.5) .. controls (-1.5,-5) and (1.25,-5) .. (1.5,-5.5);

\node at (-4.5,-1) {$S^1_{E}$};
\node at (-4.5,-2) {$S^1_{-E}$};

\node at (0,-5.5) {$A'$};
\node at (2.25,-5.5) {$B'$};

\begin{scope}[shift={(-0.5,0)}]

\clip  (1.75,-5) rectangle (3.75,-6.25);
\draw[red, thick] (-1,-6) .. controls (-0.5,-6) and (1.75,-6) .. (2,-5.5) .. controls (2.5,-4) and (3,-4) .. (3.5,-5.5);
\draw[blue, thick] (3.5,-5.5) .. controls (3,-7) and (2.5,-7) .. (2,-5.5) .. controls (1.75,-5) and (-0.5,-5) .. (-1,-5);

\end{scope}

\begin{scope}[shift={(-2.25,-0.5)}]
\draw[thick] (1,-1.5) .. controls (1.5,-1.5) and (1.5,-2.25) .. (1.75,-2.25) .. controls (2,-2.25) and (2,0.25) .. (2.25,0.25) .. controls (2.5,0.25) and (2.5,-0.5) .. (3,-0.5);
\draw[thick] (1,-0.5) .. controls (1.5,-0.5) and (1.5,0.25) .. (1.75,0.25) .. controls (2,0.25) and (2,-2.25) .. (2.25,-2.25) .. controls (2.5,-2.25) and (2.5,-1.5) .. (3,-1.5);

\clip  (1.8,-0.7) rectangle (2.2,-1.3);
\draw[red, thick] (1.75,-2.25) .. controls (2,-2.25) and (2,0.25) .. (2.25,0.25);
\draw[blue, thick] (1.75,0.25) .. controls (2,0.25) and (2,-2.25) .. (2.25,-2.25);

\end{scope}
\begin{scope}[shift={(0.25,-0.5)}]
\draw[thick] (1,-1.35) .. controls (1.5,-1.35) and (1.5,-1.35) .. (1.75,-1.35) .. controls (2,-1.35) and (2,-0.65) .. (2.25,-0.65) .. controls (2.5,-0.65) and (2.5,-0.65) .. (3,-0.65);
\draw[thick] (1,-0.65) .. controls (1.5,-0.65) and (1.5,-0.65) .. (1.75,-0.65) .. controls (2,-0.65) and (2,-1.35) .. (2.25,-1.35) .. controls (2.5,-1.35) and (2.5,-1.35) .. (3,-1.35);

\draw[thick, dotted] (1,-1.5) -- (3,-1.5);
\draw[thick, dotted] (1,-0.5) -- (3,-0.5);
\clip  (1.8,-0.7) rectangle (2.2,-1.3);
\draw[red, thick] (1.75,-1.35) .. controls (2,-1.35) and (2,-0.65) .. (2.25,-0.65);
\draw[blue, thick] (1.75,-0.65) .. controls (2,-0.65) and (2,-1.35) .. (2.25,-1.35);
\end{scope}

\end{scope}

\draw[dashed, ->] (-12,-19.5) -- (-12,-15.75);
\draw[dashed, ->] (-7.5,-19.5) .. controls (-7.5,-19) and (-7.5,-17) .. (-7.5,-16.5) .. controls (-7.5,-16.25) and (-7.5,-16.25) .. (-7.5,-16.25) .. controls (-7.5,-16) and (-7.5,-16) .. (-7.75,-16) .. controls (-8,-16) and (-8.75,-16) .. (-9,-16) .. controls (-9.25,-16) and (-9.25,-16) .. (-9.25,-15.75);
\draw[dashed, ->] (-6,-19.5) .. controls (-6,-19) and (-6,-17) .. (-6,-16.75) .. controls (-6,-16.5) and (-6,-16.5) .. (-6,-16.25) .. controls (-6,-16) and (-6,-16) .. (-6.25,-16) .. controls (-6.5,-16) and (-6.5,-16) .. (-6.5,-16) .. controls (-6.75,-16) and (-6.75,-16) .. (-6.75,-15.75);
\node at (-6.5,-13) {$L_{E_b}$};
\end{tikzpicture}     \caption{A Lagrangian cobordism corresponding to a Lagrangian surgery which does not yield a continuation map in the Fukaya category. The Lagrangian cobordism is obstructed.}
    \label{fig:obstructedpants}
\end{figure}

Finally, \cref{app:wideornarrow} has some constructions of monotone Lagrangian submanifolds. In \cref{subsec:notwidenornarrow} we complete a proposal of Abouzaid-Auroux to construct a neither narrow-nor-wide Lagrangian submanifold. The ideas of this construction are employed in \cref{subsec:oneended} to construct an embedded oriented 1-ended monotone Lagrangian cobordism. We then use this cobordism, along with a surgery trace, to construct a 2-ended monotone Lagrangian cobordism in \cref{subsec:2ended}.
 
\subsection{Acknowledgments}

I would particularly like to thank Luis Haug with whom I discussed many of the ideas of this paper; much of my inspiration comes from his work in \cite{haug2015lagrangian}. Additionally, the catalyst of this project was a discussion with Nick Sheridan at the 2019 MATRIX workshop on ``Tropical geometry and mirror symmetry,'' and the major ideas of this paper were worked out during an invitation to ETH Z\"urich from Ana Cannas da Silva for a mini-course on ``Lagrangian submanifolds in toric fibrations''.
The inclusion of a construction of a neither wide nor narrow Lagrangian in \cref{subsec:widenornarrow} was at the suggestion of Mohammed Abouzaid, who along with Denis Auroux is responsible for the idea of looking at the diagonal in $X\times \bar X$ for the construction. The fleshing out of this idea and subsequent extension to monotone Lagrangian cobordisms was worked on in a series of discussions with Cheuk Yu Mak. 
I also thank two anonymous referees, whose comments made this article substantially more intelligible and coherent.
Finally, I've benefitted from many conversations with  Paul Biran, Octav Cornea, Mark Gross, Andrew Hanlon, Ailsa Keating,  and Ivan Smith on the topic of Lagrangian cobordisms.
Some figures in this paper were created using the \texttt{matplotlib} library \cite{hunter2007matplotlib}.
This work is supported by EPSRC Grant EP/N03189X/1 (Classification, Computation, and Construction: New Methods in Geometry) and ERC Grant 850713 (Homological mirror symmetry, Hodge theory, and symplectic topology). 
 
\section{Background}
\label{subsec:background}

\subsection*{Notation}
      $X$ will always denote a symplectic manifold of dimension $2n$.
      There will be many Lagrangian submanifolds of varying dimensions in this article.
      The dimension of a submanifold will be determined by reverse alphabetical order, so 
      \[\dim(J)-2=\dim(K)-1=\dim(L)=n=\dim(M)+1.\]

      In this paper, we will frequently take local coordinates for a Lagrangian submanifold $U\subset L\subset X$, and identify the Weinstein neighborhood with a neighborhood of $U\subset \RR^n\subset \CC^n$. 
      We will denote the coordinates near $U$ by $(q_i+\jmath p_i)$.
      Unless otherwise stated, all Lagrangian submanifold considered are possibly immersed.
\subsection{Lagrangian Homotopy and Cobordism}
A homotopy of Lagrangian submanifolds is a smooth map $\li_t: L\times \RR\to X$ with the property that at each $t_0\in \RR$, $i_{t_0}: L\to X$ is an immersed Lagrangian submanifold.
For each $t_0\in \RR$, a homotopy of Lagrangian submanifolds yields a closed cohomology class $\Flux_{t_0}(\li_t)\in H^1(L, \RR)$, called the flux class of $\li_t$ at $t_0$. 
The value of the flux class on chains $c\in C_1(L, \RR)$ is defined by 
\[\Flux_{t_0}(\li_t)(c):= \int_{\li_t: c\times [0, t_0]\to X}\omega.\]
If $\Flux_{t_0}(\li_t)$ is exact for all $t_0\in \RR$, we say that this homotopy is an exact homotopy. 
In the case that $\li_t$ is an exact \emph{isotopy}, there exists a time dependent Hamiltonian $H_t: X\times \RR\to \RR$ with the following properties:
\begin{itemize}
    \item $H_t|_L$ is a primitive for the flux class in the sense that $dH_{t_0}|_{L}=\Flux_{t_0}(\li_t)$.
    \item The isotopy is generated by the Hamiltonian flow $\phi_t: X\times \RR\to X$ in the sense that \[\li_t(L)=\phi_t(\li_0(L)).\]
\end{itemize}
Even when $\li_t$ is only a \emph{homotopy} we will denote by $H_t: L\times \RR\to X$ the primitive to the flux class. 
The \emph{Hofer norm} of such an isotopy is defined to be 
\[\int_\RR \left(\sup_{q\in L} H_t(q)- \inf_{q\in L} H_t(q)\right) dt,\]
which does not depend on the choice of primitive $H_t(q)$.
 Lagrangian cobordisms are an extension of the equivalence relation of exact homotopy.
\begin{definition}[\cite{arnol1980lagrange}]
	Let $L^+, L^-$ be (possibly immersed) Lagrangian submanifolds of $X$. 
	A \emph{2-ended Lagrangian cobordism} with ends $L^+, L^-$ is a (possibly immersed) Lagrangian submanifold $K\subset (X\times \CC,\omega_X+\omega_\CC)$ for which there exists a compact subset $D\subset \CC$ so that : 
            \[K\setminus( \pi_\CC^{-1}(D))=(L^+\times \RR_{>t^+}) \cup( L^-\times \RR_{<t^-}).\]
      The sets  $\RR_{>t^+}$ and  $\RR_{<t^-}$ are rays pointing along the negative and positive real axis of $\CC$, starting at some values $t_-<0<t_+$.
	We denote such a cobordism $K:L^+\rightsquigarrow L^-$. 
	\label{def:cobordism}
\end{definition}
There is a more general theory of $k$-ended Lagrangian cobordisms, although for simplicity of notation we will only discuss the 2-ended case, and always use ``Lagrangian cobordism'' to mean ``2-ended Lagrangian cobordism''. 
All of the decomposition results of this paper extend to the $k$-ended setting.
\begin{figure}
      \centering
      \begin{tikzpicture}

\begin{scope}[decoration={    markings,   mark=at position 0.66  with {\arrow{>}} , mark=at position 03.3  with {\arrow{>}}}, shift={(5,4.5)}]

\fill[gray!20]  (-4,-3) rectangle (-3,-1);
\draw[postaction={decorate}] (-4,-1) -- (-4,-3);
\draw[postaction={decorate}](-3,-1)--(-3,-3);
\draw[dashed] (-4,-1) -- (-3,-1)  (-3,-3) -- (-4,-3);
\end{scope}
\begin{scope}[decoration={    markings,   mark=at position 0.66  with {\arrow{>}} , mark=at position 0.33  with {\arrow{>}}}, shift={(3.5,4.5)}]

\fill[gray!20]  (-4,-3) rectangle (-3,-1);
\draw[postaction={decorate}] (-4,-1) -- (-4,-3);
\draw[postaction={decorate}](-3,-1)--(-3,-3);
\draw[dashed] (-4,-1) -- (-3,-1)  (-3,-3) -- (-4,-3);
\end{scope}
\begin{scope}[decoration={    markings,   mark=at position 0.66  with {\arrow{>}} , mark=at position 0.33  with {\arrow{>}}}, shift={(2,4.5)}]

\fill[gray!20]  (-4,-3) rectangle (-3,-1);
\draw[postaction={decorate}] (-4,-1) -- (-4,-3);
\draw[postaction={decorate}](-3,-1)--(-3,-3);
\draw[dashed] (-4,-1) -- (-3,-1)  (-3,-3) -- (-4,-3);
\end{scope}
\fill[gray!20]  (-2,1) rectangle (2,-1);
\draw[fill=brown!20] (2,0) .. controls (1.5,0) and (1.5,0) .. (1,0) .. controls (0.5,0) and (0.5,0.5) .. (0,0.5) .. controls (-0.5,0.5) and (-0.5,0) .. (-1,0) .. controls (-1.5,0) and (-1.5,0) .. (-2,0) .. controls (-1.5,0) and (-1.5,0) .. (-1,0) .. controls (-0.5,0) and (-0.5,-0.5) .. (0,-0.5) .. controls (0.5,-0.5) and (0.5,0) .. (1,0);

\node at (-2.5,0) {$\mathbb C$};
\draw (1,2.5) -- (2,2.5);
\draw (-2,2.5) .. controls (-1.5,1.5) and (-1.5,3.5) .. (-1,2.5);
\draw[dashed] (-1.5,1.5) -- (-1.5,-1) ;
\draw[dashed] (1.5,1.5) -- (1.5,-1);
\node[] at (1.5,2) {};

\draw (0.5,2.5) .. controls (0,3) and (0,2) .. (-0.5,2.5);
\draw[dashed] (0,1.5) -- (0,-1);
\node at (2.55,2.5) {$T^*S^1$};
\node at (0,3.75) {$K|_0$};
\node at (-1.5,3.75) {$L^-$};
\node at (1.5,3.75) {$L^+$};
\node at (0,0) {$K$};

\end{tikzpicture}       \caption{
            The shadow projection of a Lagrangian cobordism $K\subset T^*S^1 \times \CC$ to the $\CC$ factor. This Lagrangian cobordism is the suspension of a Hamiltonian isotopy.
            }
      \label{fig:shadowprojection}
\end{figure}
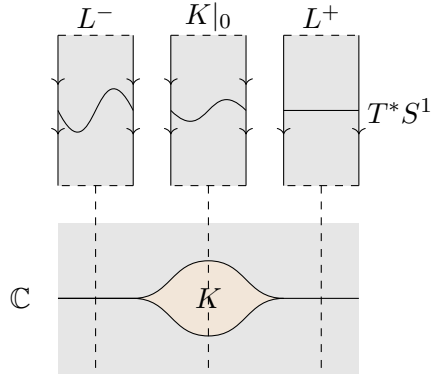
The projection to the $\CC$ component of a Lagrangian cobordism is called the \emph{shadow projection} of the Lagrangian cobordism; the infimal area of a simply connected region containing the image of the projection is called the \emph{shadow} of a Lagrangian cobordism \cite{cornea2019lagrangian}. 
We will denote this quantity by $\Area(K)$.
See \cref{fig:shadowprojection} for a diagram of a Lagrangian cobordism. 
We call the $\RR$-component of $X\times \CC$ the \emph{cobordism parameter}, and the projection to this coordinate will be denoted by $\pi_\RR: X\times \CC\to \RR$.
Given a Lagrangian cobordism $K$, we will abuse notation and use $\pi_\RR: K\to \RR$ to denote the cobordism parameter restricted to $K$.
\begin{claim}
      Let $t_0\in \RR$ be a regular value of the projection $\pi_\RR: K\to \RR$. 
      The slice of $K$ at $t_0$,
      \[K|_{t_0}:= \pi_X(\pi_\RR^{-1}(t_0)\cap K)\]
       is a (possibly immersed) Lagrangian submanifold of $X$.
\end{claim}
Lagrangian cobordance extends the equivalence relation of exactly homotopic.
We call a Lagrangian cobordism \emph{a suspension} if $\pi_\RR: K\to \RR$ has no critical points.
\begin{prop}[\cite{audin1994symplectic}]
      Let $L^+, L^-\subset X$ be two Lagrangian submanifolds. $L^+$ and $L^-$ are exactly homotopic if and only if there exists a suspension Lagrangian cobordism $K:L^+\rightsquigarrow L^-$ between these two Lagrangians. 
      \label{prop:suspension}
\end{prop}
Given an exact homotopy $\li_t: L\times \RR\to X$ whose primitive $H_t: L\to \RR$ has compact support, there is a  \emph{suspension cobordism of $\li_t$} parameterized by:
\begin{align*}
      L\times \RR\to & X\times \CC\\
            (q, t)\mapsto& (\li_t(q), t+\jmath H_t(q))\in X\times \CC.
\end{align*}
The Hofer norm of $\li_t$ is equal to the shadow of the suspension cobordism.

For the purpose of providing some geometric grounding to our discussions, we give an example of a Lagrangian cobordism which is not an exact homotopy.
We first note that every compact Lagrangian submanifold $K\subset X\times \CC$ gives an example of a Lagrangian cobordism $K:\emptyset\rightsquigarrow\emptyset$. 
While these Lagrangian cobordisms are not very interesting from a Floer-theoretic perspective (as they can be displaced from themselves), they are useful for understanding the kinds of  geometry which can appear in a  Lagrangian cobordism.
\begin{example}[Sheared Product Torus]
      \label{exam:shearedtorus}
      Consider $\CC\times \CC$ with coordinates $(q_1, p_1, q_2, p_2)$.
      The product Lagrangian torus $L_{T^2}$ is the submanifold parameterized by 
      \[(\theta_1, \theta_2)\mapsto (\cos(\theta_1)+\jmath\sin(\theta_2), \cos(\theta_1)+\jmath\sin(\theta_2))\subset \CC\times \CC.\] 
      We apply a linear symplectic transformation 
      \begin{align*}
            \phi: \CC^2\to& \CC^2\\
            (q_1, p_1, q_2, p_2)\mapsto& \left(q_1+\frac{1}{2}q_2+ \jmath \frac{4}{3}p_1-\frac{2}{3}p_2, q_2+\frac{1}{2}q_1+ \jmath  \frac{4}{3}p_2-\frac{2}{3}p_1\right)
      \end{align*}
      so that $K:=\phi(L_{T^2})$ is in general position.
      After taking this shear, $\pi_\RR: K\to \RR$ is a Morse function with four critical values corresponding to the standard maximum, saddles, and minimum on the torus.
      Several slices and the shadow of the Lagrangian cobordism are drawn in \cref{fig:lagtorus}.
\end{example}
\begin{figure}
      \centering
      \includegraphics{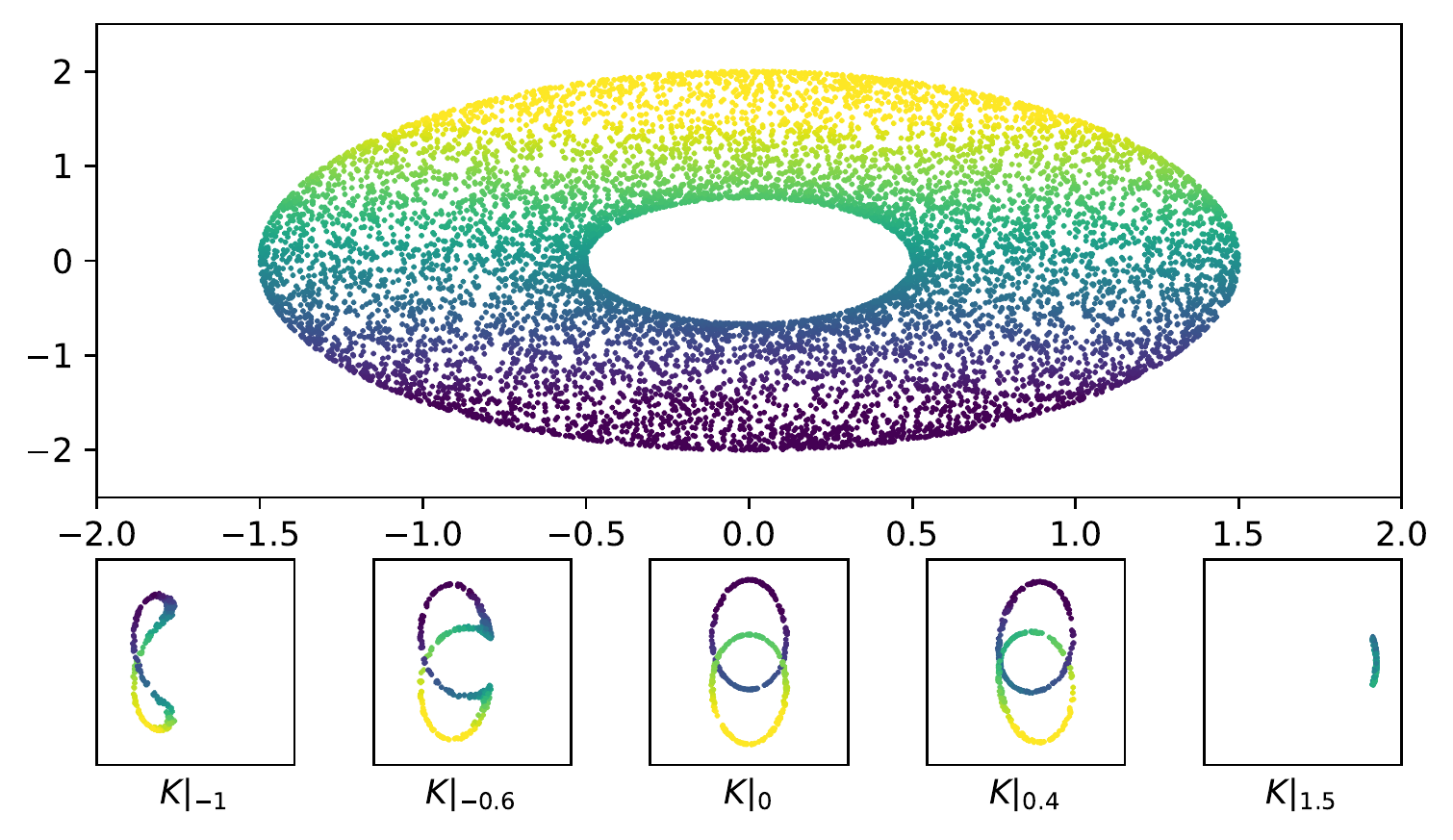}
      \caption{Scatter plots of randomly selected points lying on the sheared product torus $K=\phi(L_{T^2})$. The color corresponds to the value of $\pi_{\jmath \RR}: K\to \RR$. Top: The shadow projection of the Lagrangian submanifold $K$. Bottom: Several slices of $K=\phi(L_{T^2})$ at different real values. The critical values of the cobordism parameter $\pi_\RR: K\to \RR$ are at $\pm 1.5$ and $\pm .5$.
      We use this opportunity to highlight a common misconception: although all slices $K|_t$ are immersed, the Lagrangian $K$  is embedded.
      }
      \label{fig:lagtorus}
\end{figure}

\subsection{Previous Work: Anti-surgery}
Lagrangian anti-surgery, introduced by \cite{haug2015lagrangian}, gives a method for embedding the Lagrangian cobordism handle from  \cite{audin1994symplectic}.
Given a Lagrangian $L\subset X$, an \emph{isotropic anti-surgery disk} for $L$ is an embedded isotropic disk $i: D^{k+1}\to X$ with the following properties:
\begin{itemize}
    \item \emph{Clean Intersection:} The boundary of $D^{k+1}$ is contained in $L$. 
    Additionally, the interior of $D^{k+1}$ is disjoint from $L$, and the outward pointing vector field to $D$ is transverse to $L$. 
    \item \emph{Trivial Normal Bundle:} Over $D^{k+1}$, we can write a splitting $(TD^{k+1})^\omega= TD^{k+1}\oplus E$. 
    Furthermore, we ask that there is a symplectic trivialization $D^{k+1}\times \CC^{n-k-1}\to E$ so that over the boundary, $E|_{\partial D^{k+1}}$ is contained in $TL|_{\partial D^{k+1}}$.
\end{itemize}
Given an isotropic anti-surgery disk $D^{k+1}$ with boundary on $L$, \cite{haug2015lagrangian} produces a Lagrangian $\alpha_{D^{k+1}}(L)$, the anti-surgery of $L$ along $D^{k+1}$, along with a Lagrangian anti-surgery trace cobordism 
\[K_{\alpha_{D^{k+1}}}:L\rightsquigarrow \alpha_{D^{k+1}}(L).\]
As a manifold, $\alpha_{D^{k+1}}(L)$ differs from $L$ by $k$-surgery along $\partial D^{k+1}$, and the cobordism parameter $\pi_\RR:K_{\alpha_{D^{k+1}}}\to \RR$ provides a Morse function with a single critical point of index $k+1$.
A $k$-surgery is a modification of a manifold $L$ which replaces a subset of the form $S^k\times D^{n-k}$ with $D^{k+1}\times S^{n-k-1}$.
When compared to $L$, the anti-surgery $\alpha_{D^{k+1}}(L)$ possesses a single additional self-intersection $q_{D^n}$.
The construction is inspired by an analogous construction for Legendrian submanifolds in \cite{rizell2012legendrian}.
The terminology ``anti-surgery'' is based on the following observation: given a \emph{Lagrangian} anti-surgery disk  ${D^n}$ for $L$, the Polterovich surgery \cite{polterovich1991surgery} of $\alpha_{D^n}(L)$ at the newly created self-intersection point $q_{D^n}$ is Lagrangian isotopic to $L$.
In this sense, anti-surgery and surgery are inverse operations on Lagrangian submanifolds.
Accordingly, if $L$ arises from $L'$ by anti-surgery along a disk $D^{k+1}$, Haug states that $L$ arises from $L'$ by Lagrangian $n-k-1$ surgery. 

These surgeries and anti-surgeries appear in \cref{fig:lagtorus}, which decompose the product torus into slices related by the creation/deletion of Whitney spheres, surgeries, and anti-surgeries.
Some higher-dimensional examples of anti-surgery are in \cref{fig:handle12,fig:handle03}, which draws Lagrangians related by anti-surgery in the cotangent bundle of $\RR^2$.
In these figures, we've highlighted the isotropic anti-surgery disk corresponding modifications in red; the Lagrangians on the right-hand side all exhibit a single self-intersection at the origin.

\subsection{Floer Theoretic Properties of Lagrangian Cobordisms}
Our motivation for studying Lagrangian cobordisms comes from their Floer theoretic properties.
A fundamental result states that cobordant Lagrangians have homotopic Floer theory.
\begin{theorem}[\cite{biran2013lagrangian}]
      Suppose that $K: L^+\rightsquigarrow L^-$ is a monotone embedded Lagrangian submanifold. 
      Let $L'\subset X$ be a monotone test Lagrangian submanifold.
      Then the chain complexes $\CF(L^+, L')$ and $\CF(L^-, L')$ are chain homotopic.
      \label{thm:birancornea}
\end{theorem}

More generally, \citeauthor{biran2013lagrangian} prove that a Lagrangian cobordism with $k$-inputs $\{L^+_i\}_{i=1}^k$ and output $L^-$ yields a factorization of $L^-$ into an iterated mapping cone of the $L^+_i$. 
In the setting of two-ended monotone Lagrangian cobordisms, applications of \cref{thm:birancornea} are limited by lack of examples.
In fact,  \cite{suarez2017exact} shows that under the stronger condition that $K$ is exact, every embedded exact Lagrangian cobordism $K: L^0\rightsquigarrow L^1$ of $\dim(K)\geq 6$ has  the \emph{topology} of $L^0\times \RR$. 
It is still currently unknown if all such Lagrangian cobordisms are Hamiltonian isotopic to suspensions of Hamiltonian isotopies. 

It is expected that \cref{thm:birancornea} should extend to more general settings than monotone Lagrangian submanifolds. 
One of the broader extensions is to the class of \emph{unobstructed immersed Lagrangian cobordisms}. 
Roughly, unobstructed Lagrangians are those whose counts of holomorphic disks can be made to cancel in cohomology (see the discussion following \cref{def:boundingcochain}).
The Floer theoretic property of unobstructedness is absolutely necessary to obtain a continuation map in the Fukaya category. 
We given an example of a Lagrangian cobordism which cannot give a continuation map in the Fukaya category.
\begin{example}
      \label{exam:obstructedLagrangian}
      Let $S^1_E\subset T^*S^1$ be the section $\frac{E}{2\pi} d\theta$ which bounds an annulus of area $E$ with the zero section. Pick $E_1\neq E_2$ two positive real numbers.
      Consider the Lagrangian submanifold which is the disjoint union of two circles $S^1_{E_1} \sqcup S^1_{-E_1}\subset T^*S^1$, as drawn in \cref{fig:obstructedLagrangian}.
      By applying anti-surgery along the interval $\{0\}\times [-E_1/2\pi, E_1/2\pi]$, we obtain a Lagrangian cobordism to a Lagrangian double section of $T^*S^1$ intersecting the zero section at $(0, 0)$.
      We subsequently apply Lagrangian surgery at this self-intersection to obtain $S^1_{E_2}\sqcup S^1_{-E_2}$.
     Consider the Lagrangian cobordism built from concatenating the anti-surgery and surgery trace cobordism. This is an immersed Lagrangian cobordism which can be perturbed by Hamiltonian isotopy to make the self-intersections transverse.  Let $K\subset T^*S^1\times \CC$ be this Lagrangian cobordism, which has a single self-intersection.

      The statement of \cref{thm:birancornea} cannot be extended to a class of Lagrangian cobordism which contains $K$. 
       In this setting, $E_1\neq E_2$, and so the Lagrangians $S^1_{E_2} \sqcup S^1_{-E_2}$ and $S^1_{E_1} \sqcup S^1_{-E_1}$ are disjoint. 
      Since $S^1_{E_2} \sqcup S^1_{-E_2}$ and $S^1_{E_1} \sqcup S^1_{-E_1}$ are non-isomorphic objects of the Fukaya category, the Lagrangian cobordism $K$ cannot hope to yield a continuation map.

      We propose that the proof of \cref{thm:birancornea} fails because the Lagrangian $K$ is an obstructed Lagrangian submanifold, that is, the Floer differential on $K$ does not square to zero due to the possibility of disk/teardrop bubbling (see discussion around \cref{def:boundingcochain}). The Lagrangian cobordism $K$ bounds two holomorphic teardrops of area $A_1$ and $A_2$, whose projections under $\pi_\CC$ are drawn in \cref{fig:obstructedLagrangian}.
      When these teardrops have differing area, they collectively contribute a non-trivial $m^0$ curvature term to the Floer cohomology $\CF(K)$, which cannot be cancelled by a bounding cochain.
\end{example}
\begin{figure}
      \centering

\usetikzlibrary{matrix, arrows,  decorations.markings,decorations.pathreplacing,  patterns,  plotmarks}

\begin{tikzpicture}

\begin{scope}[shift={(1,-1)}]
\fill[gray!20]  (-3,2.5) rectangle (-1.6,-2.5);
\draw[thick](-3,-0.5) -- (-1.6,-0.5) ;
\draw[thick] (-3,0.5) -- (-1.6,0.5);
\node at (-3,0) {$\uparrow$};
\node at (-1.6,0) {$\uparrow$};
\draw (-3,2.5) -- (-3,-2.5) (-1.6,-2.5) -- (-1.6,2.5);
\draw[dotted] (-3,2.5) -- (-1.6,2.5) (-3,-2.5) -- (-1.6,-2.5);
\end{scope}

\begin{scope}[shift={(6,-1)}]
\fill[gray!20]  (-3,2.5) rectangle (-1.6,-2.5);
\draw[thick, brown] (-1.6,1.3) .. controls (-1.8,1.3) and (-1.8,1.9) .. (-2.2,1.9) .. controls (-2.3,1.9) and (-2.3,-1.9) .. (-2.4,-1.9) .. controls (-2.8,-1.9) and (-2.8,-1.3) .. (-3,-1.3);

\draw (-3,2.5) -- (-3,-2.5) (-1.6,-2.5) -- (-1.6,2.5);
\draw[dotted] (-3,2.5) -- (-1.6,2.5) (-3,-2.5) -- (-1.6,-2.5);
\node at (-3,0) {$\uparrow$};
\node at (-1.6,0) {$\uparrow$};

\draw[thick, brown] (-1.62,-1.3) .. controls (-1.82,-1.3) and (-1.82,-1.9) .. (-2.2,-1.9) .. controls (-2.3,-1.9) and (-2.3,1.9) .. (-2.4,1.9) .. controls (-2.8,1.9) and (-2.82,1.3) .. (-3.02,1.3);

\end{scope}

\begin{scope}[shift={(3,-1)}]
\fill[gray!20]  (-3,2.5) rectangle (-1.6,-2.5);
\node at (-3,0) {$\uparrow$};
\node at (-1.6,0) {$\uparrow$};
\draw (-3,2.5) -- (-3,-2.5) (-1.6,-2.5) -- (-1.6,2.5);
\draw[dotted] (-3,2.5) -- (-1.6,2.5) (-3,-2.5) -- (-1.6,-2.5);
\draw[thick, brown] (-3,0.5) .. controls (-2.8,0.5) and (-2.8,1.3) .. (-2.4,1.3) .. controls (-2.3,1.3) and (-2.3,-1.3) .. (-2.2,-1.3) .. controls (-1.8,-1.3) and (-1.8,-0.5) .. (-1.6,-0.5);
\draw[thick, brown] (-3,-0.5) .. controls (-2.8,-0.5) and (-2.8,-1.3) .. (-2.4,-1.3) .. controls (-2.3,-1.3) and (-2.3,1.3) .. (-2.2,1.3) .. controls (-1.8,1.3) and (-1.8,0.5) .. (-1.6,0.5);
\end{scope}

\begin{scope}[shift={(8.5,-1)}]
\fill[gray!20]  (-3,2.5) rectangle (-1.6,-2.5);
\node at (-3,0) {$\uparrow$};
\draw[thick](-3,1.5) -- (-1.6,1.5);
\draw (-3,2.5) -- (-3,-2.5) (-1.6,-2.5) -- (-1.6,2.5);
\draw[dotted] (-3,2.5) -- (-1.6,2.5) (-3,-2.5) -- (-1.6,-2.5);
\draw[thick] (-3,-1.5) -- (-1.6,-1.5) ;
\node at (-1.6,0) {$\uparrow$};\end{scope}

\fill[gray!20]  (-2,-4) rectangle (7.25,-7);

\draw[dashed, <-] (-1.25,-3.65) -- (-1.25,-7);
\draw[dashed, <-] (0.75,-3.65) -- (0.75,-7);
\draw[dashed, <-] (3.75,-3.65) -- (3.75,-7);
\draw[dashed, <-]  (6.3,-3.75) -- (6.3,-7);

\draw[<-] (-1,1.7) .. controls (-0.75,2.2) and (0.25,2.2) .. (0.5,1.7);
\draw[<-] (1,1.7) .. controls (1.25,2.2) and (2.25,2.2) .. (2.5,1.7);
\draw[<-] (3,1.7) .. controls (3.5,2.2) and (5.5,2.2) .. (6,1.7);
\draw[<-]   ;
\node[above] at (-0.25,2.2) {\tiny Surgery};
\node[above] at (1.75,2.2) {\tiny Exact Homotopy};
\node[above] at (4.5,2.2) {\tiny Anti-Surgery};
\node at (5.25,-3.9) {$\mathbb C$};

\node at (7.5,-2.5) {$-E_1/2\pi$};
\node at (-3,-1.5) {$-E_2/2\pi$};
\node at (-3,-0.5) {$E_2/2\pi$};
\node at (7.5,0.5) {$E_1/2\pi$};

\begin{scope}[shift={(1.7,-2.8)}]

\node at (6,1.75) {$T^*S^1$};

\end{scope}

\node[teal, circle, scale=.2, fill] at (0.7,-1) {};
\node[teal, circle, scale=.2, fill] at (3.7,-1) {};

\fill[brown!20] (3.75,-5.5) .. controls (3.25,-5) and (1,-5) .. (0.5,-5) .. controls (0,-5) and (-0.5,-5.5) .. (-1.25,-5.5) .. controls (-0.75,-5.5) and (0,-6) .. (0.5,-6) .. controls (1,-6) and (3.25,-6) .. (3.75,-5.5);
\fill[brown!20] (3.75,-5.5) .. controls (4.25,-6) and (4.5,-6) .. (4.75,-6) .. controls (5,-6) and (5.25,-5.5) .. (5.75,-5.5) .. controls (5.25,-5.5) and (5,-5) .. (4.75,-5) .. controls (4.5,-5) and (4.25,-5) .. (3.75,-5.5);

\draw (-2,-5.5) .. controls (-1.5,-5.5) and (-1.5,-5.5) .. (-1.25,-5.5) .. controls (-0.75,-5.5) and (0,-6) .. (0.5,-6) .. controls (1,-6) and (3.25,-6) .. (3.75,-5.5) .. controls (3.25,-5) and (1,-5) .. (0.5,-5) .. controls (0,-5) and (-0.75,-5.5) .. (-1.25,-5.5);
\draw (5.75,-5.5) .. controls (5.25,-5.5) and (5,-6) .. (4.75,-6) .. controls (4.5,-6) and (4.25,-6) .. (3.75,-5.5) .. controls (4.25,-5) and (4.5,-5) .. (4.75,-5) .. controls (5,-5) and (5.25,-5.5) .. (5.75,-5.5);
\draw[pattern=north west lines, pattern color=red] (0.5,-5) .. controls (0,-5) and (0,-6) .. (0.5,-6) .. controls (1,-6) and (3.25,-6) .. (3.75,-5.5) .. controls (3.25,-5) and (1,-5) .. (0.5,-5);
\draw[pattern=north west lines, pattern color=teal] (4.75,-6) .. controls (4.5,-6) and (4.25,-6) .. (3.75,-5.5) .. controls (4.25,-5) and (4.5,-5) .. (4.75,-5) .. controls (5,-5) and (5,-6) .. (4.75,-6);
\draw (5.75,-5.5) -- (7.25,-5.5);
\node at (0.75,-5.5) {$A_2$};
\node at (4.5,-5.5) {$A_1$};

\node at (-4.5,-1) {(a) Slices};
\node at (-4.5,-5.5) {(b) Shadow};

\end{tikzpicture}       \caption{An oriented immersed Lagrangian cobordism. 
      If $A_1\neq A_2$, then $K$ is an obstructed Lagrangian submanifold.
      The slashed regions correspond to the images of holomorphic teardrops with boundary on $K$ under the projection $\pi_\CC$.
      The boundary of these teardrops obstruct the solution to the Maurer-Cartan equation for $K$.}
      \label{fig:obstructedLagrangian}
\end{figure}
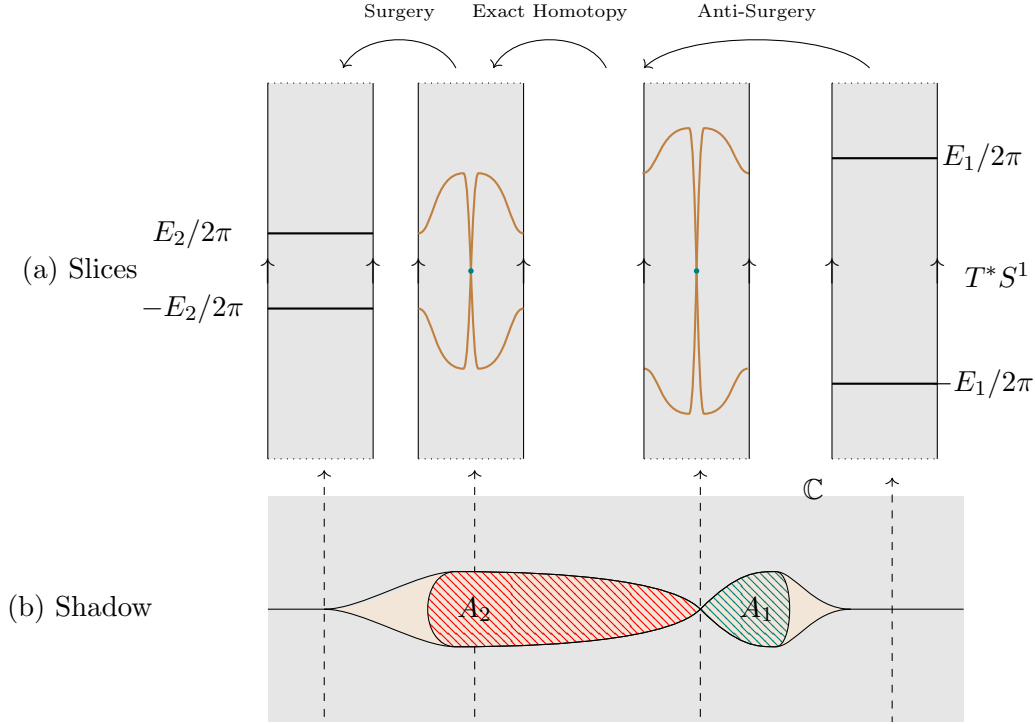
This example demonstrates that understanding when Lagrangians are unobstructed is essential for building meaningful continuation maps from Lagrangian cobordisms. We will return to \cref{exam:obstructedLagrangian} in \cref{subsec:initialcomputation}, where we examine the setting where $A_1=A_2$ and the Lagrangian cobordism $K$ is unobstructed.
In previous work, the author \cite{hicks2019wall} showed that bounding cochains for Lagrangian cobordisms could be used to compute wall-crossing transformations for Lagrangian mutations. \section{Lagrangian cobordisms are Lagrangian surgeries}
\label{par:one}
In this section we prove that every Lagrangian cobordism can be decomposed into a composition of Lagrangian surgery traces and exact homotopy suspensions. 
\Cref{subsec:decompositions} gives some constructions for decomposing Lagrangian cobordisms.
In \cref{subsec:highersurgery} we describe the standard Lagrangian surgery handle.
Finally, in \cref{subsec:cobordismsaresurgery} we show that a Lagrangian cobordism can be exactly homotoped to good position, and subsequently decomposed into surgery traces.
\subsection{Decompositions of Lagrangian Cobordisms}
\label{subsec:decompositions}
We consider two types of decompositions for Lagrangian cobordisms $K\subset X\times \CC$: across the cobordism parameter $\CC$ in \cref{subsection:truncation} and across the $X$-coordinate in \cref{subsubsec:Xdecomposition}.

\subsubsection{Gluing across $\CC$ parameter: concatenation}
Given Lagrangian cobordisms 
\begin{align*}
      K^{+0}: L^+\rightsquigarrow L^0 &&  K^{0-}: L^0\rightsquigarrow L^-
\end{align*}
 there exists a concatenation cobordism $K^{0-}\circ K^{+0}: L^+\rightsquigarrow L^-$. 
The exact homotopy class of the concatenation does not depend on the length of cylindrical component connecting the negative end of $K^{+0}$ to the positive end of $K^{0-}$.
The concatenation operation shows that Lagrangian cobordance is an equivalence relation on the set of Lagrangian submanifolds.
In the setting of differentiable manifolds, concatenation can be used to provide a decomposition of any cobordism into a sequence of standard surgery handles.

\subsubsection{Gluing across $X$ parameter: Lagrangian cobordisms with cylindrical boundary}
When describing surgery and trace cobordisms, it is also important to consider Lagrangian cobordisms  with cylindrical boundary.
We call this decomposition along the $X$-coordinate. 
\begin{definition}
      Let $L^+\subset X$ be a Lagrangian submanifold with boundary $M$. Let $M\times [0, \epsilon)\subset X$ be a collared neighborhood of the boundary. 
      A Lagrangian cobordism with cylindrical boundary $K:L^+\rightsquigarrow L^-$ is a Lagrangian submanifold $K\subset X\times \CC$ whose boundary has a collared neighborhood of the form $(M\times [0,\epsilon))\times \RR\subset X\times \CC$.
\end{definition}
We will use Lagrangians cobordisms with cylindrical boundary to describe local modifications to Lagrangian submanifolds. 
Let  $L^+=L^+_\downarrow \cup_{M} L^+_\uparrow$ be a decomposition of a Lagrangian submanifolds along a surface  $M= \partial L^+_\uparrow$. 
Given $K_{\uparrow}: L^+_\uparrow\to L^-_\uparrow$ a Lagrangian cobordism with cylindrical boundary $M$, we can obtain a Lagrangian cobordism 
\[K_\uparrow\cup_{M\times \RR} ( L^+_\downarrow\times \RR):  L^+\rightsquigarrow( L^-_\uparrow \cup_{M} L^+_\downarrow).\]
In this case, we say that the Lagrangian  $L^-:= L^-_\uparrow \cup_{M} L^+_\downarrow$ arises from modification of $L^+$ at the set $L^+_\uparrow$.
\begin{definition}
We say that $K: L^+\rightsquigarrow L^-$ decomposes across the $X$-coordinate along $M\subset L^+$ if there exist Lagrangian cobordisms with cylindrical boundary  $M$ 
\begin{align*}
      K_\uparrow:L^+_\uparrow\rightsquigarrow  L^-_\uparrow && K_\downarrow :L^+_\downarrow\rightsquigarrow L^-_\downarrow
\end{align*}
  so that $L^\pm= L^\pm_\uparrow \cup_M L^\pm_\downarrow\subset X$ is a Lagrangian submanifold and 
\[ K=K_\uparrow\cup_{M\times \RR} K_\downarrow.\]
\label{def:Xdecomposition}
\end{definition}
\begin{prop}
Suppose that $K$ decomposes across the $X$-coordinate so we may write $K=K_\uparrow\cup_{M\times \RR} K_\downarrow.$ Then $K$ is exactly homotopic to either of the following compositions of Lagrangian cobordisms.
\begin{align*}
      ( K_\uparrow\cup_{M\times \RR} (L^-_\downarrow\times \RR) )&\circ ((L^+_\uparrow \times \RR)\cup_{M\times \RR} K_\downarrow ) \\
       ((L^-_\uparrow\times \RR)\cup_{M\times \RR}K_\downarrow )&\circ (K_\uparrow \cup_{M\times \RR}(L^+_\downarrow\times \RR)).
\end{align*}
\label{prop:exchangeRelation}
\end{prop}
\begin{proof}
      Observe that the Hamiltonian $\pi_{\jmath\RR}:X\times \CC\to \RR$ restricts to the constant zero function on the collar boundary $(M\times [0, \epsilon))\times \RR$ of $K_\uparrow$. The flow of the associated Hamiltonian vector field on $X\times \CC$ is leftward translation of the cobordism parameter. Similarly, the Hamiltonian $-\pi_{\jmath\RR}: X\times \CC\to \RR$ restricts to the constant zero function on the collar boundary of $K_\downarrow$; its associated Hamiltonian flow is rightward translation. Let $\rho:(-\epsilon, \epsilon)\to (-1, 1)$ be a smooth increasing function which is constantly $\pm 1$ in a neighborhood  of $\pm \epsilon$. Let $\lj:K\to X\times \CC$ parameterize our Lagrangian cobordism. The Lagrangian isotopy
      \begin{align*}
            \lj_t: K\times \RR\to& X\times \CC\\
            q\mapsto& \left\{ \begin{array}{cc} (\pi_X\circ \lj(q), \pi_\CC \circ \lj(q)+t)& \text{ if $q\in K_\uparrow\setminus (M\times [0, \epsilon))\times \RR$}\\
            (\pi_X\circ \lj(q), \pi_\CC \circ \lj(q)-t)& \text{ if $q\in K_\downarrow\setminus( M\times (\epsilon, 0])\times \RR$}\\
            (\pi_X\circ \lj(q), \pi_\CC \circ \lj(q)+ \rho\cdot t& \text{ if $q\in (M\times (-\epsilon, \epsilon))\times \RR$}
            \end{array}\right.
      \end{align*} 
      is an exact homotopy with primitive $\pi_{\jmath \RR}$ on $K_\uparrow\setminus ((M\times [0, \epsilon))\times \RR)$ and $-\pi_{\jmath \RR}$ on $K_\downarrow\setminus(\times \RR)$. The homotopy fixes the image of  $(M\times \epsilon)\times \RR$.
\end{proof}
\subsubsection{Some general tools for Lagrangian cobordisms}
As we will work with immersed Lagrangian cobordisms, we need the following replacement of Weinstein neighborhoods. 
\begin{definition}
      Let $X$ and $Y$ be symplectic manifolds.
      A \emph{local symplectic embedding} is a map $\phi: Y\to X$ so that around every $y\in Y$ there exists a neighborhood $U$ on which $\phi|_U: U\to X$ is a symplectic embedding.
      Let $\li: L\to X$ be an immersed Lagrangian submanifold. A \emph{local Weinstein neighborhood} is a map from a neighborhood of the zero section in the cotangent bundle of $L$ to $X$
      \[\phi: B^*_\epsilon L\to X\]
      which is a local symplectic embedding, and whose restriction to the zero section makes the diagram
      \[\begin{tikzcd}
            L\arrow{r}{\li} \arrow{d}{0}& X\\
            B^*_\epsilon L \arrow{ur}{\phi}
      \end{tikzcd}
      \] 
      commute.
\end{definition}

For both decomposition along the $X$ coordinate and cobordism parameter, we will use the following generalization of \cref{prop:suspension}.
\begin{lemma}[Generalized Suspension]
      Let $I$ be a finite indexing set. 
      Suppose that we are given for each $\alpha\in I$ an  exact Lagrangian homotopy $\li^\alpha_t : L^\alpha\times \RR\to X^\alpha$, whose flux primitive is $H^\alpha_t: L^\alpha\times \RR\to \RR$.
      Let $Y$ be another manifold, and pick functions   $\rho^\alpha: Y\to \RR$.
      For fixed $q^\alpha\in L^\alpha$, let $(y, d\rho^\alpha(y)):Y\to T^*Y$ be the parameterization of the exact Lagrangian section whose primitive is  $Y\to \RR, y\mapsto H_{\rho^\alpha(y)}^\alpha(q^\alpha)$.
      Then 
      \begin{align*}
            \lj_{\rho^I}:\left(\prod_{\alpha\in I} L^\alpha\right)\times Y\to& \left(\prod_{\alpha\in I}X^\alpha\right)\times T^*Y\\
            (q^\alpha, y) \mapsto& \left(\li_{\rho^\alpha(y)}^\alpha(q^\alpha), y, \sum_{\alpha\in I} H_{\rho^\alpha(y)}^\alpha(q^\alpha)d\rho^\alpha(y)\right)
      \end{align*}
      parameterizes a Lagrangian submanifold.
      \label{claim:suspensiongeneralization}
\end{lemma}
We note that the usual suspension construction is recovered by taking $I=\{1\}$, $Y=\RR$, and $\rho(y)=t$.
\begin{proof}
      For convenience, write $\lj$ for $\lj_{\rho^I}$.
      Pick local coordinates $y_1 \ldots y_n$ for $Y$, so that we may locally identify $T^*Y$ with $\CC^n$ and write sections of the cotangent bundle as 
      \begin{align*}
            Y\to& T^*Y\\
            y_i \mapsto& \left( y_i+\jmath \sum_{\alpha\in I} H^\alpha_{\rho^\alpha(y)}(q^\alpha) \frac{\partial \rho^\alpha}{\partial y_i} \right).
      \end{align*}
      Let $\frac{d}{dt}\li_t^\alpha\in TX^\alpha$ be the vector field along the image of $\li_t^\alpha$ associated to the isotopy $\li_t^\alpha$.
      Let $(q_j^\alpha)$ denote local coordinate on the $L^\alpha$.
      Since $\li^\alpha_t$ is a exact Lagrangian homotopy with flux primitive given by $H^\alpha_t$, we have that:
      \[\iota_{\frac{d}{dt}\li_t^\alpha}\omega_{X^\alpha}|_{L^\alpha}= \sum_i \partial_{q_j} H^\alpha_t dq_j^\alpha.\]
      Let $\partial_{v^\alpha}\in TL^\alpha$, and $\partial_{y_i}\in TY$ be vectors. 
      We compute   $\lj_*(\partial_{y_j})$ and $\lj_*(\sum_{\alpha\in I} c_\alpha \partial_{v^\alpha})$ :
      \begin{align*}
            \lj_*(\partial_{y_j})=& \left(\frac{\partial \rho^\alpha}{\partial y_j} \frac{d}{dt}\li^\alpha_{\rho^\alpha(y)} , \delta_{ij} +\jmath\sum_{\alpha\in I}\left( \frac{\partial^2\rho^\alpha}{\partial y_i \partial y_j} H_{\rho^\alpha(y)}+  \frac{\partial\rho^\alpha}{\partial y_i}\frac{\partial \rho^\alpha}{\partial y_j}\frac{d}{dt}H^\alpha_{\rho^\alpha(y)}\right)\right)\\
            \lj_*\left(\sum_{\alpha\in I} c_\alpha \partial_{v^\alpha}\right)=& \left(\sum_{\alpha\in I}c_\alpha(\li_{\rho(s)})_* \partial_{v^\alpha},\jmath \sum_{\alpha\in I} c_\alpha \frac{\partial \rho^\alpha}{\partial y_i} \partial_v^\alpha H^\alpha_{\rho^\alpha(y)}\right)
      \end{align*} 
      We compute the vanishing of the symplectic form on pairs of vectors $\sum_{\alpha\in I} b_\alpha \partial_{v^\alpha}, \sum_{\beta\in I} c_\beta \partial_{w^\beta} \in \prod_{\alpha\in I} TL^\alpha$.
      This term vanishes as $TL^\alpha$ is a Lagrangian subspace of $\omega_X^\alpha$.
      \begin{align*}    
            \omega\left( \lj_*\left(\sum_{\alpha\in I} b_\alpha\partial_{v^\alpha}\right), \lj_*\left(\sum_{\beta\in I} c_\beta\partial_{w^\beta}\right)\right)=&\sum_{\alpha=\beta\in I} b_\alpha c_\beta \omega_{X^\alpha} \left( (\li_{\rho(s)})_* \partial_{v^\alpha},( \li_{\rho(s)})_* \partial_{w^\beta}\right)\\
      \end{align*}
The vanishing of the symplectic form on pairs of vectors $\lj_*(\partial_{y_i}), \lj_*\left(\sum_{\alpha\in I} c_\alpha \partial_{v^\alpha}\right)$ comes from the assumption that the Lagrangian homotopies $L^\alpha_t$ have flux primitive given by $H^\alpha_t$.
      \begin{align*}
            \omega\left( \lj_*(\partial_{y_i}), \lj_*\left(\sum_{\alpha\in I} c_\alpha \partial_{v^\alpha}\right)\right)=&\sum_{\alpha\in I} c_\alpha\left(\omega_{X^\alpha}\left( \frac{\partial \rho^\alpha}{\partial y_j} \frac{d}{dt}\li^\alpha_{\rho^\alpha(y)} ,(\li_{\rho(s)})_* \partial_{v^\alpha} \right)\right)\\&+\sum_{\alpha\in I}c_\alpha \omega_{T^*Y}\left(\delta_{ij},\frac{\partial \rho^\alpha}{\partial y_j} \partial_v^\alpha H^\alpha_{\rho^\alpha(y)}\right)\\
            =& \sum_{\alpha\in I} c_\alpha\left(\frac{\partial \rho^\alpha}{\partial y_i} \left( \iota_{ \frac{d}{dt}\li^\alpha_{\rho^\alpha(y)}}\omega_{X_\alpha}+ d H^\alpha_{\rho^\alpha(y)}\right) \partial_{v^\alpha}\right)
            = 0\\
      \end{align*}
      The vanishing of the symplectic form on pairs of vectors $\lj_*(\partial_{y_i}),\lj_*(\partial_{y_j})$ corresponds to the fact closed sections of the cotangent bundle are Lagrangian sections.
      \begin{align*}
            &\omega(\lj_*(\partial_{y_i}),\lj_*(\partial_{y_j}))=\sum_{\alpha\in I} \omega_{X^\alpha}\left(\frac{\partial \rho^\alpha}{\partial y_i} \frac{d}{dt}\li^\alpha_{\rho^\alpha(y)}, \frac{\partial \rho^\alpha}{\partial y_j} \frac{d}{dt}\li^\alpha_{\rho^\alpha(y)}\right) \\
            &+\omega_{Y}\left(\delta_{ik} +{\scriptstyle\jmath\sum_{\alpha\in I}\left( \frac{\partial^2\rho}{\partial y_i \partial y_k} H_{\rho(y)}+  \frac{\partial\rho^\alpha}{\partial y_i}\frac{\partial \rho^\alpha}{\partial y_k}\frac{d}{dt}H^\alpha_{\rho^\alpha(y)}\right)},\delta_{jk} +{\scriptstyle\jmath\sum_{\alpha\in I}\left( \frac{\partial^2\rho}{\partial y_j \partial y_k} H_{\rho(y)}+  \frac{\partial\rho^\alpha}{\partial y_j}\frac{\partial \rho^\alpha}{\partial y_k}\frac{d}{dt}H^\alpha_{\rho^\alpha(y)}\right)}\right)\\
            =&0
      \end{align*}
\end{proof}
\subsubsection{Decomposition across the cobordism parameter}
\label{subsection:truncation}
Cobordisms can be decomposed into smaller cobordisms along any regular level set of Morse function. 
We show an analogous decomposition for Lagrangian cobordisms.
\begin{prop}
Let $K: L^+\rightsquigarrow L^-$ be a Lagrangian cobordism with cylindrical boundary $M$.
Suppose that $\pi_\RR: K\to \RR$ has isolated critical values, and that $0\in \RR$ is a regular value of the projection $\pi_\RR: K\to \CC$, so that  $L^0:= K|_{0}\subset X$ is a Lagrangian submanifold.
Then there exists Lagrangian cobordisms with cylindrical boundary $M$
\begin{align*} 
      K\|_{(-\infty, 0]}: L^0\rightsquigarrow  L^-&&  K\|_{[0, \infty)}: L^+\rightsquigarrow L^0,\end{align*} so that $K$ and $ K\|_{(-\infty, 0]}\circ K\|_{[0, \infty)}$ are exactly homotopic Lagrangian cobordisms. Furthermore the construction can be done in such a way that 
      \begin{itemize}
            \item the Hofer norm of this exact homotopy is as small as desired and,
            \item if $K$ is embedded, and the slice $L^0$ is embedded, then $K\|_{(-\infty, 0]}\circ K\|_{[0, \infty)}$ is embedded as well. 
      \end{itemize}
      \label{prop:tdecomposition}
\end{prop}
A picture of this decomposition is given in \cref{fig:horizontaldecomposition}. 
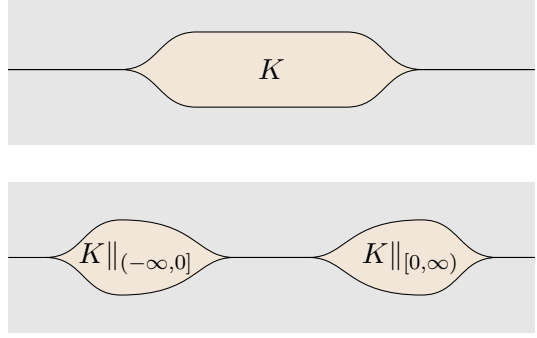
\begin{figure}
      \centering
      \begin{tikzpicture}
\begin{scope}[]

\fill[gray!20]  (-5,3) rectangle (2,1);
\draw[fill=brown!20] (-5,2) .. controls (-4.5,2) and (-4,2) .. (-3.5,2) .. controls (-3,2) and (-3,1.5) .. (-2.5,1.5) .. controls (-2,1.5) and (-2,1.5) .. (-1.5,1.5) .. controls (-1,1.5) and (-1,1.5) .. (-0.5,1.5) .. controls (0,1.5) and (0,2) .. (0.5,2) .. controls (1,2) and (1.5,2) .. (2,2) .. controls (1.5,2) and (1,2) .. (0.5,2) .. controls (0,2) and (0,2.5) .. (-0.5,2.5) .. controls (-1,2.5) and (-1,2.5) .. (-1.5,2.5) .. controls (-2,2.5) and (-2,2.5) .. (-2.5,2.5) .. controls (-3,2.5) and (-3,2) .. (-3.5,2) .. controls (-4,2) and (-4.5,2) .. (-5,2);

\end{scope}

\begin{scope}[shift={(0,3)}]
\fill[gray!20]  (-5,-2.5) rectangle (2,-4.5);

\end{scope}

\node at (-1.5,2) {$K$};
\draw[fill=brown!20] (-5,-0.5) .. controls (-4.5,-0.5) and (-5,-0.5) .. (-4.5,-0.5) .. controls (-4,-0.5) and (-4,-1) .. (-3.5,-1) .. controls (-2.5,-1) and (-2.5,-0.5) .. (-2,-0.5) .. controls (-1.5,-0.5) and (-1.5,-0.5) .. (-1,-0.5) .. controls (-0.5,-0.5) and (-0.5,-1) .. (0.5,-1) .. controls (1,-1) and (1,-0.5) .. (1.5,-0.5) .. controls (2,-0.5) and (1.5,-0.5) .. (2,-0.5) .. controls (1.5,-0.5) and (1.5,-0.5) .. (1.5,-0.5) .. controls (1,-0.5) and (1,0) .. (0.5,0) .. controls (-0.5,0) and (-0.5,-0.5) .. (-1,-0.5) .. controls (-1.5,-0.5) and (-1.5,-0.5) .. (-2,-0.5) .. controls (-2.5,-0.5) and (-2.5,0) .. (-3.5,0) .. controls (-4,0) and (-4,-0.5) .. (-4.5,-0.5);
\node at (-3.3,-0.5) {$K\|_{(-\infty, 0]}$};
\node at (0.35,-0.5) {$K\|_{[0, \infty)}$};
\end{tikzpicture}       \caption{Cobordisms can be decomposed across the cobordism parameter at non-critical values of $\pi_\RR\circ \lj : K\to \RR$.}
      \label{fig:horizontaldecomposition}
\end{figure}

\emph{Proof.} We construct a suspension Lagrangian cobordism $J\subset (X\times \CC)\times \CC$ with ends 
\[J:K\|_{(-\infty, 0]}\circ K\|_{[0, \infty)}\rightsquigarrow K\]
which will be the suspension of our exact homotopy.
Consider the  decomposition as sets $K=K|_{t\leq-\epsilon}\cup_{t=-\epsilon} K|_{[-\epsilon, \epsilon]}\cup_{t=\epsilon} K|_{[\epsilon, \infty) }$, where
\begin{align*}
      K|_{t\leq-\epsilon}=\pi^{-1}_\RR((-\infty, -\epsilon])&&K|_{[-\epsilon, \epsilon]}=\pi^{-1}_\RR([ -\epsilon, \epsilon])&&  K|_{[\epsilon, \infty) }=\pi^{-1}_\RR([\epsilon,\infty))
\end{align*}
The piece $K|_{[-\epsilon, \epsilon]}$ is a suspension so by \cref{prop:suspension} there exists a primitive $H_t:L^0\times[-\epsilon, \epsilon]\to \RR$ so that we can parameterize $K|_{[-\epsilon, \epsilon]}$ as:
\begin{align*}
      L^0\times[-\epsilon, \epsilon]\to& X\times \CC\\
      (q,t) \mapsto& (\li_t(q), t+H_t(q))
\end{align*}

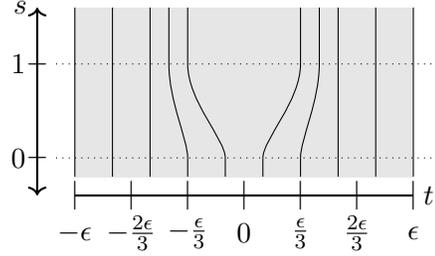
\begin{SCfigure}[50]
      \begin{tikzpicture}

\fill[gray!20]  (-2,-2.5) rectangle (2.5,-4.75);

\draw (-1.5,-2.5) .. controls (-1.5,-2.75) and (-1.5,-3) .. (-1.5,-3.25) .. controls (-1.5,-3.75) and (-1.5,-4.25) .. (-1.5,-4.5);
\draw (2.5,-2.5) .. controls (2.5,-2.75) and (2.5,-3) .. (2.5,-3.25) .. controls (2.5,-3.75) and (2.5,-4.25) .. (2.5,-4.5);
\draw (-1,-2.5) .. controls (-1,-2.75) and (-1,-3) .. (-1,-3.25) .. controls (-1,-3.75) and (-1,-4.25) .. (-1,-4.5);
\draw (-0.75,-2.5) .. controls (-0.75,-2.75) and (-0.75,-3) .. (-0.75,-3.25) .. controls (-0.75,-3.75) and (-0.5,-4.25) .. (-0.5,-4.5);
\draw (-2,-2.5) node (v1) {} .. controls (-2,-2.75) and (-2,-3) .. (-2,-3.25) .. controls (-2,-3.75) and (-2,-4.25) .. (-2,-4.5);
\draw (2,-2.5) .. controls (2,-2.75) and (2,-3) .. (2,-3.25) .. controls (2,-3.75) and (2,-4.25) .. (2,-4.5);

\draw (1,-2.5) .. controls (1,-2.75) and (1,-3) .. (1,-3.25) .. controls (1,-3.75) and (0.5,-4.25) .. (0.5,-4.5);
\draw (-0.5,-2.5) .. controls (-0.5,-2.75) and (-0.5,-3) .. (-0.5,-3.25) .. controls (-0.5,-3.75) and (0,-4.25) .. (0,-4.5);
\draw (1.25,-2.5) .. controls (1.25,-2.75) and (1.25,-3) .. (1.25,-3.25) .. controls (1.25,-3.75) and (1,-4.25) .. (1,-4.5);
\draw (1.5,-2.5) .. controls (1.5,-2.75) and (1.5,-3) .. (1.5,-3.25) .. controls (1.5,-3.75) and (1.5,-4.25) .. (1.5,-4.5);

\node at (-1.25,-5.5) {$-\frac{2\epsilon}{3}$};
\node at (-0.5,-5.5) {$-\frac{\epsilon}{3}$};
\node at (1,-5.5) {$\frac{\epsilon}{3}$};
\node at (1.75,-5.5) {$\frac{2\epsilon}{3}$};
\draw (-2,-4.75) -- (-2,-4.5) (-1.5,-4.75) -- (-1.5,-4.5) (-1,-4.75) -- (-1,-4.5) (-0.5,-4.75) -- (-0.5,-4.5) (0,-4.75) -- (0,-4.5) (0.5,-4.75) -- (0.5,-4.5) (1,-4.75) -- (1,-4.5) (1.5,-4.75) -- (1.5,-4.5) (2,-4.75) -- (2,-4.5) (2.5,-4.75) -- (2.5,-4.5);
\node at (-2,-5.5) {$-\epsilon$};
\node at (2.5,-5.5) {$\epsilon$};
\node at (-2,-5) {$|$};
\node at (-1.25,-5) {$|$};
\node at (-0.5,-5) {$|$};
\node at (0.25,-5) {$|$};
\node at (1,-5) {$|$};
\node at (1.75,-5) {$|$};
\node at (2.5,-5) {$|$};
\node at (0.25,-5.5) {0};
\draw[thick,<->] (-2.5,-5) -- (-2.5,-2.5);
\draw[thick] (-2,-5) -- (2.5,-5);
\node at (-2.75,-4.5) {$0$};
\node at (-2.75,-3.25) {$1$};
\draw[dotted] (-2.25,-3.25) -- (2.75,-3.25);
\draw[dotted] (-2.25,-4.5) -- (2.75,-4.5);
\node at (-2.5,-3.25) {$-$};
\node at (-2.5,-4.5) {$-$};
\node[right] at (2.5,-5) {$t$};
\node[left] at (-2.5,-2.5) {$s$};
\end{tikzpicture}       \caption{Contour plot of the truncation profile  $\rho(t,s)$.}
      \label{fig:truncationprofile}
\end{SCfigure}
We choose a truncation profile  $\rho(t, s): [-\epsilon, \epsilon]\times \RR \to [-\epsilon, \epsilon]$ which satisfies the following conditions  (as indicated in \cref{fig:truncationprofile}) :
\begin{align*}
    \rho(t, s)|_{s<0}=t&&\rho(t,s)|_{|t|>2\epsilon/3}= t\\
    \rho(t,1)|_{|t|<\epsilon/3}=0 && \left.\frac{\partial\rho}{\partial s}\right|_{s>1}=0
\end{align*}
Consider the Lagrangian submanifold $J|_{[-\epsilon, \epsilon]\times \RR}$ given by the generalized suspension from  \cref{claim:suspensiongeneralization}, 
\[\lj_\rho:  L^0\times[-\epsilon, \epsilon]\times \RR\to X\times T^* [-\epsilon, \epsilon]\times T^*\RR\]
This is a Lagrangian cobordism over the $s$ parameter with collared boundaries in both the $t$ and $s$ directions:
\begin{itemize}
      \item In the $s$ direction, $K|_{|t|=\epsilon}$ is a boundary for $K|_{[-\epsilon, \epsilon]}:=J|_{[-\epsilon, \epsilon]\times \{s=0\}}$.
      This extends to a collared boundary $K|_{|t|>2\epsilon/3}$, and $K|_{|t|>2\epsilon/3}\times \RR_{s} $ is a collared boundary for $J|_{[-\epsilon, \epsilon]\times \RR_{s}}$. Therefore, $J|_{[-\epsilon, \epsilon]\times \RR_{s}}$, as a cobordism in the $s$ direction, is a Lagrangian cobordism with cylindrical boundary.
      \item At each value of $s$,  the collar $(M\times [0, \epsilon'))\times [-\epsilon, \epsilon]_{t}\subset K|_{[-\epsilon, \epsilon]}$ is a collared boundary of the slice $J|_{[-\epsilon, \epsilon]\times \{s\}}$.
\end{itemize}
We're interested in the positive end of this Lagrangian cobordism, which we call $K^+|_{[-\epsilon, \epsilon]}$, so that 
\[J|_{[-\epsilon, \epsilon]\times \RR}: K^+|_{[-\epsilon, \epsilon]}\rightsquigarrow K|_{[-\epsilon, \epsilon]}\]
is a suspension Lagrangian cobordism with collared boundary.
 The Lagrangian $K^+|_{[-\epsilon, \epsilon]}$ is a suspension, and parameterized by 
\[(q, t) \mapsto \left(\li_{\rho(t, 1)}(q), t+\jmath  \frac{d\rho(t, 1)}{dt}H_{\rho(t, 1)} \right)\] 
We can therefore form a Lagrangian cobordism $J:K^+ \rightsquigarrow K$, where 
\[K^+= K|_{(-\infty, \epsilon]}\cup_{t=-\epsilon} K^+|_{[-\epsilon, \epsilon]}\cup_{t=\epsilon} K|_{[\epsilon, \infty) }.\]
By construction, $\pi_{\jmath\RR}(K^+|_{[-\epsilon/3, \epsilon/3]})=0$, so this portion looks like $L^0\times[-\epsilon/3, \epsilon/3]$. 
We therefore may assemble Lagrangian cobordisms:
 \begin{align*}
      K\|_{(-\infty, 0]}:=& K^+|_{(\infty, 0]}\cup_{L^0} (L^0\times[0, \infty))\; : L^0\rightsquigarrow L^-\\
      K\|_{[0, \infty)}:=& K^+|_{[0, \infty)} \cup_{L^0} (L^0\times(-\infty,0]) : L^+\rightsquigarrow L^0.
 \end{align*}
 Clearly, $K^+=  K\|_{(-\infty, 0]}\circ  K\|_{[0, \infty)}$. Additionally, since  $K^+$ fixes the boundary $M$, both $ K\|_{(-\infty, 0]}$ and $  K\|_{[0, \infty)}$ are Lagrangian submanifolds which fix the boundary $M$.

 It remains to show that this construction can be completed so that the Hofer norm is as small as desired. The norm is 
 \begin{align*}
      \int_\RR& \left(\sup_{(q, t)\in L\times [-\epsilon, \epsilon]}\left(\frac{\partial\rho}{\partial s}H_{\rho(s, t)}(q)\right)-\inf_{(q, t)\in L\times [-\epsilon, \epsilon]}\left(\frac{\partial\rho}{\partial s}H_{\rho(s, t)}(q)\right)\right)\\
      \leq& \left(\sup_{(q, t)\in L\times [-\epsilon, \epsilon]}H_{t}(q)-\inf_{(q, t)\in L\times [-\epsilon, \epsilon]}H_{t}(q)\right)\int \left(\sup_{t\in[-\epsilon, \epsilon]} \frac{\partial\rho}{\partial s}-\inf_{t\in[-\epsilon, \epsilon]}\frac{\partial\rho}{\partial s}\right)ds\\
      \leq& 2\epsilon \left(\sup_{(q, t)\in L\times [-\epsilon, \epsilon]}H_{t}(q)-\inf_{(q, t)\in L\times [-\epsilon, \epsilon]}H_{t}(q)\right)
 \end{align*}
 which can be made as small as desired.
  \hfill\qed\\

Since this construction occurs away from the critical locus, we additionally have a matching of Morse critical points
\begin{align*}
      \Crit(\pi_\RR:K\to \RR)=& \Crit(\pi_\RR:K\|_{(-\infty, 0]}\circ  K\|_{[0, \infty)}\to \RR)\\
      =& \Crit(\pi_\RR: K\|_{[0, \infty)}\to \RR)\cup \Crit(\pi_\RR:K\|_{(-\infty, 0]}\to \RR)
\end{align*}
Recall that the shadow of a Lagrangian cobordism $\Area(K)$ is the infimum of areas of simply connected regions containing the image of $\pi_\CC:K\to \CC$, which generalized the Hofer norm.
\begin{claim}
      Let $K$ and $ K^+:=K\|_{(-\infty, 0]}\circ K\|_{[0, \infty)}$ be the Lagrangians from \cref{prop:tdecomposition}. Then $\Area(K)=\Area(K^+)$.
\end{claim} 
\begin{proof}
If $K$ is a suspension arising from exact homotopy $H_t$, then the shadow can be computed via the Hofer norm:
      \[\Area(K)=\int_{\RR} \left(\sup_{q\in L^+}H_t(q)- \inf_{q\in L^+} H_t(q)\right)dt\]
We note that $K|_{[-\epsilon, \epsilon]}$ and $K^+|_{[-\epsilon, \epsilon]}$ are suspensions. Integrating by change of variables yields:
\begin{align*}
      \Area(K^+|_{[-\epsilon, \epsilon]})=&\int_{[-\epsilon, \epsilon]} \left(\sup_{q\in L^+} \left(\frac{d\rho(t, 1)}{dt}H_{\rho(t, 1)}(q)\right)- \inf_{q\in L^+} \left( \frac{d\rho(t, 1)}{dt}H_{\rho(t, 1)}(q)\right)\right)dt\\
      =&\int_{[-\epsilon, \epsilon]} \left(\sup_{q\in L^+}H_t(q)- \inf_{q\in L^+} H_t(q)\right)dt=\Area(K|_{[-\epsilon, \epsilon]}).
\end{align*}
\end{proof}
The same method allows us to construct Lagrangian cobordisms from Lagrangian submanifolds $ K\subset X\times \CC$.
\begin{definition}
      \label{def:lagrangiancutout}
      Let $ K\subset X\times \CC$ be a Lagrangian submanifold (not necessarily a Lagrangian cobordism) with cylindrical boundary. Suppose that $t_-, t_+$ are regular values of the projection $\pi_\RR: K\to \RR$, and that the critical values of $\pi_\RR$ are isolated. 
      Then by applying \cref{prop:tdecomposition} at $t_-, t_+$, we can define the \emph{truncation} of $K$ to $[t_-, t_+]$ which is a Lagrangian cobordism with cylindrical boundary
      \[ K\|_{[t_-, t_+]}: K|_{t_+}\rightsquigarrow K|_{t_-}.
      \]
\end{definition}
\subsubsection[Decomposition across the X coordinate]{Decomposition across the $X$-coordinate}
\label{subsubsec:Xdecomposition}
We now look at how to ``isolate'' a portion of a Lagrangian cobordism across the $X$-coordinate so that it can be decomposed in the sense of \cref{def:Xdecomposition}.
\begin{definition}
      Let $K: L^+\rightsquigarrow L^-$ be an embedded Lagrangian cobordism with $\pi_\RR: K\to \RR$ having isolated critical points.
      A dividing hypersurface for $K$ is an embedded hypersurface $M\subset L^0$ with the following properties:
      \begin{itemize}
            \item $M$ divides $L^0$ in the sense that $L^0= L^0_\downarrow\cup_M L^0_\uparrow$, where $L_{\uparrow/\downarrow}^0$ are the components of $L^0\setminus M$. 
            \item $M\subset L^0\subset K$ contains no critical points of $\pi_\RR: K\to \RR$.
      \end{itemize} 
\end{definition}
A dividing hypersurface allows for the following decomposition of our Lagrangian cobordism.

\begin{prop}
      Let $M\subset L^0$ be a dividing hypersurface for $K: L^+\rightsquigarrow L^-$. 
      Then there exists a decomposition of Lagrangian cobordisms up to exact homotopy:
      \[ K\sim \tilde K^-\circ \uu{K}_M \circ \tilde K^+\]
      so that $ \uu{K}_M$ is a Lagrangian cobordism which decomposes along $M$, and $\uu{K}_M|_0=L^0$.
      Furthermore, this construction can be performed in such a way that
      \begin{itemize}
            \item  the exact homotopy has as small Hofer norm as desired; and
            \item if $K$ is embedded and the slice $L^0$ is embedded, then $\tilde K^-\circ \uu{K}_M \circ \tilde K^+$ is embedded as well.
      \end{itemize}
      \label{prop:splittingX}
\end{prop}
\begin{proof}
      Let $\li: K\to X\times \CC$ be the parameterization of our Lagrangian cobordism.
      Consider a small collared neighborhood $M\times I_s\subset L^0$. 
      We take $B^*_\epsilon (M\times I_s)$ a small neighborhood of the zero section inside the cotangent bundle of $(M\times I_s)$. There exists a map  $\phi:B^*_\epsilon (M\times I_s)\to X$, which is locally a symplectic embedding, and sends the zero section to $\li(M\times I_s)$.
      For $H: M\times I_s\to X$ with $|dH|<\epsilon$ and the support of $H$ contained on an interior subset of $M\times I$, denote by $dH\subset X$ the (possibly immersed) submanifold parameterized by 
      \[ M\times I\xrightarrow{df} B^*_\epsilon (M\times I_s)\xrightarrow{\phi} X.\]
      There exists an open neighborhood $U\subset K$ containing $M\subset L^0\subset K$ with the property that $(U|_t)\cap B^*_\epsilon(M\times  I_s)$ is a section of the cotangent ball for each $t$. 
      Let $H_t(q, s): M\times I_s\to \RR$ be the primitive of this section for each $t$ so we can parameterize $\li:U\to X\times \CC$ by 
      \begin{align*}
            M\times I_s\times I_t\to& X\times \CC\\
            (q, s, t)\mapsto& \left(dH_t,t+\jmath \frac{d}{dt}H_t\right)
      \end{align*}
      We now consider a function $\rho(s, t): I_s\times I_t\to \RR$ which is constantly $1$ in a neighborhood of $\partial(I_s\times I_t)$, and constantly $0$ on an interior set $(-\epsilon', \epsilon')_s\times (-\epsilon'', \epsilon'')_t\subset I_s\times I_t$. 
      Consider the Lagrangian suspension cobordism 
      \begin{align*}
            \tilde \li: M\times I_s\times I_t\to & X\times \CC\\
            (q, s, t)\mapsto & \left(d(\rho(s, t)H_t), t+\jmath \frac{d}{dt}\rho(s, t)H_t\right)
      \end{align*}
      and by abuse of notation, let $\tilde \li : K\to X\times \CC$ be the Lagrangian cobordism where we have replaced $K|_U$ with the chart parameterized above. The parameterizations $\tilde \li$ and $\li$ are exactly isotopic.
      For $t_0\in (-\epsilon'', \epsilon'')$ and $(q, s)\in M\times (-\epsilon', \epsilon')$, we have that $\frac{d}{dt}\tilde \li(q, s, t)=0$. 
      Therefore, $\tilde K\|_{(-\epsilon'', \epsilon'')}$ admits a decomposition in the $X$ factor along $M$.
      We define 
      \begin{align*}
            \uu{K}_M:= \tilde K\|_{(-\epsilon'', \epsilon'')}&& \tilde K^-:= \tilde K\|_{<-\epsilon''} && \tilde K^+:=\tilde K\|_{<\epsilon''}.
      \end{align*}
      Bounding the Hofer norm is similar to the computation for \cref{prop:tdecomposition}, and is bounded by the $2(\epsilon'+\epsilon'')\left(\sup_{(q, t, s)}H_t(q, s)-\inf_{(q, t, s)}H_t(q, s)\right)$, which can be made as small as desired.
\end{proof}
We write this decomposition as 
\begin{equation}
       \uu{K}_M =  \uu{K}_{\downarrow }\cup_{M\times \RR}  \uu{K}_{\uparrow}.
       \label{eqn:decompositionNotation}
\end{equation}
By applying \cref{prop:exchangeRelation} to our decomposition along $M$ we  further split the Lagrangian cobordism as a composition. There is a sequence of exact homotopies (\cref{fig:furtherdecomp})
\[K\rightsquigarrow \tilde K^-\circ \uu{K}_M \circ \tilde K^+\rightsquigarrow \tilde K^-\circ \left( \uu{K}_{\downarrow}\cup_{M\times \RR}\times L^-_{\uparrow}\times \RR\right)\circ\left( \uu{K}_{\uparrow}\cup_{M\times \RR} (L^+_{\downarrow}\times \RR) \right)\circ \tilde K^+.\]
This second exact homotopy has Hofer norm bounded by $2\epsilon''\Area(K)$.
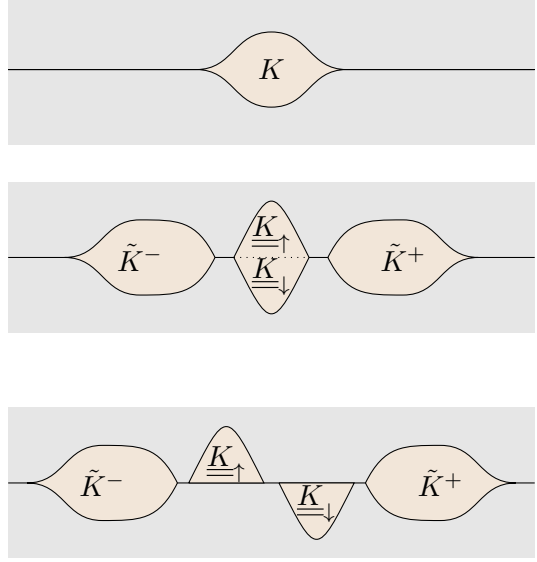
\begin{figure}
      \centering
      \begin{tikzpicture}
\begin{scope}[]

\fill[gray!20]  (-5,3) rectangle (2,1);
\draw[fill=brown!20] (-5,2) .. controls (-4.5,2) and (-3,2) .. (-2.5,2) .. controls (-2,2) and (-2,2.5) .. (-1.5,2.5) .. controls (-1,2.5) and (-1,2) .. (-0.5,2) .. controls (0.5,2) and (1.5,2) .. (2,2) .. controls (1.5,2) and (0.5,2) .. (-0.5,2) .. controls (-1,2) and (-1,1.5) .. (-1.5,1.5) .. controls (-2,1.5) and (-2,2) .. (-2.5,2) .. controls (-3,2) and (-4.5,2) .. (-5,2);

\end{scope}

\begin{scope}[]
\fill[gray!20]  (-5,-2.5) rectangle (2,-4.5);
\draw (-5,-3.5) -- (2,-3.5);
\begin{scope}[shift={(-0.25,0)}]
\draw[fill=brown!20] (-4.5,-3.5) .. controls (-4,-3.5) and (-4,-3) .. (-3.5,-3) .. controls (-3,-3) and (-2.75,-3) .. (-2.5,-3.5) .. controls (-2.75,-4) and (-3,-4) .. (-3.5,-4) .. controls (-4,-4) and (-4,-3.5) .. (-4.5,-3.5);
\node at (-3.5,-3.5) {$\tilde K^-$};

\end{scope}

\begin{scope}[shift={(-0.1,0)}]
\draw[fill=brown!20](-2.5,-3.5) .. controls (-2,-2.5) and (-2,-2.5) .. (-1.5,-3.5) .. controls (-1.5,-3.5) and (-1.5,-3.5) .. (-2.5,-3.5);

\node at (-2,-3.25) {$\uu{K}_\uparrow$};

\end{scope}
\begin{scope}[shift={(0.1,0)}]

\draw[fill=brown!20 ](-1.5,-3.5) .. controls (-1,-4.5) and (-1,-4.5) .. (-0.5,-3.5) .. controls (-0.5,-3.5) and (-0.5,-3.5) .. (-1.5,-3.5);
\node at (-1,-3.775) {$\uu{K}_\downarrow$};

\end{scope}

\begin{scope}[shift={(0.25,0)}]
\draw[fill=brown!20](-0.5,-3.5) .. controls (-0.25,-3) and (0,-3) .. (0.5,-3) .. controls (1,-3) and (1,-3.5) .. (1.5,-3.5) .. controls (1,-3.5) and (1,-4) .. (0.5,-4) .. controls (0,-4) and (-0.25,-4) .. (-0.5,-3.5);

\node at (0.5,-3.5) {$\tilde K^+$};

\end{scope}

\end{scope}

\begin{scope}[shift={(0,3)}]
\fill[gray!20]  (-5,-2.5) rectangle (2,-4.5);
\draw (-5,-3.5) -- (2,-3.5);
\begin{scope}[shift={(0.25,0)}]
\draw[fill=brown!20] (-4.5,-3.5) .. controls (-4,-3.5) and (-4,-3) .. (-3.5,-3) .. controls (-3,-3) and (-2.75,-3) .. (-2.5,-3.5) .. controls (-2.75,-4) and (-3,-4) .. (-3.5,-4) .. controls (-4,-4) and (-4,-3.5) .. (-4.5,-3.5);
\node at (-3.5,-3.5) {$\tilde K^-$};

\end{scope}

\draw[fill=brown!20] (-2,-3.5) .. controls (-1.5,-4.5) and (-1.5,-4.5) .. (-1,-3.5) .. controls (-1.5,-2.5) and (-1.5,-2.5) .. (-2,-3.5);

\node at (-1.5,-3.2) {$\uu{K}_\uparrow$};

\node at (-1.5,-3.75) {$\uu{K}_\downarrow$};

\begin{scope}[shift={(-0.25,0)}]
\draw[fill=brown!20](-0.5,-3.5) .. controls (-0.25,-3) and (0,-3) .. (0.5,-3) .. controls (1,-3) and (1,-3.5) .. (1.5,-3.5) .. controls (1,-3.5) and (1,-4) .. (0.5,-4) .. controls (0,-4) and (-0.25,-4) .. (-0.5,-3.5);

\node at (0.5,-3.5) {$\tilde K^+$};

\end{scope}

\draw[dotted] (-2,-3.5) -- (-1,-3.5);
\end{scope}

\node at (-1.5,2) {$K$};
\end{tikzpicture}       \caption{Given a dividing hypersurface $M$ in a Lagrangian cobordism $K$, one can exactly homotope $K$ to a composition where the middle Lagrangian cobordism decomposes across the $X$ coordinate. This middle component can be further decomposed by applying \cref{prop:exchangeRelation}.}
      \label{fig:furtherdecomp}
\end{figure}
We note that no part of this construction modifies the height function, so 
\[\Crit(\pi_\RR:K\to \RR)= \Crit(\pi_\RR:\tilde K^-\circ \uu{K}_M \circ \tilde K^+\to \RR).\]

As in the setting of decomposition along the $\CC$ coordinate, we can show that in good cases this decomposition does not modify the Lagrangian shadow. 
Suppose that $M\subset L^0$ is a dividing hypersurface for $K: L^+\rightsquigarrow L^-$.
      Furthermore, suppose that over the chart $M\times I_s\times I_t\subset K$ considered in the proof of \cref{prop:splittingX}, we have 
      \begin{align*}
            \sup_{q\in K\;|\; \pi_\RR(q)=t} \pi_{\jmath \RR}\li(q) > \sup_{q\in M\times I_s\times \{t\} }\pi_{\jmath \RR}\li(q)\\
            \inf_{q\in K\;|\; \pi_\RR(q)=t} \pi_{\jmath \RR}\li(q) < \inf _{q\in M\times I_s\times \{t\}} \pi_{\jmath \RR}\li(q).
      \end{align*}
      Then 
\[\Area(K)= \Area (\tilde K^-\circ \uu{K}_M \circ \tilde K^+).\]
 
\subsection{Standard Lagrangian surgery handle}
\label{subsec:highersurgery}
In this section we give a description of a standard Lagrangian surgery handle. 
We include many figures in the hope of making the geometry of Lagrangian surgery apparent and start with the simple example of Lagrangian null-cobordism for Whitney spheres. 
\subsubsection{Null-Cobordism and the Whitney Sphere}
\label{subsubsec:whitneysphere}
We first give a definition of the Whitney sphere in higher dimensions, and show that this is null-cobordant. 
\begin{definition}
    The Whitney sphere of area $A$ is the Lagrangian submanifold $L^{n,0,+}_A\subset \CC^n$ which is parameterized by 
    \begin{align*}
        \li^{n,0,+}_A: S^{n}_r\to& \CC^{n}\\
        (x_0, \ldots, x_n)\mapsto& (x_1+\jmath 2x_0x_1, x_2+\jmath 2x_0x_2, \ldots, x_n+\jmath 2x_0x_n).
    \end{align*}
    where $S^n_r=\{(x_0, \ldots x_n)\;|\; \sum_{i=0}^n x_i^2=r^2\}$, and $r=\sqrt[3]{\frac{4}{3}A}$.
\end{definition}

This Lagrangian has a single transverse self-intersection at the pair of points $(\pm r, 0, \ldots, 0)\in S^n_r$. We call these points $q^\pm \in L^{n, 0,+}_A$.
The quantity $A$ describes the area of the projection of $L^{n, 0,+}_A$ to the first complex coordinate, 
\[\frac{\Area(\pi_{\CC}(L^{n, 0,+}_A))}{2} = 2\cdot \left(\int_{0}^r 2 t \sqrt{r^2-t^2} dt\right) = A.\]
The Lagrangian submanifold $L^{1, 0,+}_A$ is the figure eight curve. An example of the Whitney 1-sphere is drawn in the \cref{fig:handle11} as the slice $K^{1, 1}|_1=L^{1,0}_{4/3}$. 
In \cref{fig:whitneysphere} we give a plot of the Whitney 2-sphere, $L^{1,0,+}_{4/3}\subset \CC^2$, presented as a set of covectors in the cotangent bundle $T^*\RR^2$.
\begin{figure}
    \centering
    \includegraphics{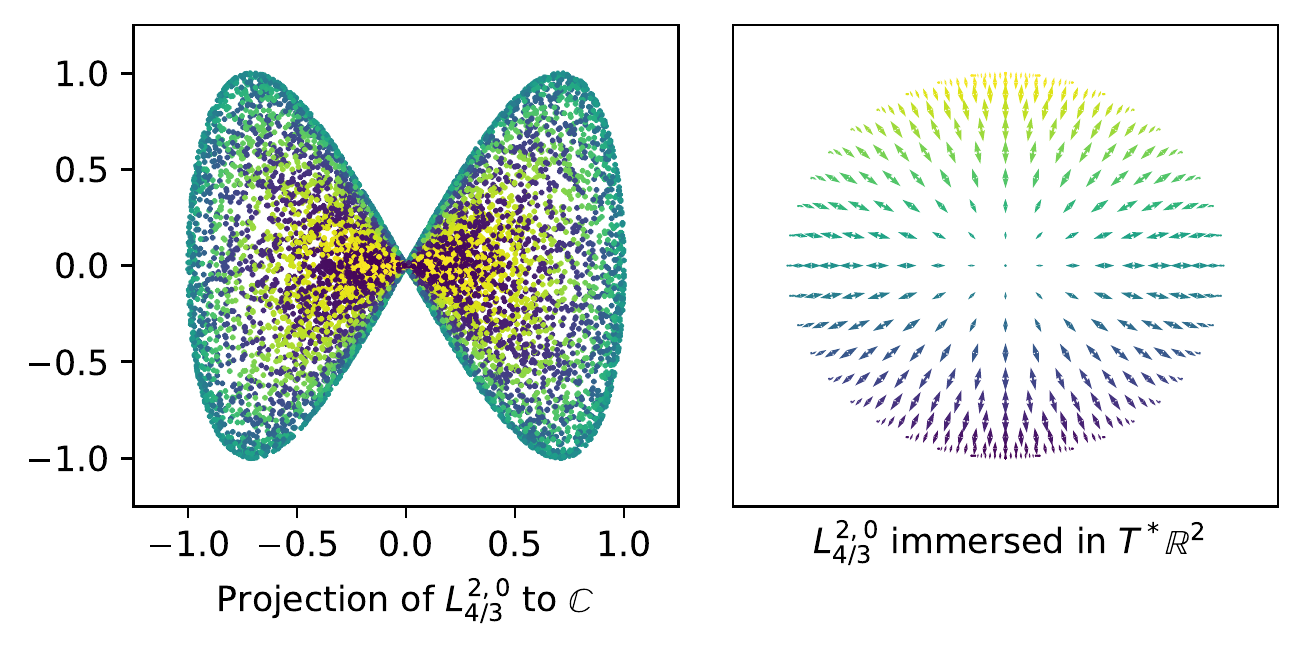}
    \caption{The Whitney Sphere  $L^{2,0,+}_{4/3}\subset \CC^2$. On the left: a projection of randomly sampled points on $L^{2,0,+}_{4/3}$ to the first complex coordinate. The area of the projection is $2\cdot 4/3$. On the right: $L^{2, 0,+}_{4/3}$ drawn as a subset of $T^*\RR^2$.}
    \label{fig:whitneysphere}
\end{figure}

The Whitney sphere can be extended in one dimension higher to a Lagrangian submanifold parameterized by the disk. 
Let $r(x_0, \ldots, x_n)=\sum_{i=0}^n x_i^2$. 
The parameterization
\begin{align*} 
    \lj^{n,1}: \RR^{n+1}\to& \CC^{n}\times \CC\\
    (x_0, \ldots, x_{n})\mapsto& (\li^{n,0}_r(x_0, \ldots, x_{n}), r^2-\jmath x_0).
\end{align*}
gives an embedded Lagrangian disk $K^{n,1}\subset \CC^{n}\times \CC$ which has the following properties:
\begin{itemize}
    \item When $r<0$, the slice $K^{n,1}|_r$ is empty;
    \item The slice $K^{n, 1}|_0$ is not regular;
    \item When $r>0$ the slice $K^{n,1}|_r$  is a Whitney sphere of area $A=\frac{4r^3}{3}$.
\end{itemize}
We will prove that this is a Lagrangian submanifold in \cref{subsubsec:handles}. In \cref{fig:handle11} we draw this Lagrangian null-cobordism and its slices, which are Whitney spheres of decreasing radius.
\begin{figure}
    \centering
    \includegraphics{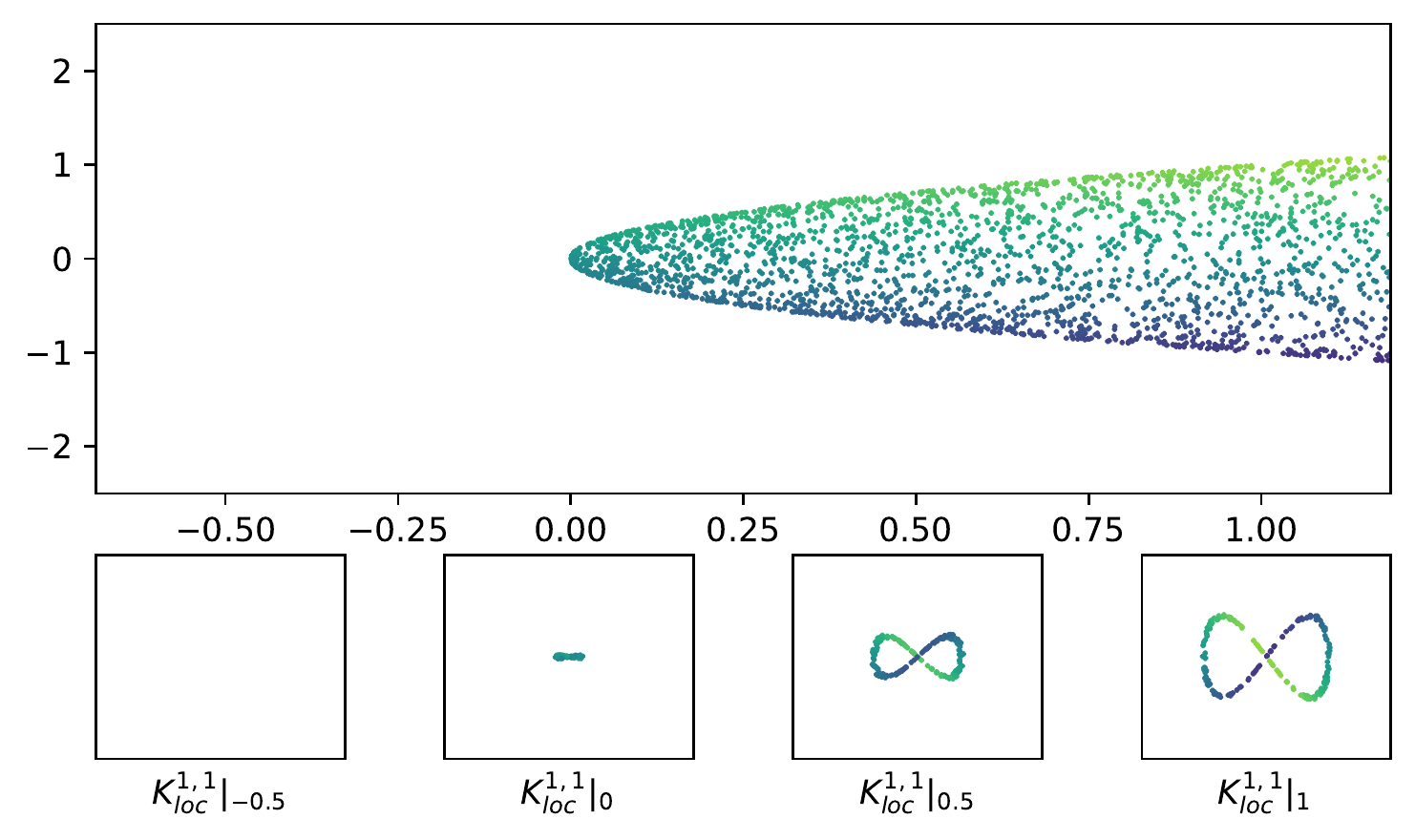}
    \caption{A plot of points on the null-cobordism $K^{1,1}\subset \CC^2$ of the Whitney sphere $L^{1,0,+}_{4/3}\subset \CC^1$. Points are consistently colored between figures. Top: The shadow projection of this null-cobordism. Bottom: Slices of the null-cobordism at different values of the cobordism parameter. }
    \label{fig:handle11}
\end{figure}
While $K^{n,1}$ is not a Lagrangian cobordism (as it does not fiber over the real line outside of a compact set,) it should be thought of as a model for the null-cobordism of the Whitney sphere.
If we desire a Lagrangian cobordism, we may apply \cref{def:lagrangiancutout} to truncate this Lagrangian submanifold and obtain a Lagrangian cobordism.
\subsubsection{Standard Surgery and Anti-surgery Handle}
\label{subsubsec:handles}
The standard Lagrangian surgery handle \cite{audin1994symplectic} is a Lagrangian $\RR^n\times \RR^1$ inside $T^*\RR^n\times T^*\RR^1$, where $T^*\RR^1$ is identified with $\CC$ by 
\[(q, p)\mapsto p+\jmath q\]

Let $\sigma_{i,k}$ be $+1$ if $i\leq k$ and $-1$ otherwise.
Consider the function 
\begin{align*}
    G^{k, n-k+1}: \RR^{n+1}\to & \RR\\
    (x_0,x_1, \ldots, x_n)\mapsto & x_0\left(\sum_{i=1}^n\sigma_{i,k}x_i^2\right)+ \frac{1}{3}x_0^3
\end{align*}
The graph of $dG^{k, n-k+1}$ parameterizes a Lagrangian submanifold inside the cotangent bundle,
\begin{align*}
    \RR^{n+1}\to& T^*\RR^n\times \CC\\
    (x_0,x_1, \ldots, x_n) \mapsto& \left(x_1+\jmath \sigma_{1,k}2x_1x_0,\ldots, x_n+\jmath \sigma_{n,k}2x_nx_0 ,x_0+\jmath\left(x_0^2+\sum_{i=1}^n\sigma_{i,k}x_i^2\right)\right),
\end{align*}
whose projection to the $\pi_{\jmath \RR}$ coordinate is a Morse function with a single critical point of index $k+1$. Our convention is that the Morse index of a critical point is dimension of the upward flow space of the point.
By multiplying the last coordinate by $-\jmath$, we interchange the real and imaginary parts of the shadow projection.
\begin{definition}
    For $k\geq 0$, the local Lagrangian $(k, n-k+1)$ surgery trace is the Lagrangian submanifold $K^{k, n-k+1}_{loc}\subset (\CC)^n\times \CC$ parameterized by 
\begin{align*}
    \lj^{k, n-k+1}: \RR^{n+1}\to& T^*\RR^n\times \CC\\
    (x_0,x_1, \ldots, x_n) \mapsto& \left(x_1+\jmath \sigma_{1,k}2x_1x_0,\ldots, x_n+\jmath \sigma_{n,k}2x_nx_0 ,x_0^2+\sum_{i=1}^n\sigma_{i,k}x_i^2-\jmath x_0\right).
\end{align*}
    The positive and negative slices of this Lagrangian submanifold will be denoted  
    \begin{align*}
        L^{k, n-k,+}_{loc}:=K^{k, n-k+1}_{loc}|_1&&  L^{k+1, n-k-1, -}_{loc}:= K^{k, n-k+1}|_{-1}.
    \end{align*}
    For $k=-1$, we define $K^{-1, n+2}_{loc}:= (K^{n, 1}_{loc})^{-1}$.
    \label{def:localsurgerytrace}
\end{definition}
This will be the local model for Lagrangian surgery trace, which we will construct in \cref{subsubsec:surgerytrace}. Before we proceed with the construction, we state some properties of the surgery trace, and give some examples.
\begin{theorem}[Properties of the standard Lagrangian Surgery Trace]
    The Lagrangian surgery trace $K^{k, n-k+1}_{A}:L^{k, n-k,+}_{A}\rightsquigarrow L^{k+1, n-k-1,-}_{A}$ has the following properties:
    \begin{itemize}
        \item $L^{k, n-k,+}_{A}$ is a Lagrangian $S^k\times D^{n-k}$ with a single self-intersection. Its intersection with the first $k$-coordinates is a Whitney isotropic $L^{k, 0, +}_A$;
        \item $L^{k+1, n-k-1,-}_{A}$  an embedded Lagrangian $D^{k+1}\times S^{n-k-1}$; and
        \item $\pi_\RR: K^{k, n-k+1}_{A}\to \RR$ is Morse, with a single critical point of index $k+1$.
    \end{itemize}
    \label{thm:standardTraceProperties}
\end{theorem}

A particularly relevant example is $K^{0,n+1}_{loc}\subset \CC^{n+1}$, which gives a local model for the Polterovich surgery trace (see \cref{fig:handle02} for the example $K^{0,2}_{loc}$).
\begin{figure}
    \centering
    \includegraphics{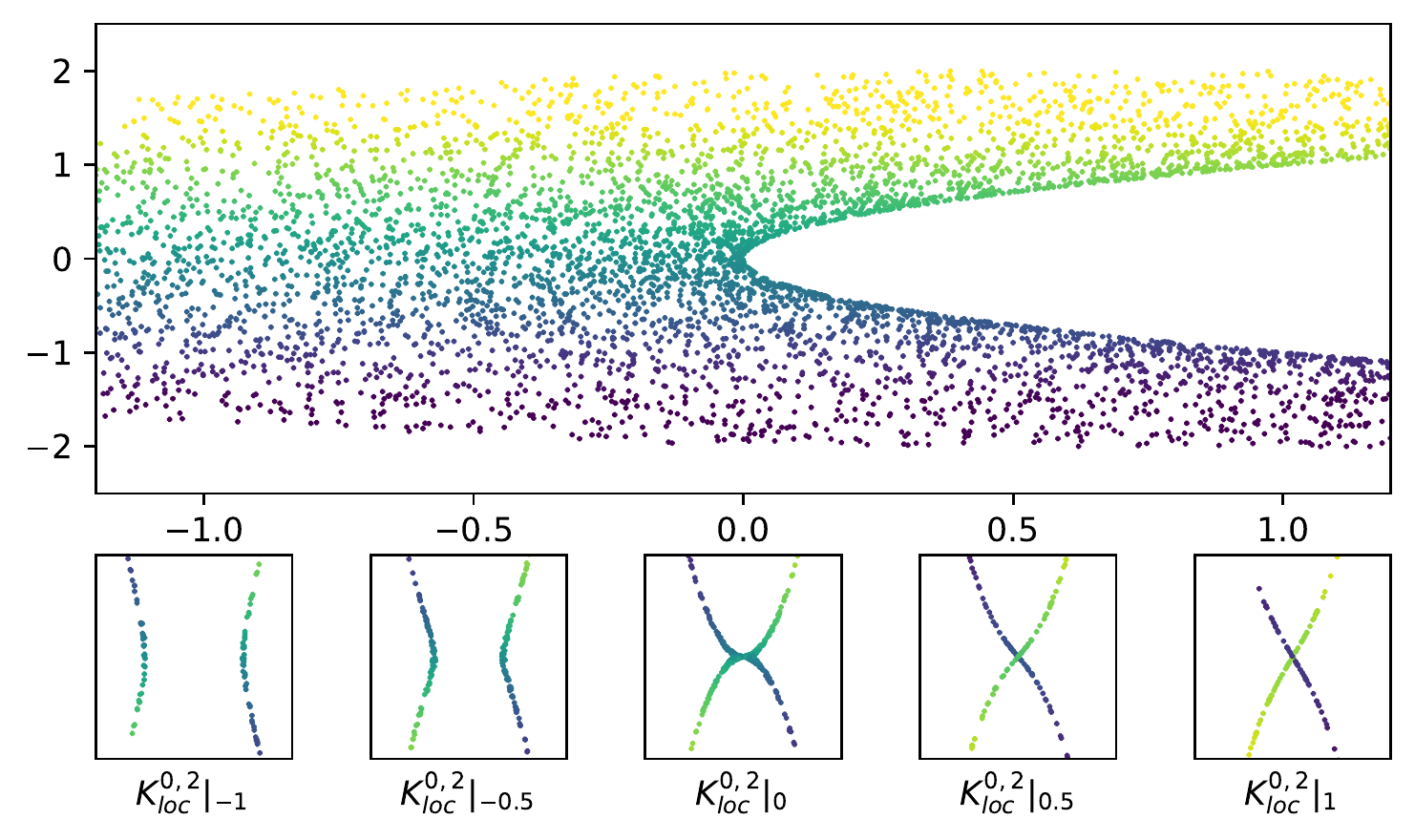}
    \caption{A Lagrangian Surgery cobordism $K^{0,2}_{loc}\subset \CC^2$. Top: The shadow projection of the surgery 1-handle $K^{0,2}_{loc}$. Bottom: Slices of the surgery. The left-most and right-most pictures are models of the surgery and anti-surgery neck.}
    \label{fig:handle02}
\end{figure}

The slice $L^{k, n-k,+}_{loc}$ is an immersed Lagrangian submanifold with a single double point,
\[\pi_X\circ \lj^{k, n-k+1}(\pm 1, 0, \ldots, 0)=(0, \ldots, 0, 0).\]
We denote these points $q_\pm \in L^{k, n-k, +}_{loc}$. 
\footnote{The following mnemonics may be useful to the reader: the positive end of the surgery cobordism is immersed and locally looks like the character ``$+$''.}
A useful observation is that when the positive end of the surgery trace is restricted to the first $k$-coordinates, 
\[L^{k, n-k,+}_{loc}|_{\CC^k}= L^{k, 0,+}_{loc}\subset \CC^k\]
we see an isotropic Whitney sphere. 
The other end of the Lagrangian surgery trace, $ L^{k+1, n-k-1, -}_{loc}$ is embedded. Furthermore,
\[L^{k+1, n-k-1,-}_{loc}|_{\CC^k}= \emptyset.\]
This allows us to interpret $K^{k, n-k+1}_{loc}$ as a null-cobordism of a Whitney isotropic in the first $k$-coordinates. 
This is a slightly deceptive characterization, as not all Whitney isotropic spheres are null-cobordant. See \cref{rem:surgeryconfiguration}.

According to our convention (which is that Lagrangian cobordisms go from the positive end to the negative end), the Lagrangian cobordism $K^{k, n-k+1}_{loc}$ resolves a self-intersection of the input end.
For this reason, we say that  $K^{k, n-k+1}_{loc}$ provides a local model of Lagrangian surgery.
Given $\lj: K\to X\times \CC$ a parameterization for a Lagrangian cobordism $K$, the inverse Lagrangian cobordism (denoted by $K^{-1}\subset X\times \CC$) is the Lagrangian submanifold parameterized by $(\pi_X\circ \lj, -\pi_\RR\circ \lj+\jmath \pi_{\jmath \RR}\circ \lj)$ (i.e. by reflecting the real parameter of the cobordism).
We call the inverse Lagrangian submanifold, $(K^{k, n-k+1}_{loc})^{-1}$ the local model for Lagrangian anti-surgery.

\begin{example}[Lagrangian Surgery Handle $K^{0,3}_{loc}$]
    In \cref{fig:handle03} we draw slices of the Lagrangian cobordism $K^{0,3}_{loc}$. 
    In the surgery interpretation, the Lagrangian self-intersection point is an isotropic Whitney sphere $L^{0,0}\subset \CC^0$, highlighted in blue.

    In the anti-surgery interpretation, the isotropic Lagrangian disk highlighted in red is contracted, collapsing the $S^1$ boundary to a transverse self-intersection.
\end{example}

\begin{figure}
    \centering
    \includegraphics{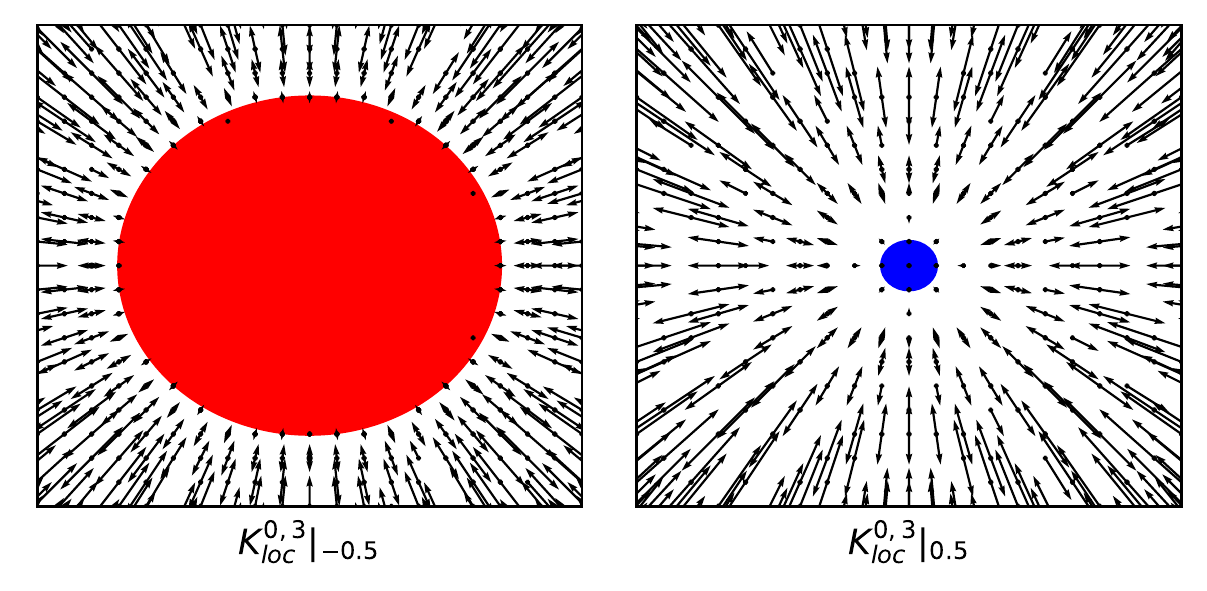}
    \caption{Slices of the surgery 2-handle $K^{0,3}_{loc}\subset \CC^3$ from before and after the critical point. These correspond to models of the surgery and anti-surgery neck. The positive end (right subfigure) is an immersed $D^2\times S^0$, with double point is indicated in blue. The negative end(left subfigure) is an embedded $S^1\times D^1$, whose Lagrangian antisurgery $D^2$ from \cite{haug2015lagrangian} is drawn in red.}
    \label{fig:handle03}
\end{figure}
\begin{example}[Lagrangian Surgery Handle $K^{1,2}_{loc}$]
    In \cref{fig:handle12} we draw slices of the Lagrangian cobordism $K^{1,2}_{loc}$. 
    In the surgery interpretation, we resolve the isotropic Whitney $S^1$ sphere highlighted in blue by replacing it with two copies (an $S^0$ family) of the null-cobordism $D^{2,0}$. 

    In the anti-surgery interpretation, the isotropic Lagrangian disk highlighted in red is contracted, collapsing the immersed $S^0$ boundary and yielding a Lagrangian with a self-intersection.
\end{example}
\begin{figure}
    \centering
    \includegraphics{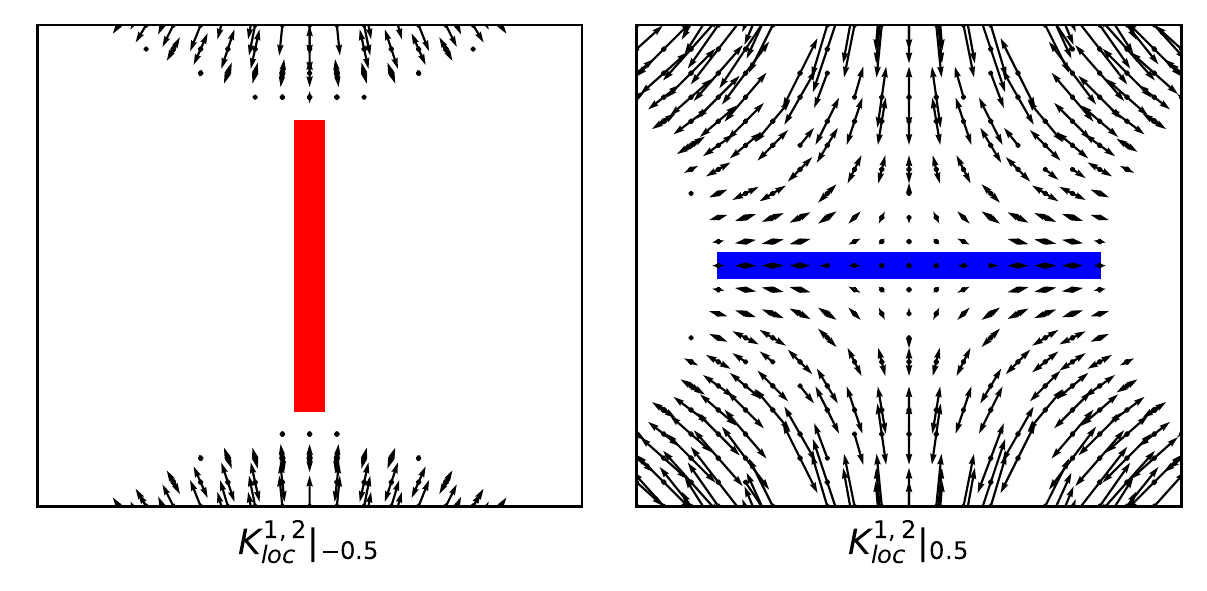}
    \caption{Slices of the surgery 2-handle $K^{1,2}_{loc}\subset \CC^2$ from before and after the critical point. These correspond to models of the surgery and anti-surgery neck. The positive end (right subfigure) is an immersed $D^1\times S^1$, whose isotropic Whitney $S^1$ is highlighted in blue. The negative end(left subfigure) is an embedded $S^0\times D^2$, whose isotropic antisurgery $D^1$ from \cite{haug2015lagrangian} is drawn in red. }
    \label{fig:handle12}
\end{figure}
An important observation is that while the cobordisms $K^{k, n-k+1}_{loc}$ and $K^{n-k, k+1}_{loc}$ are topologically inverses, they \emph{are not inverses of each other as Lagrangian cobordisms}.
 The first is a cobordism between an immersed $S^k\times D^{n-k}$ and an embedded $D^{k+1}\times S^{n-k-1}$, while the second is between an immersed $S^{n-k-1}\times D^{k+1}$ and an embedded $D^{n-k}\times S^k$ (see \cref{thm:standardTraceProperties}).
 This can be seen in examples by comparing \cref{fig:handle03,fig:handle12}.
\subsubsection{Lagrangian surgery trace}
\label{subsubsec:surgerytrace}
In this section we prove \cref{thm:standardTraceProperties}. We now apply \cref{subsubsec:Xdecomposition} to build from $K^{k, n-k+1}_{loc}$ a Lagrangian cobordism with fixed boundary.
For this construction, we write $K:=K^{k, n-k+1}_{loc}$.
Pick a radius $A\in \RR_{>0}$. 
Let $L^{k, n-k+1, 0}= K|_0$; see for instance \cref{fig:handle11}.
As a subset of the domain  $\RR^n\times \RR^1$ parameterizing $K$, the domain parameterizing the Lagrangian $L^{k, n-k+1, 0}$ is given by the locus 
\[\left\{(x_0,x_1, \ldots, x_n )\:|\; x_0^2+\sum_{i=1}^n\sigma_{i, k}x_i^2=0\right\}.\]
We then take hypersurface  $M_{k, n-k+1}\subset L^{k, n-k+1, 0}$ cut out by  $x_0^2+x_1^2+\cdots x_n^2=1$.
As in the proof of \cref{prop:splittingX}, take an extension $M\times I_s \times I_t\subset K$ which is disjoint from the subset $V= x_0^2+x_1^2+\cdots x_n^2=1/4$.
By using \cref{prop:splittingX} to perform a decomposition across the $X$-coordinate along $M$, we obtain a Lagrangian cobordism  with fixed boundary $\uu{K}_{\downarrow}$. We use the notation from \cref{eqn:decompositionNotation}, and we designate the $\downarrow$ component to be the one which contains the origin in $D^n\times D^1$.

Since $\left(\uu{K}_{\downarrow}\right)\cap V=\left(K\right)\cap V$, we obtain 
\[ \left.\left(\uu{K}_{\downarrow} \cap V\right)\right|_{1/2}= \left(\left.\left(K\right)\right|_{1/2}\right)\cap V.\]
We define the \emph{standard Lagrangian trace} of area $\frac{1}{6}$ to be 
\[K^{k, n-k+1}_{\frac{1}{6}}:=\left.\uu{K}_{\downarrow}\right\|_{[-1/2, 1/2]},\]
where the double vertical bar refers to the truncation from \cref{def:lagrangiancutout}.
The standard Lagrangian surgery trace of area $A$ is then defined to be the rescaling (under the map $z\mapsto c\cdot z$ on $\CC^n$) of the previously constructed Lagrangian submanifold, 
\[K^{k, n-k+1}_{A}:=6A\cdot K^{k, n-k+1}_{\frac{1}{6}}.\]
The ends of the standard Lagrangian surgery trace of area $A$ will be denoted:
\[K^{k, n-k+1}_{A}:L^{k, n-k,+}_{A}\rightsquigarrow L^{k+1, n-k-1,-}_{A}.\]
\begin{remark}
    Note that in the case of $k=-1, n$, this simply corresponds to truncation 
    \[K^{k, n-k+1}_{A}=K^{k, n-k+1}_{loc}\|_{<r}\]
    where $A=\frac{4r^3}{3}$.
\end{remark}
\begin{remark}
    When $k=-1$, we have a Lagrangian cobordism $K^{-1, n+2}_{A}= (K^{n, 1})^{-1}$. This case differs slightly from the standard Lagrangian Surgery trace in that the positive end $L^{-1, n+1, +}_{A}$ is empty, and the negative end $L^{0, n, -}_{A}$ is a Whitney sphere. 
\end{remark}
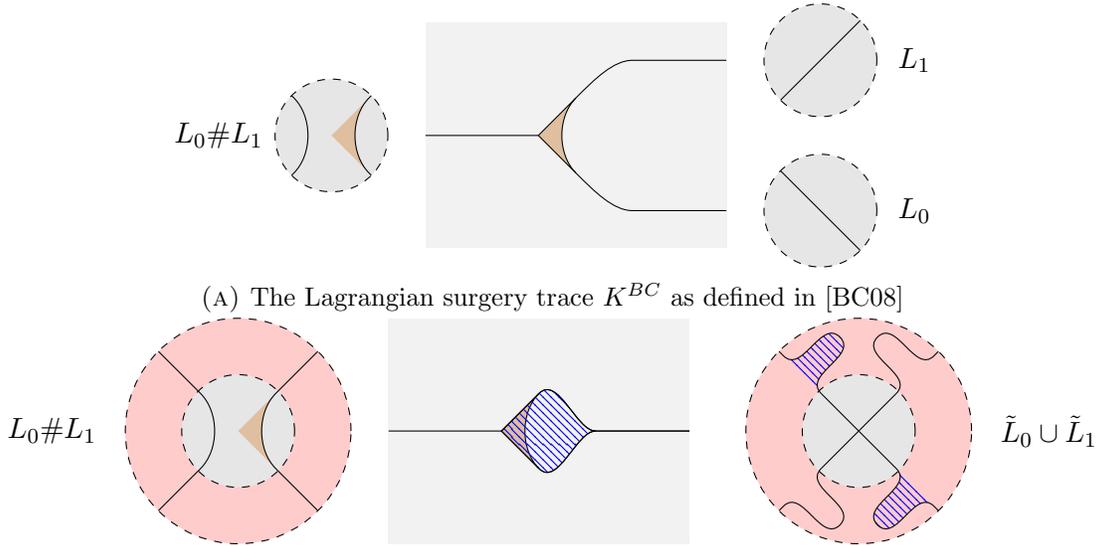
\begin{figure}
    \centering
    \begin{subfigure}{\linewidth}
        \centering
        \begin{tikzpicture}[scale=.5]

    \fill[gray!10]  (-3.5,2) rectangle (4.5,-4);
    
    \begin{scope}[shift={(1.5,-1)}]
    
    \draw[fill=gray!20, dashed]  (5.5,-2) ellipse (1.5 and 1.5);
    \clip  (5.5,-2) ellipse (1.5 and 1.5);
    \draw (4,-0.5) -- (7,-3.5);

    \end{scope}
    \begin{scope}[shift={(1.5,3)}]
    
    \draw[fill=gray!20, dashed]  (5.5,-2) ellipse (1.5 and 1.5);
    \clip  (5.5,-2) ellipse (1.5 and 1.5);
    \draw (4,-3.5) -- (7,-0.5);
    \end{scope}

    \begin{scope}[]
    \draw[fill=brown!50] (0.5,0) .. controls (0,-0.5) and (0,-0.5) .. (-0.5,-1) .. controls (0,-1.5) and (0,-1.5) .. (0.5,-2) .. controls (0,-1.5) and (0,-0.5) .. (0.5,0);
    \end{scope}

    \begin{scope}[shift={(-5.5,0)}]
    
    \draw[fill=gray!20, dashed]  (-0.5,-1) ellipse (1.5 and 1.5);
    \clip  (-0.5,-1) ellipse (1.5 and 1.5);
    \fill[fill=brown!50] (0.5,0) .. controls (0,-0.5) and (0,-0.5) .. (-0.5,-1) .. controls (0,-1.5) and (0,-1.5) .. (0.5,-2) .. controls (0,-1.5) and (0,-0.5) .. (0.5,0);
    
    \draw (1,0.5) .. controls (0.5,0) and (0.5,0) .. (0.5,0) .. controls (0,-0.5) and (0,-1.5) .. (0.5,-2) .. controls (0.5,-2) and (1,-2.5) .. (1,-2.5);
    \draw (-2,0.5) .. controls (-2,0.5) and (-2,0.5) .. (-1.5,0) .. controls (-1,-0.5) and (-1,-1.5) .. (-1.5,-2) .. controls (-2,-2.5) and (-2,-2.5) .. (-2,-2.5);
    
    \end{scope}\draw (-3.5,-1) -- (-0.5,-1);
    \draw (0.5,-2) .. controls (1,-2.5) and (1.5,-3) .. (2,-3) .. controls (2.5,-3) and (4,-3) .. (4.5,-3);
    \draw (0.5,0) .. controls (1,0.5) and (1.5,1) .. (2,1) .. controls (2.5,1) and (4,1) .. (4.5,1);
    \node at (9.5,-3) {$L_0$};
    \node at (9.5,1) {$L_1$};
    \node at (-9,-1) {$L_0\# L_1$};
    \end{tikzpicture}         \caption{The Lagrangian surgery trace $K^{BC}$ as defined in \cite{biran2008lagrangian}}
        \label{fig:bcSurgeryTrace}
    \end{subfigure}
    \begin{subfigure}{\linewidth}
        \begin{tikzpicture}[scale=.5]
    \fill[gray!10]  (-3.5,2) rectangle (4.5,-4);

    \begin{scope}[shift={(3.5,1)}]
    
    \draw[fill=red!20, dashed]  (5.5,-2) ellipse (3 and 3);
    \draw[fill=gray!20, dashed]  (5.5,-2) ellipse (1.5 and 1.5);
    \clip  (5.5,-2) ellipse (3 and 3);
    \draw (4.5,-3) -- (6.5,-1);
    \begin{scope}[shift={(-3,-1)}]
    \begin{scope}[]
    \draw  (6,-3.5) .. controls (6,-3.5) and (6,-3.5) .. (6.5,-3) .. controls (7,-2.5) and (7.5,-4) .. (8,-3.5) .. controls (8.5,-3) and (7,-2.5) .. (7.5,-2);
    
    \end{scope}
    \begin{scope}[xscale=-1, shift={(-17,0)}]
    \draw[pattern=north west lines, pattern color=blue]   (6,-3.5) .. controls (6,-3.5) and (6,-3.5) .. (6.5,-3) .. controls (7,-2.5) and (7.5,-4) .. (8,-3.5) .. controls (8.5,-3) and (7,-2.5) .. (7.5,-2);
    
    \end{scope}
    \begin{scope}[yscale=-1, shift={(0,2)}]
    \draw[pattern=north west lines, pattern color=blue] (6,-3.5) .. controls (6,-3.5) and (6,-3.5) .. (6.5,-3) .. controls (7,-2.5) and (7.5,-4) .. (8,-3.5) .. controls (8.5,-3) and (7,-2.5) .. (7.5,-2);
    
    \end{scope}
    \begin{scope}[scale=-1, shift={(-17,2)}]
    \draw  (6,-3.5) .. controls (6,-3.5) and (6,-3.5) .. (6.5,-3) .. controls (7,-2.5) and (7.5,-4) .. (8,-3.5) .. controls (8.5,-3) and (7,-2.5) .. (7.5,-2);
    
    \end{scope}\draw (7.5,0) -- (9.5,-2);
    \end{scope}
    \end{scope}

    \begin{scope}[]
    \draw[fill=brown!50] (0.5,0) .. controls (0,-0.5) and (0,-0.5) .. (-0.5,-1) .. controls (0,-1.5) and (0,-1.5) .. (0.5,-2) .. controls (0,-1.5) and (0,-0.5) .. (0.5,0);
    \end{scope}

    \begin{scope}[shift={(-7,0)}]
    
    \draw[fill=red!20, dashed]  (-0.5,-1) ellipse (3 and 3);
    \draw[fill=gray!20, dashed]  (-0.5,-1) ellipse (1.5 and 1.5);
    \clip  (-0.5,-1) ellipse (3 and 3);
    \fill[fill=brown!50] (0.5,0) .. controls (0,-0.5) and (0,-0.5) .. (-0.5,-1) .. controls (0,-1.5) and (0,-1.5) .. (0.5,-2) .. controls (0,-1.5) and (0,-0.5) .. (0.5,0);
    
    \draw (2,1.5) .. controls (0.5,0) and (0.5,0) .. (0.5,0) .. controls (0,-0.5) and (0,-1.5) .. (0.5,-2) .. controls (0.5,-2) and (2,-3.5) .. (2,-3.5);
    \draw (-3,1.5) .. controls (-3,1.5) and (-3,1.5) .. (-1.5,0) .. controls (-1,-0.5) and (-1,-1.5) .. (-1.5,-2) .. controls (-3,-3.5) and (-3,-3.5) .. (-3,-3.5);
    
    \end{scope}\draw (-3.5,-1) -- (-0.5,-1);
    \draw (0.5,-2) .. controls (1,-2.5) and (1.5,-1) .. (2,-1) .. controls (2.5,-1) and (4,-1) .. (4.5,-1);
    \draw (0.5,0) .. controls (1,0.5) and (1.5,-1) .. (2,-1) .. controls (2.5,-1) and (4,-1) .. (4.5,-1);
    
    \draw[pattern=north west lines, pattern color=blue] (2,-1) .. controls (1.5,-1) and (1,0.5) .. (0.5,0) .. controls (0,-0.5) and (0,-0.5) .. (-0.5,-1) .. controls (0,-1.5) and (0,-1.5) .. (0.5,-2) .. controls (1,-2.5) and (1.5,-1) .. (2,-1);
    \node[right] at (12.5,-1) {$\tilde L_0 \cup \tilde L_1$};
    \node[left] at (-11,-1) {$L_0\# L_1$};
    \end{tikzpicture}         \centering
        \caption{The standard Lagrangian surgery trace $K^{0,2}_A$. The application of \cref{prop:splittingX}  to $K^{BC}$ modifies the Lagrangian cobordism over the red neighborhood. The flux $A$ is the difference between the blue and brown areas.}
        \label{fig:ourSurgeryTrace}
    \end{subfigure}
    \caption{Comparison between Polterovich surgery and the standard Lagrangian surgery trace. Notice that the fluxes swept over the surgeries differ (both in magnitude and sign).}
\end{figure}
\begin{example}
    The ends of the Lagrangian surgery trace $K^{0,n}_A:L^{0,n,+}_A\rightsquigarrow L^{1,n-1, -}_A$ which resolves a single transverse intersection of a $n$-dimensional Lagrangian do not quite agree with the standard pictures drawn for the Polterovich surgery. In particular, the flux of the surgery (which determines the map $\omega: H_2(X, L^{1,n-1,-}_A)\to \RR$ in terms of $\omega: H_2(X, L^{0,n,+}_A)\to \RR$) is surprisingly counterintuitive. We now describe the flux swept out by the local model for the standard Lagrangian surgery trace when $\dim(X)=2$. This example is based on the computation of flux which appears in \cite[Section 4.1]{hicks2019wall} and the discussion surrounding \cite[Figure 8]{haug2015lagrangian}.

    In \cite{polterovich1991surgery} the local model for Polterovich surgery of two Lagrangian submanifolds intersecting transversely at a point replaces the Lagrangians $L_0, L_1\subset \RR^2$ (as drawn on the right-hand side of \cref{fig:bcSurgeryTrace}) with the Lagrangian $L_0\#L_1$ (as drawn on the left-hand side of \cref{fig:bcSurgeryTrace}). \Cref{fig:bcSurgeryTrace} also depicts the Lagrangian surgery cobordism $K^{BC}: (L_0, L_1)\rightsquigarrow (L_0\# L_1)$ as defined in \cite{biran2008lagrangian}. The flux of the surgery --- the area highlighted in brown on the left-hand side --- is equal to the shadow of $K^{BC}$.  This 3-ended Lagrangian cobordism is not a Lagrangian cobordism with cylindrical boundary (as it has three ends), so it is \emph{not} a local model for the standard Lagrangian trace (as defined in \cref{subsubsec:surgerytrace}).
        
    To obtain a 2-ended Lagrangian cobordism with cylindrical boundary from $K^{BC}$, one must apply a Lagrangian isotopy which cylindricalizes the boundary (\cref{prop:splittingX}). The resulting Lagrangian cobordism $K^{0,2}_A:(\tilde L_0\cup \tilde L_1)\rightsquigarrow L_0\# L_1$ is drawn in \cref{fig:ourSurgeryTrace}. A subtle point is that the positive end of this Lagrangian cobordism is no longer $L_0\cup L_1$.   The construction from \cref{prop:splittingX} covers $K^{0,2}_A$ with two charts. The first chart agrees with the Lagrangian cobordism $K^{BC}$ from before. The second chart, contained in the region highlighted in red, is a suspension of a Hamiltonian isotopy of $L_i$ restricted to the red region. The flux of this suspension is the blue hatched region in \cref{fig:ourSurgeryTrace}, and equal to $\Area(K^{0, 2}_A)$. Observe that $\Area(K^{0, 2}_A)>\Area(K^{BC})$.  
    As a consequence, the area bounded by $\tilde L_0\cup \tilde L_1$ and $L_0\# L_1$ has the \emph{opposite sign} of the area between $L_0\cup L_1$ and $L_0\# L_1$!  The quantity $A= \Area(K^{0, 2}_A)-\Area(K^{BC})$ describes the symplectic area bounded by $\tilde L_0\cup \tilde L_1$ and $L_0 \# L_1$.
    \label{exam:fluxDirection}
\end{example}
The construction of a standard Lagrangian surgery handle allows us to define the standard Lagrangian surgery trace.
\begin{definition}
    We say that $K: L^+\rightsquigarrow L^-$ is a standard Lagrangian surgery trace if it admits a decomposition across the $X$ coordinate as $K=K_-\cup_{S^{n-k}\times S^n} K^{k, n-k+1}_{A}$, where $K_-$ is a suspension Lagrangian cobordism with collared boundary. 
\end{definition}
While the standard Lagrangian surgery trace is a useful cobordism to have, a geometric setup for performing Lagrangian surgery on a given Lagrangian $L^+$ is desirable.
Such a criterion is given in  \cite{haug2015lagrangian} by the anti-surgery disk. 
In that paper, it was noted that the presence of a Whitney isotropic $k$ sphere was a necessary but not sufficient condition for implanting a Lagrangian surgery handle.
We give a sufficient characterization in \cref{rem:surgeryconfiguration}.
 
\subsection{Cobordisms are iterated surgeries}
\label{subsec:cobordismsaresurgery}
Having described the Lagrangian surgery operation and trace cobordism, we show that all Lagrangian cobordisms decompose into a concatenation of surgery traces and exact homotopies.
This characterization is analogous to the handle body decomposition of cobordisms from the data of a Morse function. 

\begin{theorem}
      Let $K: L^+\rightsquigarrow L^-$ be a Lagrangian cobordism. 
      Then there is a sequence of Lagrangian cobordisms  
      \begin{align*} 
            K_{H^i_t}: &L_{i+1}^-\rightsquigarrow L_{i}^+ \text{ for $i\in \{0, \ldots, j\}$}\\
            K_{i}^{k_i, n-k_i+1}:& L_i^+\rightsquigarrow L_i^- \text{ for $i\in \{1, \ldots j\}$}
      \end{align*}
      which satisfy the following properties:
      \begin{itemize}
            \item $L_{j+1}^-=L^+$ and $L_0^+=L^-$
            \item Each $K^{k_i, n-k_i+1}$ is a Lagrangian surgery trace;
            \item Each $K_{H^i_t}$ is the suspension of an exact homotopy and;
            \item There is an exact homotopy between  
            \[K\sim K_{H^j_t}\circ K_j^{k_j, n-k_j+1}\circ K_{H^{j-1}_t}\circ \cdots \circ K_{H^1_t}\circ K_1^{k_1, n-k_1+1}\circ K_{H^0_t}.\]
            Furthermore, this construction can be performed in such a way that the exact homotopy has as small Hofer norm as desired.
      \end{itemize}
      \label{thm:cobordismsaresurgery}
\end{theorem}
The decomposition comes from using the function $\pi_\RR: K\to \RR$ to provide a handle body decomposition of $K$. We note that unless $K$ is the suspension of an exact isotopy, the decomposition will necessarily be immersed (as the Lagrangian surgery traces are all immersed).
\subsubsection{Morse Lagrangian Cobordisms}
We first must show that $K$ can be placed into general position by exact homotopy so that $\pi_\RR$ is a Morse function (as in \cref{exam:shearedtorus}).
\begin{claim}[Morse Lemma for Lagrangian Cobordisms]
      Let $K\subset X\times \CC$ be a Lagrangian cobordism.
      There exists $K'$, a Lagrangian cobordism exactly homotopic to $K$, with $\pi_\RR: K'\to \RR$ a Morse function.
      Furthermore, the construction can be conducted so that 
      \begin{itemize}
            \item the Hofer norm of the exact homotopy is as small as desired; and 
            \item if $K$ is embedded, then $K'$ is embedded as well.
      \end{itemize}
      \label{claim:morsecobordism}
\end{claim}
\begin{proof}
      By abuse of notation, we will use $K$ to denote the smooth manifold parameterizing the (possibly immersed) Lagrangian cobordism $K$. 
      Let $\phi:B^*_\epsilon K\to X$ be a local Weinstein neighborhood. At each point $x\in B^*_\epsilon K$ there is a Darboux neighborhood of $\phi(x)$, which can be chosen to be the product of Darboux neighborhoods of $\pi_X\circ \phi(x)$ and $\pi_\CC\circ \phi(x)$. Therefore, there exists around $x\in B^*_\epsilon K$ Darboux coordinates $(q_0, p_0, \ldots, q_n, p_n)$ so that 
      \begin{align*}
            q_0:= \pi_\RR\circ \phi && p_0:= \pi_{\jmath \RR} \circ \phi
      \end{align*} 
      are the pullbacks of the real and imaginary coordinates to the local Weinstein neighborhood. 
      Let $C^\infty_\epsilon (B^*_\epsilon K)$ be the smooth functions $B^*_\epsilon K \to \RR$ with compact support disjoint from the boundary.  
      Let $C^\infty_{cob}(K;\RR)$ be the functions which agree with $\pi_\RR: K\to \RR$ outside of a compact set.
      Given $H\in C^\infty_{\epsilon}(B^*_\epsilon K)$, let $\psi^t_H$ be the time $t$ Hamiltonian flow of $H$, and let $\lj^t_H=\phi\circ \psi^t_H$ be the corresponding exactly homotopic immersion of $K$. 
      We obtain a map 
      \begin{align*}
            \mathcal P: C^\infty_\epsilon (B^*_\epsilon K) \to& C^\infty_{cob}(K;\RR)\\
                  H\mapsto& \pi_\RR\circ \lj^1_H.
      \end{align*}
      so that  $\mathcal P(H)$ is the real coordinate of the immersion $\lj^t_H$. 
      We will show that this map is a submersion, and in particular open.
      Let $f:K\to \RR$ be a function with compact support, representing a tangent direction of $C^\infty_{cob}(K;\RR)$. 
      As $K\subset B^*_\epsilon K$ is embedded, $f$ can be extended to a compactly supported function $F:B^*_\epsilon K \to \RR$ so that $F|_K=f$, and $F\in C^\infty_\epsilon (B^*_\epsilon K)$.
      The flow of $H$ in the $q_0$ coordinate is 
      \[\frac{d  {q}_0}{dt}=\frac{dH}{d  {p}_0}.\]
      We define our Hamiltonian $H_f:B^*_\epsilon K\to \RR$  by the integral
      \[H_{f}(x):=\int_{\pi_X(x)\times q_0(x)\times (-\infty,  {p}_0(x))} F d {p}_0.\]
      With this choice of Hamiltonian, the Hamiltonian flow at time zero of the real coordinate at a point $x\in K\subset B^*_\epsilon K$ is given by
      \[\frac{d}{dt} (\pi_\RR\circ \lj^t_{H_f})|_{t=0} (x)=\left.\frac{d q_0}{dt}\right|_{t=0}(x)=f(x).\]
      This shows that $\mathcal P$ is a submersion at $0$.
      Since every open set of $C^\infty_{cob}(K;\RR)$ contains a Morse function, and the image of $\mathcal P$ is open, there is a choice of Hamiltonian $H$ near $0$ so that $\mathcal P(H)=\pi_\RR\circ \lj^1_{H_f}$ is a Morse function on $K$. 

      Because the Hamiltonian can be chosen near zero, we can choose it so that $\sup_{x\in K} H(x)-\inf_{x\in K} H(x)$ is bounded by a constant as small as desired. This shows that the exact homotopy associate to the time 1 flow of $H$ has as small Hofer norm as desired.
\end{proof}
A similar argument shows that every $K$ is exactly homotopic to $K'$ with the property that $\Crit(\pi_\RR: K'\to \RR)$ is disjoint from $\mathcal I^{si}(K')$, the set of self-intersections of $K'$.
If $\pi_\RR: K\to \RR$ is a Morse function whose critical points are disjoint from its self-intersections, we say that the Lagrangian cobordism is a \emph{Morse-Lagrangian cobordism}. 
\subsubsection{Placing Cobordisms in good position}
The Lagrangian condition forces a certain amount of independence between the $\pi_\RR\circ \lj$ and $\pi_{\jmath \RR}\circ \lj$ projections of the Lagrangian cobordism.
\begin{claim}
      Let $\lj: K\to X\times \CC$ be a Morse-Lagrangian cobordism.
      Then $x\in K$ cannot be a critical point of both $\pi_\RR\circ \lj$ and $\pi_{\jmath\RR}\circ \lj$. 
\end{claim}
\begin{proof}
      If so, then $\lj_*(T_pK)\subset TX\subset TX\times T\CC$.
      Since $\lj_*(T_pK)$ is a Lagrangian subspace, it cannot be contained in any proper symplectic subspace of $T(X\times \CC)$. 
\end{proof}
Even when $K$ is a Morse-Lagrangian cobordism, it need not be the case that at a critical point $q\in \Crit(\pi_\RR: K\to \RR)$ that $K|_{q-\epsilon}$ is obtained from $K|_{q+\epsilon}$  by surgery.
In the simplest counterexample, $K|_{q-\epsilon}$ could be obtained from $K|_{q+\epsilon}$ by anti-surgery; simply knowing the index of a $\pi_\RR\circ \lj$ critical point does not determine if it arises from surgery or antisurgery! 

\begin{example}
      \label{exam:zipper}
      We provide some intuition for what additional information is needed to determine if a critical point gives a surgery or an antisurgery.
      Suppose that $q$ is an index-1 point of a 2-dimensional Lagrangian cobordism $K$. Then there exist local coordinates around $q$ so that $\pi_\RR\circ \lj$ can be written as $q_0^2-q_1^2$. Since the critical points of $\pi_{\jmath \RR}\circ \lj$ are disjoint from those of $\pi_\RR\circ \lj$, the differential of $\pi_{\jmath \RR} \circ \lj$ is non-vanishing at $q$. We now assume that  $\pi_{\jmath \RR} \circ \lj$ is linear in $(q_0, q_1)$ -coordinates (note that this will generally not be the case). We then can write $(\pi_{\jmath \RR} \circ \lj )=aq_0+bq_1$. The slices of this Lagrangian cobordism are then the level sets of $q_0^2-q_1^2=t$, and the primitive describing the exact homotopy between those slices is $aq_0+bq_1$. We look at three cases (summarized in \cref{fig:3cases})
      \begin{enumerate}
            \item If $|a|>|b|$, then $aq_0+bq_1$ restricted to $q_0^2-q_1^2=t$ will have two critical points when $t>0$. Let $q_t, q_t'$ be these two  critical points. Let us make another assumption (which in general does not hold), which is that these critical points are fixed points of the homotopy $K|_t$ (i.e. $\pi_X\circ \lj(q_{t_0})=\pi_X\circ \lj(q_{t_1})$  and $\pi_X\circ \lj(q'_{t_0})=\pi_X\circ \lj(q'_{t_1})$ for all $t_0, t_1>0$). Since $\lim_{t\searrow 0} q_t= \lim_{t\searrow 0}q_t' =(0,0)$ we obtain that $\lj(q_{t})=\lj(q'_{t})$ whenever $t>0$. We conclude that the positive slices are immersed (making $K$ a surgery).
            \item If instead $|b|<|a|$, the same argument holds except that $aq_0+bq_1$ has critical points on the negative slices of $K$. $K$ then gives an antisurgery.
            \item The last case is degenerate: when $|a|=|b|$, both the positive and negative slices are embedded, but the critical slice is immersed along a set of codimension 0 (see \cref{fig:degenerate})! 
      \end{enumerate}

\end{example}

\begin{figure}
      \label{fig:3cases}
      \centering
      \begin{tikzpicture}[scale=.7]
\begin{scope}[]\begin{scope}[]

\draw[dotted]  (0,0) ellipse (1.5 and 1.5);

\clip  (0,0) ellipse (1.5 and 1.5);
\draw (-1.5,1.5) -- (1.5,-1.5) (1.5,1.5) -- (-1.5,-1.5);
\begin{scope}[black!60!green]

\draw (-1.3,1.5) .. controls (0,0.15) and (0,0.15) .. (1.3,1.5);
\draw (-1,1.5) .. controls (0,0.7) and (0,0.7) .. (1,1.5);
\draw (-0.5,1.5) .. controls (0,1.2) and (0,1.2) .. (0.5,1.5);
\end{scope}
\begin{scope}[rotate=90,black!60!blue]

\draw (-1.3,1.5) .. controls (0,0.15) and (0,0.15) .. (1.3,1.5);
\draw (-1,1.5) .. controls (0,0.7) and (0,0.7) .. (1,1.5);
\draw (-0.5,1.5) .. controls (0,1.2) and (0,1.2) .. (0.5,1.5);
\end{scope}
\begin{scope}[rotate=180,black!60!green]

\draw (-1.3,1.5) .. controls (0,0.15) and (0,0.15) .. (1.3,1.5);
\draw (-1,1.5) .. controls (0,0.7) and (0,0.7) .. (1,1.5);
\draw (-0.5,1.5) .. controls (0,1.2) and (0,1.2) .. (0.5,1.5);
\end{scope}
\begin{scope}[rotate=270,black!60!blue]

\draw (-1.3,1.5) .. controls (0,0.15) and (0,0.15) .. (1.3,1.5);
\draw (-1,1.5) .. controls (0,0.7) and (0,0.7) .. (1,1.5);
\draw (-0.5,1.5) .. controls (0,1.2) and (0,1.2) .. (0.5,1.5);
\end{scope}

\begin{scope}[red]

\draw (-1.5,1.5) -- (1.5,1.5) (-1.5,1) -- (1.5,1) (-1.5,0.5) -- (1.5,0.5) (-1.5,0) -- (1.5,0) (-1.5,-0.5) -- (1.5,-0.5) (-1.5,-1) -- (1.5,-1) (-1.5,-1.5) -- (1.5,-1.5);

\draw (-1.5,1.25) -- (1.5,1.25) (-1.5,0.75) -- (1.5,0.75) (-1.5,0.25) -- (1.5,0.25) (-1.5,-0.25) -- (1.5,-0.25) (-1.5,-0.75) -- (1.5,-0.75) (-1.5,-1.25) -- (1.5,-1.25)   ;

\end{scope}
\end{scope}
\node[black!60!blue] at (-3,0) {$\pi^{-1}_{\mathbb R}(t>0)$};
\node[black!60!green] at (0,-2) {$\pi^{-1}_{\mathbb R}(t<0)$};

\node[red] at (-2.5,-1) {$\pi_{\jmath\RR}^{-1}(s)$};
\end{scope}

\begin{scope}[shift={(5, 0)}]\begin{scope}[]

\draw[dotted]  (0,0) ellipse (1.5 and 1.5);

\clip  (0,0) ellipse (1.5 and 1.5);
\draw (-1.5,1.5) -- (1.5,-1.5) (1.5,1.5) -- (-1.5,-1.5);
\begin{scope}[black!60!green]

\draw (-1.3,1.5) .. controls (0,0.15) and (0,0.15) .. (1.3,1.5);
\draw (-1,1.5) .. controls (0,0.7) and (0,0.7) .. (1,1.5);
\draw (-0.5,1.5) .. controls (0,1.2) and (0,1.2) .. (0.5,1.5);
\end{scope}
\begin{scope}[rotate=90,black!60!blue]

\draw (-1.3,1.5) .. controls (0,0.15) and (0,0.15) .. (1.3,1.5);
\draw (-1,1.5) .. controls (0,0.7) and (0,0.7) .. (1,1.5);
\draw (-0.5,1.5) .. controls (0,1.2) and (0,1.2) .. (0.5,1.5);
\end{scope}
\begin{scope}[rotate=180,black!60!green]

\draw (-1.3,1.5) .. controls (0,0.15) and (0,0.15) .. (1.3,1.5);
\draw (-1,1.5) .. controls (0,0.7) and (0,0.7) .. (1,1.5);
\draw (-0.5,1.5) .. controls (0,1.2) and (0,1.2) .. (0.5,1.5);
\end{scope}
\begin{scope}[rotate=270,black!60!blue]

\draw (-1.3,1.5) .. controls (0,0.15) and (0,0.15) .. (1.3,1.5);
\draw (-1,1.5) .. controls (0,0.7) and (0,0.7) .. (1,1.5);
\draw (-0.5,1.5) .. controls (0,1.2) and (0,1.2) .. (0.5,1.5);
\end{scope}

\begin{scope}[red, rotate=90]

\draw (-1.5,1.5) -- (1.5,1.5) (-1.5,1) -- (1.5,1) (-1.5,0.5) -- (1.5,0.5) (-1.5,0) -- (1.5,0) (-1.5,-0.5) -- (1.5,-0.5) (-1.5,-1) -- (1.5,-1) (-1.5,-1.5) -- (1.5,-1.5);

\draw (-1.5,1.25) -- (1.5,1.25) (-1.5,0.75) -- (1.5,0.75) (-1.5,0.25) -- (1.5,0.25) (-1.5,-0.25) -- (1.5,-0.25) (-1.5,-0.75) -- (1.5,-0.75) (-1.5,-1.25) -- (1.5,-1.25)   ;

\end{scope}
\end{scope}
\end{scope}

\begin{scope}[shift={(10, 0)}]\begin{scope}[]

\draw[dotted]  (0,0) ellipse (1.5 and 1.5);

\clip  (0,0) ellipse (1.5 and 1.5);
\draw (-1.5,1.5) -- (1.5,-1.5) (1.5,1.5) -- (-1.5,-1.5);
\begin{scope}[black!60!green]

\draw (-1.3,1.5) .. controls (0,0.15) and (0,0.15) .. (1.3,1.5);
\draw (-1,1.5) .. controls (0,0.7) and (0,0.7) .. (1,1.5);
\draw (-0.5,1.5) .. controls (0,1.2) and (0,1.2) .. (0.5,1.5);
\end{scope}
\begin{scope}[rotate=90,black!60!blue]

\draw (-1.3,1.5) .. controls (0,0.15) and (0,0.15) .. (1.3,1.5);
\draw (-1,1.5) .. controls (0,0.7) and (0,0.7) .. (1,1.5);
\draw (-0.5,1.5) .. controls (0,1.2) and (0,1.2) .. (0.5,1.5);
\end{scope}
\begin{scope}[rotate=180,black!60!green]

\draw (-1.3,1.5) .. controls (0,0.15) and (0,0.15) .. (1.3,1.5);
\draw (-1,1.5) .. controls (0,0.7) and (0,0.7) .. (1,1.5);
\draw (-0.5,1.5) .. controls (0,1.2) and (0,1.2) .. (0.5,1.5);
\end{scope}
\begin{scope}[rotate=270,black!60!blue]

\draw (-1.3,1.5) .. controls (0,0.15) and (0,0.15) .. (1.3,1.5);
\draw (-1,1.5) .. controls (0,0.7) and (0,0.7) .. (1,1.5);
\draw (-0.5,1.5) .. controls (0,1.2) and (0,1.2) .. (0.5,1.5);
\end{scope}

\begin{scope}[red, rotate=45]

\draw (-1.5,1.5) -- (1.5,1.5) (-1.5,1) -- (1.5,1) (-1.5,0.5) -- (1.5,0.5) (-1.5,0) -- (1.5,0) (-1.5,-0.5) -- (1.5,-0.5) (-1.5,-1) -- (1.5,-1) (-1.5,-1.5) -- (1.5,-1.5);

\draw (-1.5,1.25) -- (1.5,1.25) (-1.5,0.75) -- (1.5,0.75) (-1.5,0.25) -- (1.5,0.25) (-1.5,-0.25) -- (1.5,-0.25) (-1.5,-0.75) -- (1.5,-0.75) (-1.5,-1.25) -- (1.5,-1.25)   ;

\end{scope}
\end{scope}
\end{scope}

\end{tikzpicture}       \caption{Three cases giving surgery, anti-surgery, and a degenerate Lagrangian cobordism}
\end{figure}
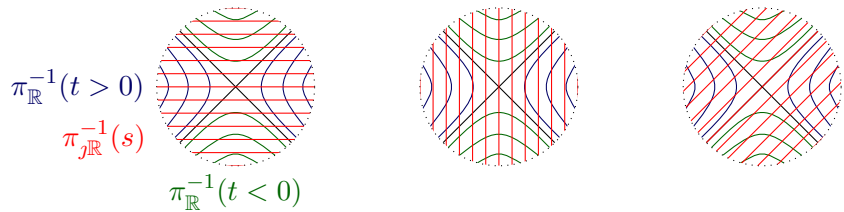
\begin{figure}
      \centering
      \includegraphics{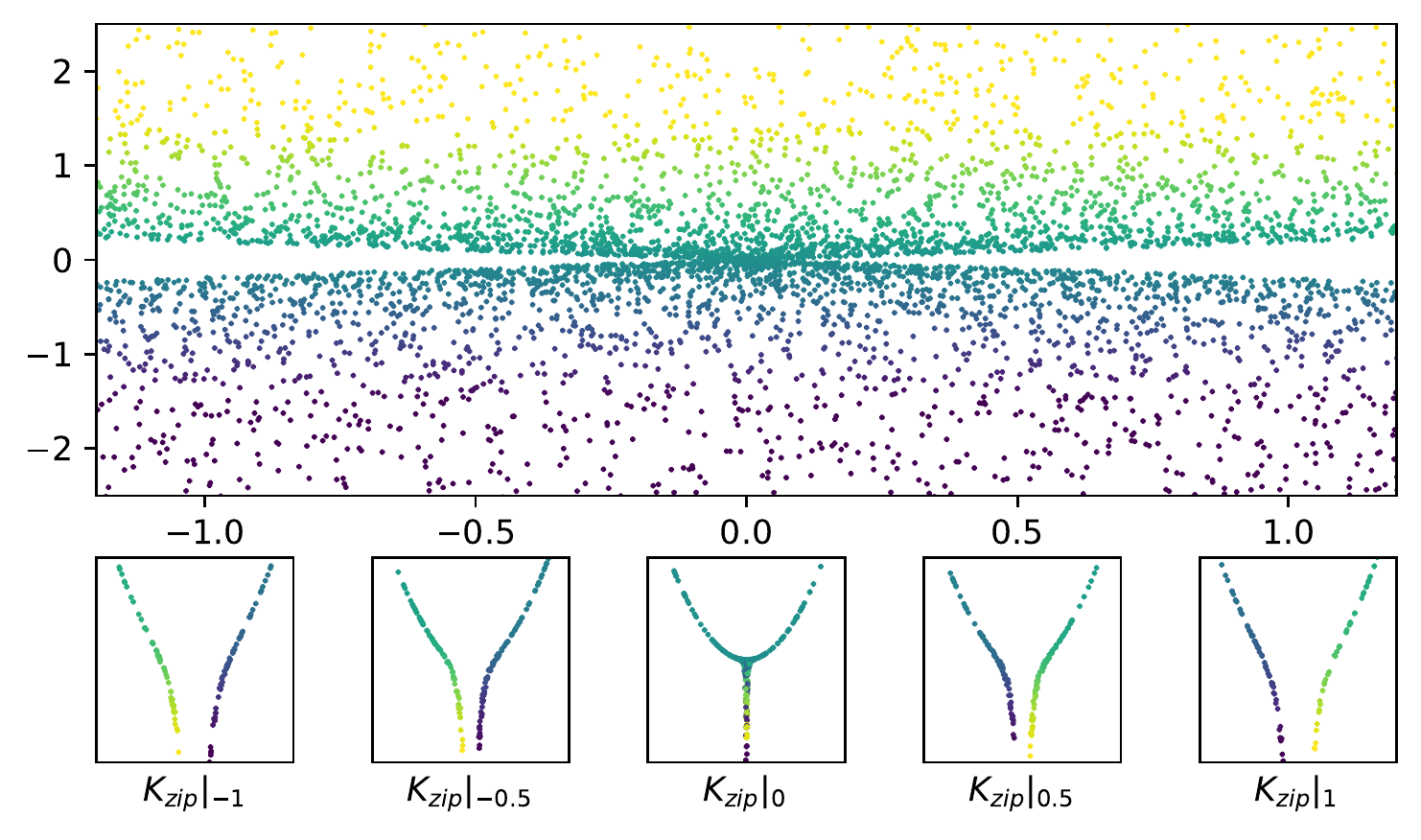}
      \caption{A Morse-Lagrangian cobordism where the critical slice has a codimension-0 set of self-intersection. This Lagrangian cobordism is parameterized by $(x_0, x_1)\mapsto(x_0+x_1-2\jmath x_0x_1, 2(x_0^2-x_1^2)+\jmath(x_0+x_1))$}
      \label{fig:degenerate}
\end{figure}
In order for our critical points of $\pi_\RR$ to correspond to surgeries (case 1 of \cref{exam:zipper}), we need to apply another exact homotopy based on an interpolation between Morse functions.
\begin{claim}[Interpolation of Morse Functions]
      Let $f,g: \RR^n\to \RR$ be Morse functions, each with a single critical point of index $k$ at the origin and $f(0)=g(0)=0$.
      Take $V$ any neighborhood which contains the origin.
      There exists a smooth family of functions $h_c: \RR^n\times[0,1]\to \RR$ which satisfies the following properties. 
      \begin{itemize}
            \item In the complement of $V$,  $h_c|_{\RR^n\setminus V}= f|_{\RR^n\setminus V}$;
            \item There exists a small neighborhood $U$ of the origin so that $ h_1|_U= g|_U$  and;
            \item $h_1$ is Morse with a single critical point;
            \item At time $0$,  $ h_0(x)= f(x)$.
      \end{itemize}
      \label{claim:connectingmorse}
\end{claim}
\begin{proof}
      Pick coordinates $x_1, \ldots, x_n$, and $y_1, \ldots, y_n$ so that in a neighborhood of the origin,
      \begin{align*}
            f= \sum_{i=1}^n \sigma_{i, k} x_i^2 && g=\sum_{i=1}^n \sigma_{i, k} y_i^2.
      \end{align*}
      Let $\phi: \RR^n\to \RR^n$ be a linear map so that $\phi_*(\partial_{y_i})=\partial_{x_i}$.
      Pick $\phi_c:\RR^n\times [0, 1/2]\to \RR^n$ an isotopy of linear maps smoothly interpolating between $\phi_0=\id$ and $\phi_{\frac{1}{2}}=\phi$.
      Take $U_2\subset V$ a small ball around the origin with the property that for all $c$,  $\phi_c(U_2)\subset V$.
      Now consider an path of diffeomorphisms  $\tilde \phi_c: \RR^n\to \RR^n$ satisfying the constraints
      \begin{align*}
            \tilde \phi_c|_{\RR^n\setminus V}= \id && \tilde \phi_0=\id && \tilde \phi_{\frac{1}{2}}|_{U_2}=\phi|_{U_2}
      \end{align*}
      For $c\in [0, 1/2]$ we define $h_c:= f\circ \tilde \phi_c$. 
      
      We now define $h_c$ for $c\in [1/2, 1]$.
      Pick $U_3\subset U_2$ a neighborhood of the origin with the property that for every $q\in U_3$ and $\partial_v\in T_q \RR^n$:
      \[|(d(f\circ \phi)-dg)(\partial_v)|\leq  \frac{1}{2}|df(\partial v)|.\]
      Take an interior subset  $U\subset U_3$ which is a neighborhood of the origin.
      Let $\rho$ be a bump function, which is constantly 1 on $U$, and $0$ outside $U_3$. Let $\tau:[1/2, 1]\to [0, 1]$ be an increasing function smoothly interpolating between $\tau(1/2)=0$ and $\tau(1)=1$. 
      For $c\in [1/2,1]$, let 
      \[h_c:= (1- \tau(c)\rho)\cdot  h_{1/2}+ \tau(c)\rho g.\]

      It remains to show that $h_1$ is Morse, with a unique critical point at the origin.
      For any $q\in \RR^n\setminus U_3$, we have that $dh_1= (\tilde \phi_{\frac{1}{2}})^* df$, which is nonvanishing.
      For any $q\in U$, we have $dh_1= dg$, which vanishes if and only if $q=0$. 
      For $q\in U_3\setminus U$, take $\partial_v\in T_q \RR^n$ with the property that $d\rho(\partial_v)=0$.
      Then 
      \[|dh_1(\partial_v)| =| (1-\rho) df(\partial_v)+\rho df(\partial_v) |>\frac{1}{2} |df(\partial v)|>0.\]
      This proves that $q$ is not a critical point of $h_1$.
\end{proof}
\begin{prop}
      Let $K$ be a Morse Lagrangian cobordism. 
      Let $q\in K$ be a critical point of the projection $\pi_\RR: K\to \RR$ of index $k+1$.
      There exists 
      \begin{itemize}
            \item A neighborhood of the origin $U\subset T^*D^n\times \CC$, and a symplectic embedding $\phi: U\to X\times \CC$ which respects the splitting so that $\phi(0)=q$ and
            \item  $K'$ a Morse Lagrangian cobordism exactly homotopic to $K$. 
      \end{itemize}
      so that $K'\cap U=(K^{k, n-k+1}_{loc})\cap U$ under the identification given by $\phi$. Furthermore
      \begin{itemize}
            \item the critical points of $\pi_\RR|_{K'}$ are in bijection with the critical points of $\pi_\RR|_{K}$;
            \item the Hofer norm of the exact homotopy is as small as desired; and
            \item if $K$ is embedded, then $K'$ is embedded as well.
      \end{itemize}
      \label{prop:goodposition}
\end{prop}
\emph{Proof.} At $q$ take the Lagrangian tangent space $T_qK\subset T_{\pi_X(q)} X\oplus T_{\pi_\CC(x)} \CC$.
      Since $q$ is a critical point of $\pi_\RR$, we have that $T_qK\subset  T_{\pi_X(q)} X\oplus \jmath \RR$. 
      By dimension counting, any set of vectors $\partial v_0, \partial_{v_1}, \ldots, \partial_{v_{n}} \in T_qK$ with the property that $(\pi_X)_*(\partial_{v_1})\neq 0$ for all $v_i$ cannot form a basis of a Lagrangian subspace for $\pi_X(T_qK)$; therefore, there exists a vector $\partial_s\in T_qK$ so that $\pi_X(\partial_s)=0$, and we may split $T_qK=(\pi_X)_*(T_qK)\oplus \jmath \RR$. 
      Choose a Darboux chart  $U\subset T^*\RR^n\times \CC$, $\phi:U\to X\times \CC$ which respects the product decomposition, and has $\phi_*(\RR^n\times \jmath \RR)=T_q K$. Write $\underline K$ for $\RR^n\times \jmath \RR|_U$. 
      Because $K$ and $\underline K$ have the same tangent space, we can further restrict to a Weinstein neighborhood of $\underline K$ so that $K$ is an exact section of the tangent bundle $B^*_\epsilon\underline K$.
      By taking a possibly smaller neighborhood, we will identify $\underline K = D^{n}_{r_0}\times [-s_0, s_0]$, where the $s$-coordinate denotes the $\jmath \RR$ direction.
      Let $F$ be the primitive of this section, so that $K|_{B^*\underline K}= dF$.
     
      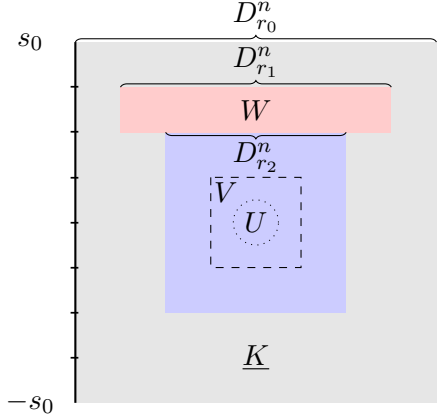
\begin{SCfigure}[50]
            \centering
            
\begin{tikzpicture}[scale=1.2]

\fill[gray!20]  (-3,2.5) rectangle (1,-1.5);
\node at (-3.5,2.5) {$s_0$};

\node at (-3.5,-1.5) {$-s_0$};
\fill[blue!20]  (-2,1.5) rectangle (0,-0.5);
\fill[red!20]  (-2.5,1.5) rectangle (0.5,2);
\draw[dashed]  (-1.5,1) rectangle (-0.5,0);
\node at (-2,0.5) {};
\node at (-1.3336,0.8336) {$V$};

\draw[decoration={brace, mirror},decorate](0.5,2) -- node[above] {$D^{n}_{r_1}$}(-2.5,2);

\draw[decoration={brace},decorate](0,1.5) -- node[below] {$D^{n}_{r_2}$}(-2,1.5);
\draw[decoration={brace, mirror},decorate](1,2.5) -- node[above] {$D^{n}_{r_0}$}(-3,2.5);
    \node at (-1,-1) {$\underline K$};
\draw[ thick] (-3,2.5) -- (-3,-1.5);
\node at (-3,2) {-};
\node at (-3,1.5) {-};
\node at (-3,1) {-};
\node at (-3,0.5) {-};
\node at (-3,0) {-};
\node at (-3,-0.5) {-};
\node at (-3,-1) {-};
\node at (-1,1.75) {$W$};
\draw[dotted]  (-1,0.5) ellipse (0.25 and 0.25);
\node at (-1,0.5) {$U$};
\end{tikzpicture}             \caption{Nested neighborhoods for interpolating primitives. At the end of interpolation: our Lagrangian handle matches the standard surgery handle over $U$; the region $V$ is used to interpolate between $g= \partial_s G$ and $f= \partial_s F$; the region $W$ is subsequently used to correct $h_c$ so that its integral over the $s$-coordinate matches $F$ at the boundary of $K$. }
            \label{fig:nested}
      \end{SCfigure}
      If we let $s$ be the coordinate on $\underline K$ which travels in the $\jmath \RR$ direction, then we can compute $\pi_\RR: K\to \RR$ at $q\in K$ by
      \[f(q):=\pi_\RR(q)= \partial_sF(q),\]
      which is a Morse function on $\underline K$ with a single critical point.

      We now implement the handle in this neighborhood. Let $G^{k, n-k+1}:\underline K\to \RR$ be the primitive for the handle from \cref{def:localsurgerytrace}, so that $g:=\partial_s G^{k, n-k}$ is a Morse function on $\underline K$.
      We will use \cref{claim:connectingmorse} to obtain a function $h_c$ interpolating between $f$ and $g$, and define a preliminary primitive $H_c^{pre}:=\int h_c ds$.
      The section associated to $H_c^{pre}$ will satisfy all the Morse properties we desire; however it does not agree with $F$ outside near the boundary of $\underline K$.
      This is because while $f$ and $h_c$ agree at the boundary of $\underline K$, there is no reason for  $\int_{s_0}^s h_c(x, s)ds$ match $\int_{s_0}^s f(x, s)  ds= F(x, s)$ near the boundary of $\underline K$. 
      Therefore, we need to add a correcting term to $h_c$ in order to make these integrals agree near the boundary of $\underline K$.
      
      We set up this correction using a neighborhood as drawn in \cref{fig:nested}.
      Take a set $W = D^n_{r_1}\times [s_0/2, 3s_0/4] \subset \underline K$. 
      Let $\rho: W\to \RR_\geq 0$ be a bump function, vanishing on the boundary, with the property that there exists an $r_2<r_1$ so that for all $x,y\in D^n_{r_2}$, $\rho(x,s)=\rho(y, s)$. 
      Furthermore, assume that  $\int_{\{x\}\times [s_0/2, 3s_0/4]} \rho(x, s) ds=1$.
      Let $\alpha = \sup_{(x, s)\in \underline K} |d\rho|$.
      Let $\beta = \inf_{(x, s)\in W} |df|$. Since $f$ has no critical points in $W$, this is greater than $0$.
      To each choice of $V\subset D^n_{r_2}\times [-s_0/2, s_0/2]$ and associated interpolation $h_c:\underline V\to \RR$, we can define a function
       \[     A_c(x):= \int_{\{x\}\times [-s_0/2, s_0/2]} \left(h_c(x, s) -f(x, s) \right)ds.
     \]
      We may choose $V$ small enough so that our interpolation satisfies 
      \begin{align*}
             \sup_{x, c}|A_c(x)| < \frac{\beta}{2\alpha}&& \sup_{x}|dA_1(x)|<\frac{\beta}{2}
      \end{align*}
      Now consider the function 
      \[H_c(x, s)=\int_{\{x\}\times [-s_0, s]} h_c(x,s) - A_c(x) \rho(x, s) ds.\]
      Then $\partial_s H_1(x, s) = h_c(x, s)- A_c(x)\rho(x, s)$. 
      We have that $d(\partial_sH_1)= dh_c- d(A_c(x)\rho)$. By construction $|d(A_c(x)\rho)|< |dh_c|$ inside the region $W$, and $A_c(x)\rho$ vanishes outside of $W$. It follows that $\partial_s H_1(x, s)$ has no critical points outside of $V$.
      The derivative $\partial_sH_1(x, s)$ is Morse, agrees with $g$ in a neighborhood of the origin, where it has a single critical point. 
      Furthermore, near the boundary of $\underline K$, we have 
      \[H_c(x, s) = F(x, s) \text{ for all } (x, s) \in \underline K\setminus \left(D^n_{r_2}\times [-s_0/2, 3s_0/4]\right).\]

      Consider the Lagrangian section of $T^*\underline K$ given by $dH_1$. 
      This Lagrangian section is exactly isotopic to  $K|_{T^*\underline K}$, with exact primitive vanishing at the boundary. 
      We therefore have an exactly homotopic family of Lagrangian cobordisms
      \[K_c:= K\setminus (K|_{T^*\underline K}) \cup dH_c,\]
      where $\Crit(\pi_\RR: K_0 \to \RR)= \Crit(\pi_\RR : K_1\to \RR)$ and $K_1|_{T^*\underline K} = K^{k, n-k+1}_{loc}|_{T^*\underline K}$.

      Finally, for the bound on the Hofer norm: this is given by 
      \[\int_{0}^1 \sup_{x,s}(H_c(x,s)-F(x, s))-\inf_{x, s}(H_c(x, s)-F(x,s)) dc.\] By choosing our initial neighborhood $D^n_{r_0}$ sufficiently small (so both $F(x, s)$ and $G^{k, n-k}$ are nearly zero over the neighborhood) we make $\sup_{c,x, s}|(H_c(x,s)-F(x, s)|$ as small as desired.
\hfill{\qed}

\subsubsection{Cobordisms are concatenations of surgeries}
We now prove that every Lagrangian cobordism is exactly homotopic to the concatenation of standard surgery handles.
\begin{proof}[Proof of \cref{thm:cobordismsaresurgery}]
      Let $K: L^+\rightsquigarrow L^-$ be a Lagrangian cobordism.
      After application of \cref{claim:morsecobordism}, we obtain $K'$, a Lagrangian cobordism exactly homotopic to  $K$ with the property that $\pi_\RR: K'\to \RR$ is Morse with distinct critical values.
      Enumerate the critical points $\{q_i\}_{i=1}^l=\Crit(\pi_\RR: K'\to \RR)$.  
      By \cref{prop:goodposition}, we may furthermore assume that $K'$ is constructed so that there exists symplectic neighborhoods $U_i\subset T^*D^n\times T^*D^1$ so that $K'\cap U_i = K^{k_i,n-k_i+1}_{loc}\cap U_i$.

      For each $q_i$ take $\epsilon_i$ small enough so that the $(q_i-\epsilon_i, q_i+\epsilon_i)$ are disjoint. Take $r_i$ small enough so that the ball of radius $r_i$ centered at the critical point $q_i$ is contained within the charts $U_i$.
      $K$ is exactly homotopic to the composition 
      \[K\|_{(-\infty, q_l-\epsilon]}\circ K\|_{[q_l-\epsilon,q_l+\epsilon ]}\circ K\|_{[q_l+\epsilon, q_{l-1}+\epsilon]}\circ \cdots \circ K\|_{[q_2+\epsilon, q_1-\epsilon]} \circ K\|_{[q_1-\epsilon, q_1+\epsilon]} \circ K\|_{[q_1+\epsilon, \infty)}.\]
      By applying \cref{prop:splittingX} on $K'\|_{(q_i-\epsilon_i, q_i+\epsilon_i)}$ at the dividing hypersurface  
      \[M_i:=\left\{(x_0, \ldots, x_n)\in U_i \;\middle|\; \sum_{j=0}^n x_j^2= r_i^2\right\}\] 
      we obtain an exact homotopy 
      \[K'\|_{(q_i-\epsilon_i, q_i+\epsilon_i)}\sim K^-_i\circ \uu{K_i}_{M_i} \circ K^+_i\]
      where $K^\pm_i$ are suspensions, and $K_i^{k_i, n-k_i+1}:= \uu{K_i}_{M_i}$  are Lagrangian surgery traces.

      For each $i$, let $K_{H^i_t}= K^+_{i+1} \circ K^-_{i}$ be the suspension of an exact homotopy. 
      Then 
      \[K\sim K_{H^l_t}\circ K_l^{k_l, n-k_l+1}\circ K_{H^{l-1}_t}\circ \cdots \circ K_{H^1_t}\circ K_1^{k_1, n-k_1+1}\circ K_{H^0_t}.\]

      We finally check the Hofer norm of the exact homotopy above. At each step where we employ an exact homotopy, the operations from \cref{claim:morsecobordism,prop:goodposition,prop:tdecomposition,prop:splittingX} could be conducted in such a way to make the Hofer norm of their associated exact homotopies as small as desired. Since there are a finite number of operations being conducted, the Hofer norm of the exact homotopy between $K$ and a decomposition can be made as small as desired.
\end{proof}
\begin{conjecture}
      Suppose that $L$ is an unobstructed Lagrangian, whose bounding cochain has valuation $c$. Let $L'$ be exactly homotopic to $L$, with the Hofer norm of the exact homotopy less than $c$. Then $L'$ is unobstructed.
      \label{con:hofunobstructed}
\end{conjecture}
The conjecture is based on the following observation: for Hamiltonian isotopies, the suspension cobordism $K$ has the property that $\CF(K)$ is an $A_\infty$ mapping cocylinder between $\CF(L^+)$ and $\CF^(L^-)$, meaning that there are $A_\infty$ projection maps $\pi^\pm: \CF(K)\to \CF(L^\pm)$ and a map (defined on chains, but not a $A_\infty$-homomorphism) $\underline i^+:\CF( L^+)\to \CF(K)$ which can be extended to an $A_\infty$ homomorphism $i^+: \CF(L^+)\to \CF(K)$. The homomorphism $i^+$  is an $A_\infty$ homotopy inverse to $\pi^+$. A key point is that the lowest order portion of $i^+$ comes from the Morse continuation map, so $i^+$ has valuation zero. 

For exact homotopy, we expect that there still exist projection maps $\pi^\pm$. There are several difficulties in the construction of these map (principally, it requires a rigorous definition of the Floer theory in this setting).
In contrast to the isotopy case, the map $i^+$ will be given by at lowest order by counts flow lines (between Morse generators) and holomorphic strips (between generators associated to self-intersections). Thus, while we may be able to define a map (not an $A_\infty$ homomorphism) $i^+:\CF( L^+)\to \CF(K)$, the map may \emph{decrease} valuation. We conjecture that the decrease in valuation is bounded by the shadow of the Lagrangian cobordism. Under these circumstances, there is a version of the $A_\infty$-homotopy transfer theorem \cite{hicks2019wall} which allows $i^+$ to be extended to an $A_\infty$ homotopy inverse to $\pi^+$.

As the exact homotopies we consider for our decomposition of Lagrangian cobordisms have as small Hofer norm as desired, a corollary of the conjecture is that decomposition of Lagrangian cobordism preserves unobstructedness.

Finally, we make a remark about anti-surgery versus surgery. 
We've shown that every Lagrangian cobordism can be decomposed as a sequence of exact homotopies and Lagrangian surgery traces; in particular, the Lagrangian anti-surgery trace $(K^{k, n-k+1})^{-1}$ can be rewritten as a Lagrangian surgery and exact homotopy. 
An anti-surgery takes an embedded Lagrangian $L$ and adds a self-intersection; one can equivalently think of this as starting with an embedded Lagrangian, applying an exact homotopy to obtain a pair of self-intersection points, and then surgering away one of the self-intersections.  This is drawn in \cref{fig:antisurgerytosurgery}
\begin{figure}
      \centering
      
\begin{tikzpicture}[scale=.5, xscale=-1]

\begin{scope}[shift={(0,-5)}]
\fill[fill=gray!20]  (-5.5,4.5) rectangle (-0.5,0.5);
\draw[thick] (-5.5,3) -- (-0.5,3) (-5.5,2) -- (-0.5,2);
\end{scope}
\begin{scope}[shift={(12,-5)}]

\fill[fill=gray!20]  (0.5,4.5) rectangle (5.5,0.5);
\draw[thick] (0.5,2) .. controls (1,2) and (1,1.5) .. (1.5,1.5) .. controls (2,1.5) and (2,3.5) .. (2.5,3.5) .. controls (3,3.5) and (3,3) .. (3.5,3) .. controls (5,3) and (5,3) .. (5.5,3);
\draw[thick] (5.5,2) .. controls (5,2) and (5,2) .. (3.5,2) .. controls (3,2) and (3,1.5) .. (2.5,1.5) .. controls (2,1.5) and (2,3.5) .. (1.5,3.5) .. controls (1,3.5) and (1,3) .. (0.5,3);

\end{scope}
\begin{scope}[]
\fill[fill=gray!20]  (0.5,-0.5) rectangle (5.5,-4.5);
\draw[thick] (0.5,-2) .. controls (1,-2) and (1,-1.5) .. (1.5,-1.5) .. controls (2,-1.5) and (2,-3) .. (2.5,-3) .. controls (3,-3) and (3,-1.5) .. (3.5,-1.5) .. controls (4,-1.5) and (4,-2) .. (4.5,-2) node (v1) {};
\draw[thick] (4.5,-3) .. controls (4,-3) and (4,-3.5) .. (3.5,-3.5) .. controls (3,-3.5) and (3,-2) .. (2.5,-2) .. controls (2,-2) and (2,-3.5) .. (1.5,-3.5) .. controls (1,-3.5) and (1,-3) .. (0.5,-3);
\draw[thick] (4.5,-2) -- (5.5,-2) (4.5,-3) -- (5.5,-3);

\end{scope}

\begin{scope}[shift={(6,0)}]
\fill[fill=gray!20]  (0.5,-0.5) rectangle (5.5,-4.5);

\draw[thick] (0.5,-2) .. controls (1,-2) and (1,-1.5) .. (1.5,-1.5) .. controls (2,-1.5) and (2,-3) .. (2.5,-3) .. controls (2.75,-3) and (2.75,-2.75) .. (3,-2.75) .. controls (3.25,-2.75) and (3.25,-3.5) .. (3.5,-3.5) .. controls (4,-3.5) and (4,-3) .. (4.5,-3) .. controls (5,-3) and (5,-3) .. (5.5,-3);
\draw[thick] (0.5,-2) .. controls (1,-2) and (1,-1.5) .. (1.5,-1.5);

\draw[thick] (0.5,-3) .. controls (1,-3) and (1,-3.5) .. (1.5,-3.5) .. controls (2,-3.5) and (2,-2) .. (2.5,-2) .. controls (2.8,-2) and (2.75,-2.5) .. (3,-2.5) .. controls (3.25,-2.5) and (3.2,-1.5) .. (3.5,-1.5) .. controls (4,-1.5) and (4,-2) .. (4.5,-2) .. controls (5,-2) and (5,-2) .. (5.5,-2);

\end{scope}
\draw[->] (-2,1) .. controls (-1,2) and (1,2) .. (2,1);
\draw[->] (4,1) .. controls (5,2) and (7,2) .. (8,1);
\draw[->] (10,1) .. controls (11,2) and (13,2) .. (14,1);
\node at (0,2) {Exact Homotopy};
\node at (6,2) {Surgery};
\node at (12,2) {Exact Homotopy};
\end{tikzpicture}       \caption{Rewriting an anti-surgery as exact homotopies and surgeries.}
      \label{fig:antisurgerytosurgery}
\end{figure}
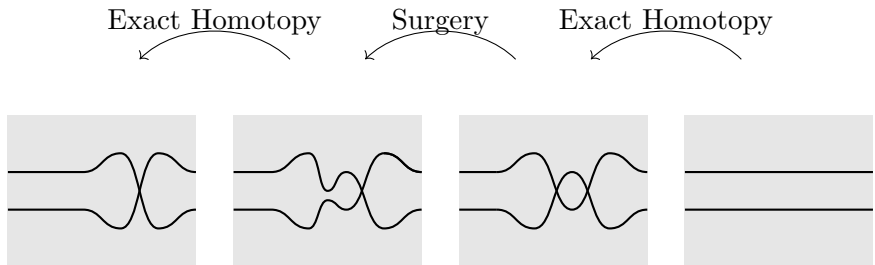
 
\section{Teardrops on Lagrangian Cobordisms}
\label{sec:teardrops}
One of the main observations about the decomposition described in \cref{thm:cobordismsaresurgery} is that the slices near a critical point of $\pi_\RR: K\to \RR$ differ:
\begin{itemize}
      \item topologically by a surgery and
      \item as immersed submanifolds of $X$ by the creation or deletion of a self-intersection.
\end{itemize}
An interpretation of the difference is that surgery trades topological chains of $L^-$ for self-intersections of $L^+$. When $K$ is an embedded oriented Lagrangian submanifold, one can show that this exchange preserves the $\ZZ/2\ZZ$ grading of the topological chains. As a consequence,
 $\chi(L^+)=\chi(K)=\chi(L^-)$. 
 
In this section we extend this statement to the immersed setting, and provide evidence that this equality can be upgraded to an isomorphism of Floer cohomology groups. In \cref{subsubsec:observations} we provide a definition for the Euler characteristic of an immersed Lagrangian submanifold with transverse self-intersections (\cref{eq:immersedEulerCharacteristic}). Using \cref{thm:cobordismsaresurgery} we then prove in \cref{prop:eulerchar} that for this definition of Euler-characteristic $\chi^{si}(L^-)=\chi^{si}(L^+)$.

The computation does not immediately extended to the Euler characteristic of a Lagrangian cobordism $K:L^+\rightsquigarrow L^-$ with immersed ends as $K$ will not have transverse self-intersections. 
We therefore require a standard form for Lagrangian cobordism with transverse self-intersections. The standard form we choose (Lagrangian cobordisms with double bottlenecks) is adopted from  \cite{mak2018dehn} who also studied the Floer cohomology of immersed Lagrangian cobordisms.
We subsequently show in \cref{prop:selfIntersectionSummary} that for Lagrangian cobordisms with double bottlenecks $(K,t^-, t^+): (L^-, H^-)\rightsquigarrow (L^+, H^+)$,  
\begin{equation}
       \chi^{bot}(K, t^-, t^+)=\chi^{si}(L^-)=\chi^{si}(L^+)
       \label{eq:selfIntersectionSummary}
\end{equation}
where $\chi^{si},\chi^{bot}$ are the Euler characteristic of an appropriate set of Floer cochains for immersed Lagrangians and Lagrangian cobordisms with double bottlenecks.

In \cref{subsec:basicimmersed} we review the construction of immersed Lagrangian Floer cohomology, and in \cref{subsec:teardrops} we provide evidence that \cref{eq:selfIntersectionSummary} can be extended to chain homotopy equivalences between the immersed Lagrangian Floer cohomologies of $K, L^-$ and $L^+$ when $K$ is a standard surgery trace. 

\subsection{Grading of Self-intersections, and an observed pairing on cochains}
\label{subsubsec:observations}
We recover this equality of Euler characteristic on the chain level for Lagrangian cobordisms with self-intersections (\cref{prop:eulerchar}). 
Suppose that $X$ has a nowhere vanishing section $\Omega$ of $\Lambda^n_\CC(T^*M, J_M)$.
We say that $\li:L\to X$ is graded if there exists a function $\theta: L\to \RR$ so that the determinant map
\begin{align*}
      \det: L\to &S^1\\
      q\mapsto &\Omega(T_qL)^{\tensor 2}/|\Omega(T_q L)|^2
\end{align*}
can be expressed as $\det(q)=e^{2\jmath\pi \theta}.$
In this setting, we can define a self-intersection corrected Euler characteristic for immersed Lagrangian submanifolds. 
Let 
\[\mathcal I^{si}(L)=\{(p\to q)\;|\;   p, q\in L , \li(p)=\li(q)\}\]
be the set of ordered self-intersections\footnote{The notation reflects our interpretation of each element $(p\to q)$ as being a short Hamiltonian chord starting at $p$ and ending at $q$.}. Note that each self-intersection of $L$ gives rise to two elements of $\mathcal I^{si}$.
The \emph{index} of a self-intersection $(p\to q)$ is defined as
\[\ind(p\to q):=n+\theta(q)-\theta(p)-2\measuredangle(\li_* T_pL, \li_* T_qL),\]
where $\measuredangle(V, W)$ is the K\"ahler angle between two Lagrangian subspaces. We particularly suggest reading the exposition in \cite{alston2019immersed} on computation of this index.
\begin{claim}[Index of Handle Self Intersections]
      \label{claim:indexOfHandle}
      Consider the parameterization of one boundary of the standard Lagrangian surgery handle obtained from restricting \cref{def:localsurgerytrace} to the positive slice:
      \[\lj^{k, n-k+1}|_{t=1}:L^{k, n-k,+}\to \CC^n.\]
      Equip $\CC^n$ with the standard holomorphic volume form.
      The index of the self-intersections  are 
      \begin{align*}
            \ind(q_-\to q_+)=n-k-1. && \ind(q_+\to q_-)= k+1.
      \end{align*}
\end{claim}
\begin{proof}
      To reduce clutter, we write $\lj$ for $\lj^{k, n-k+1}$.
      The hypersurface $x_0^2+\sum_{i=1}^n \sigma_{i, k} x_i^2=1$ in $\RR^{n+1}$ describes the slice $L^{k, n-k, +}$ as a hypersurface of the local surgery trace.
      Its tangent space at a point $\gamma(\theta)=(\sin(\theta), 0, \ldots, 0, \cos(\theta))$ is spanned by the basis
      $\{\cos(\theta)\partial_1 - \sin(\theta) \partial_0, \partial_2, \partial_3, \ldots, \partial_n\}$. 
      Let $e_1, \ldots, e_n$ be the standard basis of $T_\CC \CC^n$. 
      Then 
      \begin{align*}
            \li_*(\partial_0)=&\sum_i 2\jmath \sigma_{i,k}x_i e_i\\
            \lj_*(\partial_i)=&(1+2\jmath\sigma_{i,k} x_0)e_i \;\;\; \text{ for $i=1, \ldots, n$}
      \end{align*}
       so that the tangent space at $T_{\lj(\gamma(\theta))}L^{k, n-k, +}_{loc}$ is spanned by vectors $\vec v_i(\theta)$, where
      \begin{align*}
            \vec v_1(\theta)=&\left(\cos(\theta)+2\jmath \left(\sum_{i=1}^n\sigma_{i,k}(\cos^2(\theta)- \sin^2(\theta)\right)\right)e_1 \\
             \vec v_i(\theta)=& (1+2\jmath\sigma_{i,k} \cos(\theta))e_i \;\;\; \text{For $i=2, \ldots, n$.}
      \end{align*}
      Let $z_i(\theta)$ be coefficients so that $\vec v_i(\theta) =z_i(\theta) e_i$. 
      Then $\arg(z_i(\theta))$ is decreasing for $i\leq k$ and increasing for $i>k$. The endpoints of the $z_i$ are given by  
      \begin{align*}
            z_1(0)= 1+2\jmath && z_1(\pi)= -1+2\jmath\\
            z_i(0)= 1+2\jmath \sigma_{i, k}  && z_i(\pi)= 1-2\jmath  \sigma_{i, k}.
      \end{align*}
      To compute the index of $(q_-\to q_+)$, we complete the path $z_i^2(\theta)/|z_i^2(\theta)|$ to a loop by taking the short path, and sum the total argument swept out by each of the $z_i$. 
      For ease of computation, let $\alpha=\arctan(2)$.
      \begin{itemize}
      \item When $i=1$, the loop $z_1^2(\theta)/|z_1^2(\theta)|$ sweeps out $-2\pi - 4\alpha$ radians; the short path completion yields a contribution of $-2\pi$ to the total index of this loop.
      \item When $1< i \leq k$, the loop $z_i^2(\theta)/|z_i^2(\theta)|$ sweeps out $-4\alpha$ radians; the short path completion yields a total contribution of $0$ from this loop. 
      \item When $k< i\leq n$, the loop $z_i^2(\theta)/|z_i^2(\theta)|$ sweeps out $4\alpha$ radians; the short path completion yields a total contribution of $2\pi$ from this loop.
      \end{itemize}
      The total argument swept out is $(n-k-1)\cdot 2\pi$.
      The index of the self-intersection is 
            \[\ind(q_-\to q_+)=n-k-1.\]
      By similar computation (or using duality) we see that 
      \[\ind(q_+\to q_-)= k+1.\]
      \label{example:indexcomputation}
\end{proof}
Let $L$ be a compact graded Lagrangian submanifold, and $f:L\to \RR$ be a Morse function.
The set of Floer generators is 
\begin{equation}
      \mathcal I(L)=\Crit(f)\cup \mathcal I^{si}(L).
      \label{eq:immersedEulerCharacteristic}
\end{equation}
For each $x\in \Crit(f)$, let $\ind(x)$ be the Morse index. 
Define the self-intersection Euler characteristic to be $\chi^{si}(L):=\sum_{x\in \mathcal I(L)} (-1)^{\ind(x)}$. 
\begin{prop}
      Let $K: L^+\rightsquigarrow L^-$ be a Lagrangian cobordism.
      Then $\chi^{si}(L^-)=\chi^{si}(L^+)$.
      \label{prop:eulerchar}
\end{prop}
As was pointed out to me by Ivan Smith, this also follows in the case that $K$ is embedded by the much simpler argument that the Euler characteristic is the signed self-intersection, and noting that we can choose a Hamiltonian push-off so the intersections of $K\cap \phi(K)$ are in index preserving bijections with intersections $L^\pm \cap \phi(L^\pm)$. 
Nevertheless, we give proof using decomposition as this will motivate \cref{subsec:teardrops}. 
\begin{proof}
      As each exact homotopy preserves $\chi^{si}(L^-)$, we need only check the case that  $K^{k, n-k+1}: L^+\rightsquigarrow L^-$ is a surgery trace. 
      Choose Morse functions $f^\pm: L^\pm\to \RR$ which in the local model of the surgery neck given by \cref{def:localsurgerytrace} agree with the coordinate $x_n$.
      The critical points of $f^\pm$ agree outside of the surgery region. Inside the surgery region, we have
      \begin{align*}
            \Crit(f^+)\cap L^{k, n-k,+}_{loc}     =& \Crit(x_n:L^{k, n-k,+}_{loc}\to \RR)&\\
            \Crit(f^-)\cap  L^{k+1, n-k-1, -}_{loc} =& \Crit(x_n: L^{k+1, n-k-1, -}_{loc}\to \RR).
      \end{align*}
      Recall that in the coordinates from \cref{def:localsurgerytrace} these Lagrangian submanifolds are parameterized by domains in $\RR^{n+1}$  cut out by the equations
      \begin{align*}
            L^{k, n-k,+}_{loc}=&\left\{(x_0, \ldots, x_n) \;|\; x_0^2+\sum_{i=1}^n \sigma_{i,k} x_i^2=1\right\}\\
             L^{k+1, n-k-1, -}_{loc}=&\left\{(x_0, \ldots, x_n) \;|\; x_0^2+\sum_{i=1}^n \sigma_{i,k} x_i^2=-1\right\}
      \end{align*}
      We compute the critical points of $f$ restricted to the level sets of $\pi_{\RR}\circ \lj$ using the method of Lagrange multipliers. 
      \begin{align*}
      \nabla f  =&  \langle 2x_0, 2\sigma_{1,k}x_1, \ldots, 2\sigma_{n, k} x_n\rangle \\
      \grad \pi_\RR\circ \lj =& \langle 0, 0, \cdots 0, 1\rangle
      \end{align*}
      So the only critical points in the surgery region occur when $x_0, \ldots, x_{n-1}=0$. From this we obtain the following cases:
      \begin{itemize}
            \item If $k=-1$, then $f^+$ restricted to $L^{-1, n+1, +}_{loc}$ has no critical points (as it is empty) and $L^{0, n, -}$ has critical points of index $0$ and $n$, which we call $e, x$.
            \item If $-1<k<n$, then $f^+$ restricted to $L^{k, n-k,+}_{loc}$ has no critical points, and $f^-$ restricted to $L^{k+1, n-k-1, -}_{loc}$ has critical points of index $k+1$ and $n-k-1$. We call these critical points $x_\pm$.
            \item If $k=n$, then $f^+$ restricted to $L^{n, 0,+}_{loc}$ has critical points of index $0, n$ which we call $e, x$. The function $f^-$ restricted to $L^{k+1, n-k-1, -}_{loc}$ has no critical points (as it is empty).
      \end{itemize}
      The differences between  $\mathcal I(L^+), \mathcal I(L^-)$ are listed in \cref{tab:indexdifferences}.
\begin{table}
            \begin{subtable}[t]{\linewidth}
                  \centering
      \begin{tabular}[t]{c|ccc}
            Index & $\mathcal I^{si}(L^+) $ & $\Crit(f^+)$ & $\Crit(f^-)$\\\hline
            $k+1$ & $(q_+\to q_-)$ & & $x_+$\\
            $n-k-1$ & $(q_-\to q_+)$ & & $x_-$ \\
      \end{tabular} 
      \caption{Floer cochains of  $\mathcal I(L^+)$ and $\mathcal I(L^-)$ for a Lagrangian surgery trace $K^{k, n-k+1}: L^+ \rightsquigarrow L^-$, when $-1<k<n$}
\end{subtable}\vspace{.5cm}
\begin{subtable}[t]{.45\linewidth}
      \centering
      \begin{tabular}[t]{c|cc}
            Index & $\mathcal I^{si}(L^+) $ & $\Crit(f^+)$  \\ \hline
            $-1$ & $(q_-\to q_+)$\\
            $0$ &  & $e$\\
            $n$ && $ x$ \\
            $n+1$ & $(q_+\to q_-)$\\
      \end{tabular}
      \caption{In the degenerate case $K^{n, 1}$, $\mathcal I(L^-)$ is empty.  }
\end{subtable}
\begin{subtable}[t]{.45\linewidth}
      \centering
      \begin{tabular}[t]{c|cc}
            Index & $\mathcal I^{si}(L^-) $ & $\Crit(f^-)$  \\ \hline
            $-1$ & $(q_-\to q_+)$\\
            $0$ &  & $e$\\
            $n$ && $ x$ \\
            $n+1$ & $(q_+\to q_-)$\\
      \end{tabular}
      \caption{In the degenerate case $K^{-1, n+2}$, $\mathcal I(L^+)$ is empty. }
\end{subtable}
\caption{Chain level differences in Floer generators before and after surgery.}
\label{tab:indexdifferences}
\end{table}
      From the values listed in \cref{tab:indexdifferences} it follows that $\chi^{si}(L^-)=\chi^{si}(L^+)$.
\end{proof}
This computation leads to the following question:  can we extend \cref{prop:eulerchar} to an equivalence of Floer theory. 
 \subsection{Doubled Bottlenecks}
\label{subsec:doublebottleneck}
Problematically, the decomposition given by \cref{thm:cobordismsaresurgery} is not very useful for understanding Floer cohomology for Lagrangian cobordisms with immersed ends, as such Lagrangian cobordisms will not have transverse self-intersections.
This is because the standard definition of Lagrangian cobordisms (\cref{def:cobordism}) does not allow us to easily work with immersed Lagrangian ends. 
Rather, \cite{mak2018dehn} gives a definition for a \emph{bottlenecked Lagrangian cobordism} which gives a method for concatenating Lagrangian cobordisms with immersed ends in a way that preserves transversality of self-intersections.
\subsubsection{Bottlenecked Lagrangian Cobordisms}
\label{subsubsec:bottleneckedLagrangian}
\begin{definition}
    Let $\li_0:L_0\to X$ be an immersed Lagrangian with transverse self-intersections. A bottleneck datum for $L$ is an extension of the immersion to an exact homotopy $\li_t:L_0\times I\to X$ with primitive $H_t: L_0\times I\to \RR$ satisfying the following conditions:
    \begin{itemize} 
        \item\emph{Bottleneck:}  $H_0=0$ and  there exists a bound $C\in \RR$ so that  $|\frac{dH_t}{dt}|\leq C$ everywhere.
        \item\emph{Embedded away from $0$:} If $\li_t(q_0)=\li_t(q_1)$, then either $q_0=q_1$ or $t=0$.
    \end{itemize}
    For simplicity of notation\footnote{The primitive of an exact homotopy doesn't determine the exact homotopy $\lj$; however, many properties of the bottleneck are determined by $H_t$.}, we will denote the datum of a bottleneck by  $(L, H_t)$. 
\end{definition}
We say that a Lagrangian cobordism $K\subset X\times\CC$ is bottlenecked at time $t_0$ if $K|_{[t_0-\epsilon, t_0+\epsilon]}$ is the suspension of a bottleneck datum.
The image of the Lagrangian cobordism under $\pi_\CC$ has a pinched profile (see \cref{fig:cobordismsWithBottlenecks}).
By application of the open mapping principle, the pinch point prevents pseudoholomorphic disks with boundary on the Lagrangian cobordism from passing from one side of the Lagrangian cobordism to the other (hence the name bottleneck).
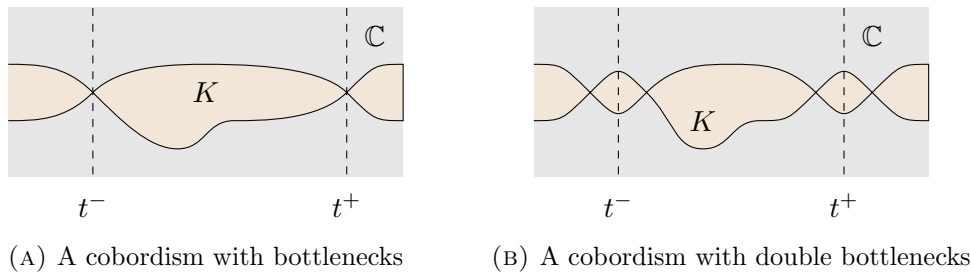
\begin{figure}
    \centering
    \begin{subfigure}{.45\linewidth}
        \centering
        \begin{tikzpicture}[scale=.75]
    \fill[gray!20]  (-3,1) rectangle (4,-2);
    \draw[fill=brown!20] (-3,-1) .. controls (-2.5,-1) and (-2,-1) .. (-1.5,-0.5) .. controls (-1,-1) and (-0.5,-1.5) .. (0,-1.5) .. controls (0.5,-1.5) and (0.5,-1) .. (1,-1) .. controls (1.5,-1) and (2.5,-1) .. (3,-0.5) .. controls (3.5,-1) and (3.5,-1) .. (4,-1) .. controls (4,-0.5) and (4,0) .. (4,0) .. controls (3.5,0) and (3.5,0) .. (3,-0.5) .. controls (2.5,0) and (1,0) .. (0.5,0) .. controls (0,0) and (-1,0) .. (-1.5,-0.5) .. controls (-2,0) and (-2.5,0) .. (-3,0);
    \draw[dashed] (-1.5,1) -- (-1.5,-2) (3,1) -- (3,-2);
    \node at (-1.5,-2.5) {$t^-$};
    \node at (3,-2.5) {$t^+$};
    \node at (0.5,-0.5) {$K$};
    \node at (3.5,0.5) {$\mathbb C$};
    \end{tikzpicture}         \caption{A cobordism with bottlenecks}
    \end{subfigure}
    \begin{subfigure}{.45\linewidth}
        \centering
        \begin{tikzpicture}[scale=.75]
\fill[gray!20]  (-3,1) rectangle (4,-2);
\draw[fill=brown!20] (-3,-1) .. controls (-2.5,-1) and (-2.5,-1) .. (-2,-0.5) .. controls (-1.5,-1) and (-1.5,-1) .. (-1,-0.5) .. controls (-0.5,-1) and (-0.5,-1.5) .. (0,-1.5) .. controls (0.5,-1.5) and (0.5,-1) .. (1,-1) .. controls (1.5,-1) and (1.5,-1) .. (2,-0.5) .. controls (2.5,-1) and (2.5,-1) .. (3,-0.5) .. controls (3.5,-1) and (3.5,-1) .. (4,-1) .. controls (4,-0.5) and (4,-0.5) .. (4,0) .. controls (3.5,0) and (3.5,0) .. (3,-0.5) .. controls (2.5,0) and (2.5,0) .. (2,-0.5) .. controls (1.5,0) and (1,0) .. (0.5,0) .. controls (0,0) and (-0.5,0) .. (-1,-0.5) .. controls (-1.5,0) and (-1.5,0) .. (-2,-0.5) .. controls (-2.5,0) and (-2.5,0) .. (-3,0);

\draw[dashed] (-1.5,1) -- (-1.5,-2) (2.5,1) -- (2.5,-2);
\node at (-1.5,-2.5) {$t^-$};
\node at (2.5,-2.5) {$t^+$};
\node at (0,-1) {$K$};
\node at (3,0.5) {$\mathbb C$};\end{tikzpicture}         \caption{A cobordism with double bottlenecks}
    \end{subfigure}
    \caption{Lagrangian cobordisms with bottlenecks or double bottlenecks allow us to discuss Lagrangian cobordisms between immersed Lagrangian submanifolds. }
    \label{fig:cobordismsWithBottlenecks}
\end{figure}
\begin{example}[Whitney Sphere]
The first interesting example of a bottleneck comes from the Whitney $n$-sphere $L^{n, 0}_A\subset \CC^n$.
We treat $L^{n, 0}_A$ as a Lagrangian cobordism inside of $\CC^{n-1}\times \CC$. 
The shadow projection $\pi_\CC: L^{n, 0}_A\to \CC$ is drawn in \cref{fig:whitneysphere}.
The bottleneck on the Lagrangian cobordism occurs at $t=0$, with bottleneck datum corresponding to an exact homotopy of the Whitney $(n-1)$-sphere $L^{n-1, 0}_A\subset \CC^{n-1}$. 
This exact homotopy is parameterized by
\[
    \li^{n-1,0}_{A(t)}:S^{n-1}_{r(t)}\to \CC^{n-1}
\] 
where $A(t)=\frac{3}{4}(1-t^2)^\frac{3}{2}$, and $r(t)=\sqrt{1-t^2}$.
The primitive for this exact homotopy is $H_t=2x_0t$.
Each $L^{n, 0}_{A(t)}$ contain a pair of distinguished points, $q_\pm=(\pm r(t), 0, \ldots, 0)$, which correspond to the self-intersection of the Whitney $(n-1)$ sphere. 
Note that $q_+$ is the maximum of $H_t$ on each slice, and $q_-$ is the minimum of $H_t$ on each slice. 
\label{exam:whitneybottle}
\end{example}
Given an immersed Lagrangian $\li: L\to  X$ with transverse self-intersections, there exists a standard way to produce a bottleneck (which we call a \emph{standard bottleneck}).  
Pick a local Weinstein neighborhood $\phi: B^*_r L\to X$. 
Let $h: L\to X$ be a function such that $|dh|<r$, and that $h(p)\neq h(q)$ whenever $(p\to q)\in \mathcal I^{si}(L)$.
Let $H_t=t\cdot h$.
Then the Lagrangian cobordism submanifold parameterized by 
\begin{align*}
    L\times (-\epsilon, \epsilon)\to &X\times T^*(-\epsilon, \epsilon)\\
    (q, t) \mapsto& \left(\phi\left(d\left(\frac{t^2}{2}\cdot h\right)_q\right), t+\jmath H_t(q)\right)
\end{align*}
is an example of a Lagrangian cobordism with a bottleneck.
At each self-intersection $(q_0\to q_1)\in \mathcal I^{si}(L)$, it will either be the case that $h(q_0)< h(q_1)$ or $h(q_1)< h(q_0)$.

\begin{definition}
Let $\li_t:L_0\times I\to X$ with primitive $H_t: L_0\times I\to \RR$ be a bottleneck datum. We say that $(q_0\to q_1)\in \mathcal I^{si}(L_0)$ has a maximum grading in the base from the bottleneck datum if $\frac{dH_t}{dt}(q_0, 0)> \frac{dH_t}{dt}(q_1, 0)$ ; otherwise, we say that this generator receives an minimum grading in the base from the bottleneck.
\end{definition}

\begin{remark}
    Our convention for maximal/minimal grading from the base is likely related to the convention of positive/negative perturbations chosen in \cite[Remark 3.2.1]{biran2020lagrangian}.
\end{remark}

\begin{example}[Whitney Sphere, revisited]
    For the bottleneck $\li^{n-1, 0}_{A(t)}: S^{n-1}_{r(t)}\to \CC^{n-1}$ constructed in \cref{exam:whitneybottle} , 
    \[\left.\frac{dH_t}{dt}(q_\pm)\right|_{t=0}=\pm 1\]
    so that $(q_-\to q_+)$ has a maximum grading in the base.

    We can construct another bottleneck datum so that $(q_-\to q_+)$ has a minimum grading in the base.
    As $H^1(S^{n-1})=0$, all Lagrangian homotopies are exact homotopies. 
    Consider the homotopy $\li^{n-1,0}_{s(t)}:L^{n-1,0}_{B(t)}\to \CC^{n-1}$, where $B(t)=1+t^2$, and $t\in [-1, 1]$.
    This bottleneck is the exact homotopy of which first decreases the radius of the Whitney sphere, then increases the radius of the Whitney sphere.
    For this bottleneck, $(q_-\to q_+)$ inherits a minimum grading from the base.

    While both $\li^{n-1,0}_{A(t)}, \li^{n-1,0}_{B(t)}$ provide bottlenecks for the Whitney sphere, they are really quite different as Lagrangians in $\CC^{n-1}\times \CC$. The first  bottleneck can be completed to a null-cobordism by simply adding in two caps (yielding the Whitney $n$-sphere in $\CC^n$); the second bottleneck cannot be closed off without adding in either a handle or another self-intersection.
    \label{exam:whitneyrevisited}
\end{example}
\begin{remark}
    Although not relevant to the discussion of bottlenecks, \cref{exam:whitneyrevisited} gives us an opportunity to address the discussion at the end of \cite[Section 3.4]{haug2015lagrangian} related to Whitney degenerations and Lagrangian surgery.
    The question Haug asks is: Does containing a Whitney isotropic sphere suffice for implanting a surgery model?
    Haug shows that this is not sufficient condition. 
    The specific example considered is the Whitney sphere $L^{2, 0,+}\subset \CC^2$, which contains a $1$-Whitney isotropic $L^{1, 0,+}\subset L^{2, 0, +}\subset \CC^2$.
    If there was a Lagrangian surgery trace $K^{1, 2}: L^{2, 0, +}\rightsquigarrow L^-$ which collapsed the 1-Whitney isotropic, then $L^-$ would have the topology of an embedded pair of spheres. Since no such Lagrangian submanifold exists in $\CC^2$, we conclude that possessing a $1$-Whitney isotropic does not suffice for implanting a surgery handle.

    Upon a closer examination, we see that the Whitney 1-isotropic has a small normal neighborhood $L^{1, 0,+}\times I\subset L^{2, 0, +}\subset \CC\times \CC$ which gives it the structure of a Lagrangian bottleneck. 
    This is the bottleneck $\li^{n-1, 0}_{A(t)}$ described above.

    Consider instead a Lagrangian submanifold $L^+$ which contains a Whitney $k$-isotropic $L^{k, 0, +}\times I\subset L$ with a neighborhood giving it the structure of the $\li^{n-1, 0}_{B(t)}$ bottleneck. 
    Then there exists a Lagrangian surgery trace $K^{k, n-k}: L^+\rightsquigarrow L^-$.
    This can be immediately observed for instance in \cref{fig:handle12} --- the right hand side is exactly isotopic to $\li^{1, 0}_{B(t)}$ (with the cobordism parameter in the vertical direction).
    \label{rem:surgeryconfiguration}
\end{remark} 
As \cref{exam:whitneyrevisited} shows, when one turns an immersed Lagrangian submanifold $L_0$ into a Lagrangian cobordism with a bottleneck $K$ at $t=0$, the self-intersections of $K$ with cobordism parameter $0$ are in bijection with the self-intersections of $L_0$. The gradings of the self-intersections of $K$ differ from the gradings of the self-intersections of $L_0$. When $(q_0\to q_1)\in \mathcal I^{si}(L^0)$ has a maximum grading in the base, the gradings of self-intersections on $K$ are:
\begin{align*}
    \ind((q_0,0)\to (q_1, 0))=\ind(q_0\to q_1)&& \ind((q_1,0)\to (q, 0))=\ind(q_1\to q_0)+1.
\end{align*}
From a Floer-theoretic viewpoint this is problematic, as the underlying philosophy of \cref{thm:birancornea} is that the Floer cohomology of a Lagrangian cobordism should agree with the Floer cohomology of the slice of the cobordism, and this mismatch in index shows that these two groups cannot be the same by Euler characteristic considerations. 
While bottlenecks give us a way to relate the immersed points of $K$ with the immersed points of the slice at the bottleneck, the self-intersections of $K$ with minimum grading from the base will receive the wrong grading. An additional problem with self-intersections of $K$ with minimum grading from the base is that one cannot necessarily obtain compactness for moduli spaces of holomorphic teardrops with output on a self-intersection with minimum grading in the base. 
We therefore use \emph{double bottlenecks} instead.
\begin{definition}
    Let $\li: L\to X$ be an immersion with transverse self-intersections, and $H: L\to \RR$, $\phi:B^*_r L\to X$ as in the construction of a standard bottleneck. 
    Furthermore, assume that $dH=0$ at each self-intersection point.
    Let $\rho(t)= t(t+\epsilon)(t-\epsilon)$, and $\tilde H_t(q):=\rho'(t) H(q)$.
    A standard double bottleneck datum is the exact homotopy:
    \begin{align*}
        \li_t:    L\times (-\epsilon, \epsilon)\to &X\\
        (q, t) \mapsto& \phi(d(\tilde H_t)_q).
    \end{align*}
    with the property that $\pi_X \li_t(q_0)=\pi_X \li_t(q_1)$ if and only if $\li(q_0)=\li(q_1)$ and $t=\pm \epsilon /\sqrt{3}$ (these are the critical points of $\rho$).

    We say that a Lagrangian cobordism $K\subset X\times\CC$ has a double bottleneck at time $t_0$ if $K|_{[t_0-\epsilon, t_0+\epsilon]}$ is the suspension of a standard double bottleneck datum.

    If $K$ is a Lagrangian cobordism with standard double bottleneck datum at times $t^-, t^+$ given by the data $H^\pm: L^\pm\to \RR, \phi^\pm: B^*_rL^\pm\to X$, then we will write 
    \[(K, t^+, t^-):(L^+, H^+)\rightsquigarrow (L^-, H^-)\]
    and say that $(K, t^+, t^-)$ is a Lagrangian cobordism between $L^+, L^-$ with double bottlenecks determined by $H^+, H^-$. 
    \label{def:doublebottleneck}
\end{definition}
Observe that in the setting where $K:L^+\rightsquigarrow L^-$ is a Lagrangian cobordism, $L^\pm$ are embedded, and $t^\pm$ are chosen as in \cref{def:cobordism}, then  $(K, t^+, t^-):(L^+, 0)\rightsquigarrow (L^-,0)$ is a Lagrangian cobordism with double bottlenecks.

Each self-intersection $(q_0\to q_1)\in \mathcal I^{si}(L)$ corresponds to two self-intersections $((q_0, \pm \epsilon/\sqrt 3)\to (q_1, \pm\epsilon /\sqrt 3))$ in $\mathcal I^{si}(L, H_t)$; if $((q_0,  \epsilon/\sqrt 3)\to (q_1, \epsilon /\sqrt 3))$ has a maximal grading from the base, then $((q_0, - \epsilon/\sqrt 3)\to (q_1, -\epsilon /\sqrt 3))$ has a minimal grading from the base (and vice versa).
To each immersed point $(q_0\to q_1)\in \mathcal I^{si}(L)$, we can associate a value:
\[A_{(q_0\to q_1)}:= \int_{\epsilon/\sqrt{3}}^{\epsilon/\sqrt{3}} \left(\tilde H_t(q_1)-\tilde H_t(q_0)\right) dt,\]
Finally, we observe that for each $(q_0\to q_1)$, the curves $t+\jmath \tilde H_t(q_0)$ and $t+\jmath \tilde H_t(q_1)$ bound a strip in $u:\RR\times I\to \CC$ whose area is $A_{(q_0\to q_1)}$.
Since $q_0, q_1$ are critical points of $h$, they are fixed by the homotopy and  we obtain a holomorphic strip $\{\li_0(q_0)\}\times u: \RR\times I\to X\times \CC$ with boundary on the double bottleneck.

Given Lagrangian cobordisms with double bottlenecks 
\begin{align*}
    (K^{+0}, t^+, t^0):(L^+, H^+)\rightsquigarrow (L^0, H^0)&&
    (K^{0-}, t^0, t^-):(L^0, H^0)\rightsquigarrow (L^-, H^-)
\end{align*}
there exists a composition $(K^{0-}\circ K^{+0},t^+, t^-):(L^+, H^+)\rightsquigarrow (L^-, H^-)$. The composition is covered by the charts $K^{+0}|_{>t^0-\epsilon}, K^{0-}|_{<t^0+\epsilon}$ which overlap over the double bottleneck defined by $L^0, H^0$ and $\phi^0$. In the setting where $H^+, H^0, H^-=0$ (so that $L^+, L^0, L^-$ are embedded) this agrees with the usual definition of composition of Lagrangian cobordisms. 
\subsubsection{$\chi^{si}$ for Lagrangian cobordisms with double bottlenecks}
\label{subsubsec:eulerCharForDoubledBottlenecks}
For immersed compact Lagrangian submanifolds $L$ with transverse self-intersections, the Floer cochains are defined in \cref{eq:immersedEulerCharacteristic} to be critical points of an auxiliary Morse function $f: L\to \RR$ or ordered pairs of points in $L$ whose image in $X$ agree. From this data we defined a self-intersection Euler characteristic. When defining the self-intersection Euler characteristic for Lagrangian cobordisms, we must state which auxiliary Morse functions are admissible, and how to count the self-intersection points at the double bottlenecks.
At a minimum, the self-intersection Euler characteristic that we define for Lagrangian cobordisms with double bottlenecks should satisfy the relation:
\begin{equation}
    \chi^{bot}((K^{0-}\circ K^{+0},t^+, t^-))=\chi^{bot}(K^{+0}, t^+, t^0)+ \chi^{bot}(K^{0-}, t^0, t^-) - \chi^{si}(L^0)
    \label{eq:compositionAndEulerChar}
\end{equation}
We first handle the issue of the auxiliary Morse function. Pick $f^\pm : L^\pm \to \RR$ Morse functions. 
For the Lagrangian cobordism $(K, t^+, t^-)$  we take an admissible Morse function (adapted from \cite[Definition 2.1.3]{hicks2019wall}) which is a Morse function $f:K\to \RR$ satisfying:
\begin{itemize}
\item The Morse flow restricted to the fibers above real coordinates $t^-$ and $t^+$ are determined by $f^\pm$,
    \begin{align*}
        (\pi_\RR)_*\grad f|_{\pi_\RR^{-1}(t^-) }=0  && (\pi_\RR)_*\grad f|_{\pi_\RR^{-1}(t^-+)}=0 \\
        (\pi_X)_*\grad f|_{\pi_\RR^{-1}(t^-) }=\grad f^-  && (\pi_\RR)_*\grad f|_{\pi_\RR^{-1}(t^+)}=\grad f^+
    \end{align*}
        \item The gradient points outward in the sense that 
        \begin{align*}
             dt(-\grad f)|_{\pi_\RR(q)<t^-}<0& &dt(-\grad f)|_{t^- <\pi_\RR(q) < t^-+\epsilon}>0\\
             dt(-\grad f)|_{\pi_\RR(q)>t^+}>0& &dt(-\grad f)|_{t^+-\epsilon <\pi_\RR(q) < t^-+}<0.
        \end{align*}
\end{itemize}
For such a choice of Morse function, $\Crit(f^+), \Crit(f^-)$ naturally are subsets of $\Crit(f)$, and all critical points of $f$ have cobordism parameter between $t^-$ and $t^+$ (inclusive). For embedded Lagrangian cobordisms, the Euler characteristic defined using the cochains of an admissible Morse functions satisfies \cref{eq:compositionAndEulerChar}.

This leaves us with handling the self-intersections of the Lagrangian cobordism with double bottlenecks $(K, t^+, t^-)$. 
From the design of the double bottlenecks, we see that there are inclusions $\mathcal I^{si}(L^\pm)\into \mathcal I^{si}(K)$ sending each intersection to the corresponding intersection in the double bottleneck with maximum grading from the base. This inclusion preserves degree.
Unfortunately, the double bottlenecks contain additional intersections not corresponding to elements of $\mathcal I^{si}(L^\pm)$ coming from those intersections with minimal degree in the base. We must judiciously throw out some of these intersections. 
\begin{definition}
    Let $(K, t^+, t^-):(L^+, H^+)\rightsquigarrow (L^-, H^-)$ be a Lagrangian cobordism with double bottlenecks, parameterized by $\lj:K\to X\times \CC$. The bottlenecked Floer generators of $(K, t^+, t^-)$ are 
    \begin{align*}
        \mathcal I^{bot}&(K, t^+, t^-):=\Crit(f)\\
        \sqcup&\{(q^0\to q^1)\in \mathcal I^{si}\;|\; t^-\leq\pi_\RR\circ \lj(q^0)\leq t^+\}\\
        \sqcup&\{(q^0\to q^1)\in \mathcal I^{si}\;|\; \pi_\RR\circ \lj(q^0)<t^-, \text{ and $(q^0\to q^1)$ has max grad. from base}\}\\
        \sqcup&\{(q^0\to q^1)\in \mathcal I^{si}\;|\; t^+<\pi_\RR\circ \lj(q^0), \text{ and $(q^0\to q^1)$ has max grad. from base}\}.
    \end{align*}
\end{definition}
For $(K, t^+, t^-):(L^+, H^+)\rightsquigarrow (L^-, H^-)$ a Lagrangian cobordism with double bottlenecks, we define 
\[\chi^{bot}(K, t^+, t^-):= \sum_{x\in \mathcal I^{bot}(K, t^+, t^-)} (-1)^{\ind(x)}.\]
\begin{claim}
    $\chi^{bot}(K, t^+, t^-)$ satisfies \cref{eq:compositionAndEulerChar}.
\end{claim}
\begin{proof}
    From the definition $\chi^{bot}(K, t^+, t^-)$ we see that 
    \begin{align*}
        \mathcal I^{bot}&(K^{+-}, t^+, t^-)=\left(\mathcal I^{bot}(K^{0-}, t^0, t^-)\sqcup \left(\mathcal I^{bot}(K^{+0}, t^+, t^0)\setminus \Crit(f^0)\right)\right)\\        
        \sqcup&\{(q^0\to q^1)\in \mathcal I^{si}(K^{0-})\;|\; \pi_\RR\circ \lj(q^0)<t^-, \text{ and $(q^0\to q^1)$ has \emph{min} grad. from base}\}\\
        \sqcup&\{(q^0\to q^1)\in \mathcal I^{si}(K^{+0})\;|\; t^+<\pi_\RR\circ \lj(q^0), \text{ and $(q^0\to q^1)$ has \emph{min} grad. from base}\}.
    \end{align*}
    There is a bijection between $\mathcal I^{si}(L^0)$ and the union of the latter 2 terms, however this bijection increases the index of the critical points by 1. \Cref{eq:compositionAndEulerChar} immediately follows.
\end{proof}

\subsubsection{Decomposition of Lagrangian cobordism via double bottlenecks}
The methods used in \cref{prop:tdecomposition} can similarly be used to give a decomposition of a Lagrangian submanifold $K\subset X$ into bottlenecked Lagrangian submanifolds. 
\begin{prop}
    Let $K\subset X\times \CC$ be a Lagrangian submanifold, and $t^-, t^+$ regular values of the projection $\pi_\RR: K\to \RR$. Suppose $H_t: L^\pm\times[t^\pm-\epsilon, t^\pm+\epsilon]\to \RR$ generates an exact homotopy whose suspension is $K|_{[t^\pm-\epsilon, t^\pm+\epsilon]}$.
    Furthermore, suppose that over this region, $H_t(q_1)\neq H_t(q_2)$ for each double point $(q_1\to q_2)$ of the slice $K|_t$.
    Then there exist choices of Lagrangian bottleneck data $(L^\pm, H^\pm)$ and a Lagrangian cobordism with double bottleneck:
    \[(K, t^+, t^-):(L^+, H^+)\rightsquigarrow (L^-, H^-).\]
    \label{prop:doubleBottleneck}
\end{prop}

While the standard surgery handle does not have transverse self-intersections, there is a geometrically pleasing construction of a bottlenecked Lagrangian surgery trace.
\begin{figure}
    \centering
    \includegraphics{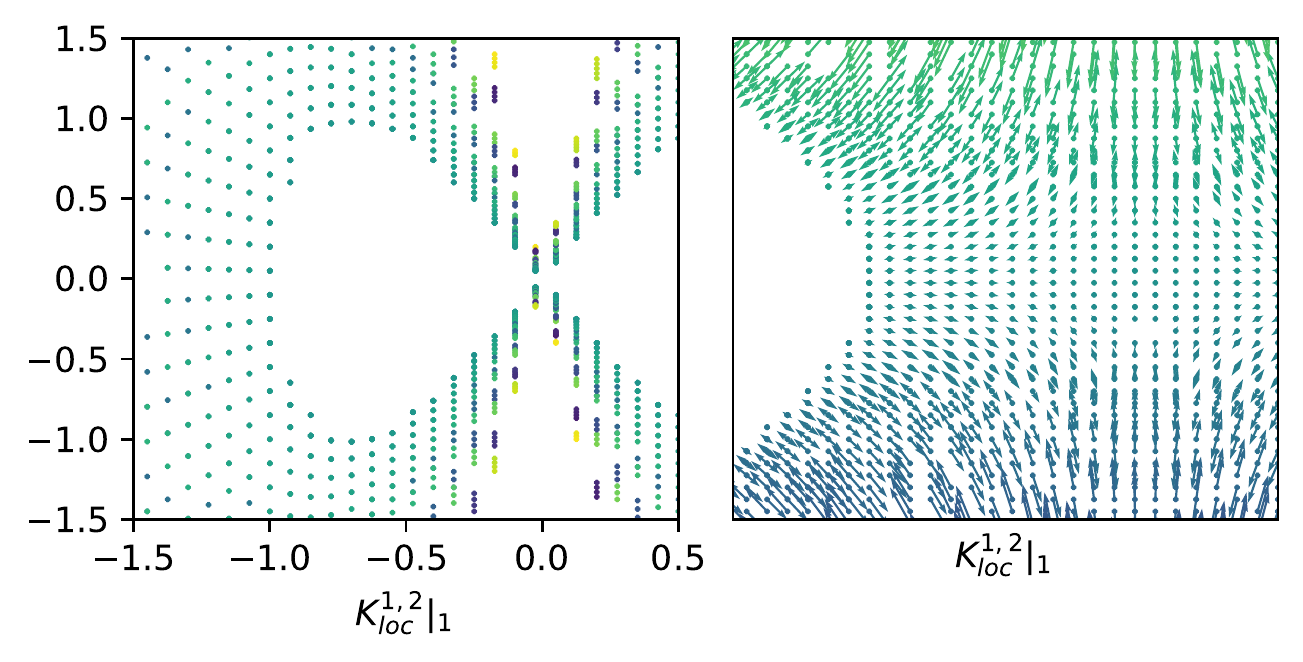}
    \caption{Left: The shadow of $K_{1,2}|_1$ as projected to the first complex coordinate, with bottleneck at the origin. Right: The Lagrangian $K_{1,2}|_1$ drawn as a section of $T^*\RR^2$.}
    \label{fig:bottleneckedhandle}
\end{figure}
Consider $L^{k+1, n-k,+}\subset \CC^{n+1}$.We now will now treat $L^{k+1, n-k,+}$ as if it were a Lagrangian cobordism, with first coordinate of $\CC^{n+1}$ serving as the cobordism coordinate. We will therefore write $L^{k+1, n-k,+}\subset \CC^n\times \CC_1$.
From this viewpoint, $L^{k+1, n-k,+}\subset \CC^{n}\times \CC_1$ can also be considered as a Lagrangian surgery trace, as 
$\pi_{\RR_1}:K^{k+1, n-k+1,+}|_{\pi_{\RR_1}<1/2}\to \RR$ has a single critical point.
Furthermore, this Lagrangian has a single self-intersection, forming a bottleneck under the $\CC_1$ projection: see the shadow projection drawn in \cref{fig:bottleneckedhandle}.
By  applying \cref{prop:doubleBottleneck} to $K^{k+1, n-k+1,+}|_{\pi_{\RR_1}<1/2}$,  we obtain a double bottlenecked Lagrangian surgery trace. 
\begin{definition}
    We call the resulting Lagrangian cobordism with double bottlenecks $K^{k, n-k+1}_{A, B}\subset \CC^{n+1}$, where the quantity $A$ measures the symplectic area of the class of the teardrop on $K^{k, n-k+1}_{A, B}$, while $B$ is the area of holomorphic strip associated with a double bottleneck. 
    \label{def:doubleBottleneckSurgeryTrace}
\end{definition} 
$K^{k, n-k+1}_{A, B}$ has 2 self-intersections and a single critical point of the height function.
By following the method of proof in \cref{thm:cobordismsaresurgery}, we obtain a doubled-bottleneck decomposition of Lagrangian cobordisms.
\begin{corollary}
    Let $(K,t^-, t^+): (L^+, H_t^+)\rightsquigarrow (L^-, H_t^-)$ be a Lagrangian cobordism with double bottlenecks.
    Then there a sequence of Lagrangian cobordisms with double bottlenecks 
    \begin{align*} 
          (K_{H^i_t}, t_{i+1}^-, t_i^+): &(L_{i+1}^-,h_{i+1,t}^-)\rightsquigarrow( L_{i}^+,h_{i,t}^+) \text{For $i\in \{0, \ldots, j\}$}\\
          (K_{ A_i, B_i}^{k_i, n-k_i+1},t_i^+,t_i^-):&( L_i^+,h_{i,t}^+)\rightsquigarrow (L_i^- ,h_{i,t}^-)\text{ For $i\in \{1, \ldots j\}$}
    \end{align*}
    which satisfy the following properties:
    \begin{itemize}
          \item $L_{j+1}^-=L^+$ and $L_0^+=L^-$
          \item Each $K_{A_i, B_i}^{k_i, n-k_i+1}$ is a Lagrangian surgery trace with double bottlenecks;
          \item Each $K_{H^i_t}$ is the suspension of an exact homotopy and;
          \item There is an exact homotopy between  
          \[K\sim K_{H^j_t}\circ K_{A_j, B_j}^{k_j, n-k_j+1}\circ K_{H^{j-1}_t}\circ \cdots \circ K_{H^1_t}\circ K_{A_1, B_1}^{k_1, n-k_1+1}\circ K_{H^0_t}\]
          whose Hofer norm can be made as small as desired.
    \end{itemize}
    \label{thm:cobordismsarebottleneckedsurgery}
\end{corollary}
We now upgrade \cref{prop:eulerchar} to compare the Euler characteristic of a Lagrangian cobordism with immersed ends to the Euler characteristic of the ends.
\begin{prop}
    Let  $(K, t^-, t^+):(L^+, H^+)\to (L^-, H^-)$ be a Lagrangian cobordism with double bottlenecks. Then  we have a matching of Euler characteristics 
\[          \chi^{bot}(K)=\chi^{si}(L^-)=\chi^{si}(L^+)\]
\label{prop:selfIntersectionSummary}
\end{prop}
\begin{proof}
    As in \cref{prop:eulerchar}, it suffices to check for the double bottleneck surgery traces.
    \[K^{k, n-k+1}_{A, B}: L^{k, n-k,+}_{loc}   \rightsquigarrow    L^{k+1, n-k-1, -}_{loc}\]
    We discuss the cases $-1<k<n$. Equip $K^{k, n-k+1}_{A, B}$ with an admissible Morse function $f$ which is given by perturbing the Morse-Bott function $\rho\circ \pi_\RR$ in \cref{fig:morseprofile} by the functions $f^\pm$ from \cref{prop:eulerchar} at the maxima of $\rho\circ \pi_\RR$. $\rho\circ \pi_\RR$ has a single critical point (the critical point of the Morse surgery handle) with real coordinate between $t^+, t^-$ which lives in degree $n-k$. The remaining critical points of $f$ come from the perturbations $f^\pm$, which we enumerated in \cref{tab:indexdifferences}. The self-intersections of $K^{k, n-k+1}_{A, B}$ correspond to two copies of the self-intersections of $L^{k, n-k,+}_{loc} $ whose indexes are shifted depending on whether they receive maximum or minimum grading from the base.
    The bottleneck cochains $\mathcal I^{bot}(K^{k, n-k+1}_{A, B})$ are listed in \cref{fig:morseprofile}. From reading the table, we see that $\chi^{bot}(K^{k, n-k+1}_{A, B})=(-1)^{k+1}+(-1)^{n-k+1}$, which agrees with $\chi^{si}( L^{k+1, n-k-1, -}_{loc})$
\end{proof}
\begin{figure}
    \centering
    
\usetikzlibrary{patterns}

\begin{tikzpicture}
\usetikzlibrary{patterns, decorations.markings}
\begin{scope}[shift={(-9,-12.5)}]
\begin{scope}

\draw[fill=gray!20, dashed]  (-3,-1.5) ellipse (1 and 1.5);

\clip (-3,-1.5) ellipse (1 and 1.5);

\node[circle, fill=black, scale=.25] at (-3,-1.25) {};

\node[circle, fill=black, scale=.25] at (-3,-1.75) {};

\node[above] at (-3,-1.25) {$x_+$};

\node[below] at (-3,-1.75) {$x_-$};
\begin{scope}[thick,decoration={    markings  , mark=at position 0.5 with {\arrow{>}}}]

\draw[postaction={decorate}] (-3,-1.25) -- (-4,-1.25);
\draw[postaction={decorate}] (-3,-1.25) -- (-2,-1.25);
\draw[postaction={decorate}] (-4,-1.75) -- (-3,-1.75);
\draw[postaction={decorate}] (-2,-1.75) -- (-3,-1.75);
\end{scope}
\end{scope}

\fill[gray!10]  (-4,-4) rectangle (4,-7);

\draw[fill=brown!20,thick] (-3,-5.5) .. controls (-2.5,-5.5) and (-2,-6) .. (-1.5,-6) .. controls (-1,-6) and (1,-6) .. (1.5,-5.5) .. controls (1,-5) and (-1,-5) .. (-1.5,-5) .. controls (-2,-5) and (-2.5,-5.5) .. (-3,-5.5);
\draw[fill=brown!20,thick] (4,-5.5) .. controls (3.5,-5) and (2,-5) .. (1.5,-5.5) .. controls (2,-6) and (3.5,-6) .. (4,-5.5);
\draw[pattern=north west lines, pattern color=orange] (4,-5.5) .. controls (3.5,-5) and (2,-5) .. (1.5,-5.5) .. controls (2,-6) and (3.5,-6) .. (4,-5.5);
\draw (-4,-5.5) -- (-3,-5.5);
\draw[pattern=north west lines, pattern color=orange] (1.5,-5.5) .. controls (1,-6) and (-1.5,-6) .. (-1.5,-5.5) .. controls (-1.5,-5) and (1,-5) .. (1.5,-5.5);

\node at (-4.5,-2) {$S^1_{r_1}$};
\node at (-4.5,-1) {$S^1_{r_2}$};

\node[fill, scale=.5, circle] at (-1.5,-5.5) {};

\node[left] at (-1.5,-5.5) {$y$};
\node at (0,-5.5) {$A$};
\node at (3.25,-5.5) {$B$};

\begin{scope}[shift={(-0.5,0)}]

\clip  (1.75,-5) rectangle (4.75,-6.25);
\draw[red, thick] (-1,-6) .. controls (-0.5,-6) and (1.5,-6) .. (2,-5.5) .. controls (2.5,-5) and (4,-5) .. (4.5,-5.5);
\draw[blue, thick] (4.5,-5.5) .. controls (4,-6) and (2.5,-6) .. (2,-5.5) .. controls (1.5,-5) and (-0.5,-5) .. (-1,-5);

\end{scope}

\begin{scope}[shift={(-2.25,-0.5)}]

\draw[fill=gray!20, dashed]  (2,-1) ellipse (1 and 1.5);

\clip (2,-1) ellipse (1 and 1.5);
\draw[thick, postaction={decorate},decoration={    markings  , mark=at position 0.25 with {\arrow{>}},mark=at position 0.75 with {\arrow{>}}}] (1,-1.25) .. controls (1.5,-1.25) and (1.5,-2.25) .. (1.75,-2.25) .. controls (2,-2.25) and (2,0.25) .. (2.25,0.25) .. controls (2.5,0.25) and (2.5,-0.75) .. (3,-0.75);
\draw[thick, postaction={decorate},decoration={    markings  , mark=at position 0.25 with {\arrow{<}},mark=at position 0.75 with {\arrow{<}}}] (1,-0.75) .. controls (1.5,-0.75) and (1.5,0.25) .. (1.75,0.25) .. controls (2,0.25) and (2,-2.25) .. (2.25,-2.25) .. controls (2.5,-2.25) and (2.5,-1.25) .. (3,-1.25);

\clip  (1.8,-0.7) rectangle (2.2,-1.3);
\draw[red, thick] (1.75,-2.25) .. controls (2,-2.25) and (2,0.25) .. (2.25,0.25);
\draw[blue, thick] (1.75,0.25) .. controls (2,0.25) and (2,-2.25) .. (2.25,-2.25);
\end{scope}
\begin{scope}[shift={(0.75,-0.5)}]

\draw[fill=gray!20, dashed]  (2,-1) ellipse (1 and 1.5);

\clip (2,-1) ellipse (1 and 1.5);
\draw[thick, postaction={decorate},decoration={    markings  , mark=at position 0.25 with {\arrow{>}},mark=at position 0.75 with {\arrow{>}}}] (1,-1.25) .. controls (1.5,-1.25) and (1.5,-1.85) .. (1.75,-1.85) .. controls (2,-1.85) and (2,-0.15) .. (2.25,-0.15) .. controls (2.5,-0.15) and (2.5,-0.75) .. (3,-0.75);
\draw[thick, postaction={decorate},decoration={    markings  , mark=at position 0.25 with {\arrow{<}},mark=at position 0.75 with {\arrow{<}}}] (1,-0.75) .. controls (1.5,-0.75) and (1.5,-0.15) .. (1.75,-0.15) .. controls (2,-0.15) and (2,-1.85) .. (2.25,-1.85) .. controls (2.5,-1.85) and (2.5,-1.25) .. (3,-1.25);

\clip  (1.8,-0.7) rectangle (2.2,-1.3);
\draw[red, thick] (1.75,-1.7) .. controls (2,-1.7) and (2,-0.3) .. (2.25,-0.3);
\draw[blue, thick] (1.75,-0.3) .. controls (2,-0.3) and (2,-1.7) .. (2.25,-1.7);
\end{scope}

\end{scope}

\draw[dashed, ->] (-12.5,-19.5) -- (-12.5,-16);
\draw[dashed, ->] (-7.5,-19.5) .. controls (-7.5,-19) and (-7.5,-17) .. (-7.5,-16.5) .. controls (-7.5,-16.25) and (-7.5,-16.25) .. (-7.5,-16.25) .. controls (-7.5,-16) and (-7.5,-16) .. (-7.75,-16) .. controls (-8,-16) and (-8.75,-16) .. (-9,-16) .. controls (-9.25,-16) and (-9.25,-16) .. (-9.25,-15.75);
\draw[dashed, ->] (-5,-19.5) .. controls (-5,-19) and (-5,-17) .. (-5,-16.75) .. controls (-5,-16.5) and (-5,-16) .. (-5.5,-16) .. controls (-5,-16) and (-5.75,-16) .. (-6,-16) .. controls (-6.25,-16) and (-6.25,-16) .. (-6.25,-16) .. controls (-6.5,-16) and (-6.5,-16) .. (-6.5,-15.75);

\begin{scope}[]
\draw[<->] (-13,-20) -- (-13,-22) -- (-4.5,-22);
\draw (-13,-21.25) .. controls (-12.75,-20.5) and (-12.75,-20.5) .. (-12.5,-20.5) .. controls (-12.25,-20.5) and (-7.75,-21.75) .. (-7.5,-21.75) .. controls (-7.25,-21.75) and (-5.25,-20.25) .. (-5,-20.25) .. controls (-4.75,-20.25) and (-4.75,-20.25) .. (-4.5,-21.75);
\node[above] at (-13,-20) {$\rho$};
\node[right] at (-4.5,-22) {$\pi_{\mathbb R}$};

\end{scope}

\node[left] at (-15,-15) {(a) Slices};
\node[left] at (-14.6793,-19) {(b) Shadow};
\node[left] at (-13.8293,-21) {(c) Morse Profile};
\node at (-16.3,-24) {(d)};

\begin{scope}[yscale=0.5, shift={(0,-22.5)}]
\node at (-12.5,-24.5) {$x_+$};
\node at (-12.5,-26.5) {$x_-$};
\node at (-10.5,-27.5) {$y$};
\node at (-7.5,-25.5) {$(q_+,0)\to (q_-, 0)$};
\node at (-4.5,-24.5) {$(q_+,1)\to (q_-, 1)$};
\node at (-7.5,-26.5) {$(q_-,0)\to (q_+,0)$};
\node at (-14,-25.5) {$k+2$};
\node at (-14,-24.5) {$k+1$};
\node at (-14,-26.5) {$n-k-1$};
\node at (-14,-27.5) {$n-k$};
\node at (-14,-23.5) {Index};
\draw (-15,-24) -- (-3,-24);
\draw (-13,-23.5) -- (-13,-29);
\draw[thin, gray] (-15,-25) -- (-3,-25) (-15,-26) -- (-3,-26) (-15,-27) -- (-3,-27) (-15,-28) -- (-3,-28);
\node at (-7.5,-23.5) {Elements of $\mathcal I (K^{k, n-k+1}_{A, B})$};

\end{scope}

\end{tikzpicture}     \caption{The slices, shadow, and choice of profile $\rho$ used in constructing the Morse function. The drawn example is for $k=0, n=2$. (a) Critical points of the function $f^\pm$. The gradient flow of $f^\pm$ on the slice is indicated by the arrows. (b) The shadow of our Lagrangian cobordism. (c) The profile of a Morse-Bott function on $K^{k, n-k+1}_{A,B}$ which is subsequently perturbed by $f^\pm$ at the Morse-Bott maxima and minimum. (d) A tabulation of the indices of elements in $\mathcal I^{bot}(K^{k, n-k+1}_{A, b})$. Observe that cochains listed in rows $k+1, n-k-1$ agree with those from \cref{tab:indexdifferences}.}
    \label{fig:morseprofile}
\end{figure}
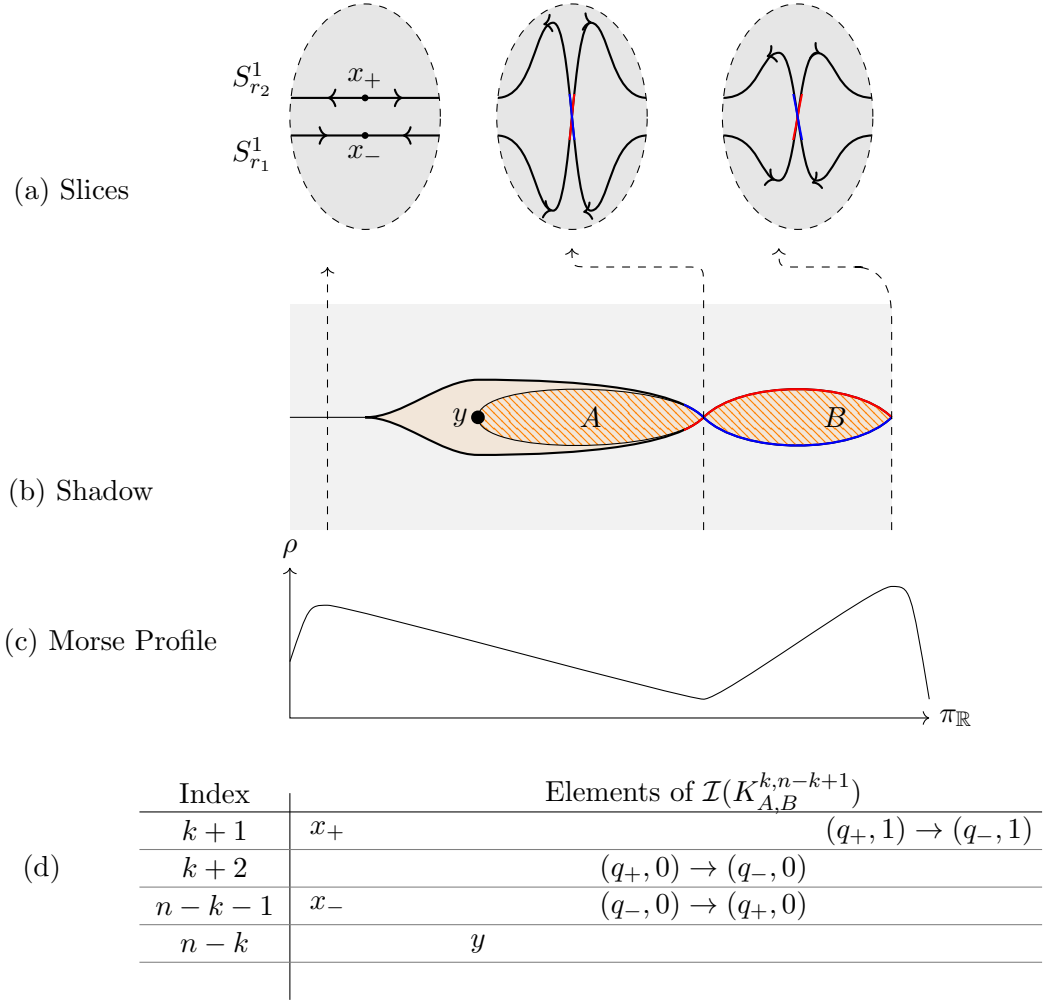
 \subsection{Overview of Immersed Floer cohomology}
\label{subsec:immersedoverview}
\label{subsec:basicimmersed}
\subsubsection{Immersed Lagrangian Floer Cohomology}
There are several models based on the work of \cite{akaho2008immersed} which produce a filtered $A_\infty$ algebra associated to an immersed Lagrangian $\li: L\to X$. 
We follow notation from \cite{palmer2019invariance} adapted to the Morse cochain setting, although much of our intuition for this filtered $A_\infty$ algebra come from \cite{fukaya1997zero,biran2008lagrangian,lipreliminary,charest2015floer}.

Let $f: L \to \RR$ be a Morse function, $\li: L\to X$ be a graded Lagrangian immersion with transverse self-intersections,  and  $\mathcal I^{si}(\li)=\{(p\to q)\;|\; p, q\in L, \li(p)=\li(q)\}$ be the set of ordered preimages of transverse self-intersections. Assume that $\Crit(f)\cap \mathcal I(L)=\emptyset$. 

\begin{definition}[\cite{fukaya2010lagrangian}]
    Let $R$ be a commutative ring with unit. The \emph{universal Novikov ring} over $R$ is the set of formal sums
    \[
        \Lambda_{\geq 0}:=\left\{\left.\sum_{i=0}^\infty a_i T^{\lambda_i} \;\right|\;a_i\in R, \lambda_i\in \RR_{\geq 0},  \lim_{i\to\infty} \lambda_i=\infty.\right\}
    \]

    Let $k$ be a field. The \emph{Novikov Field} is the set of formal sums
    \[
        \Lambda := \left\{\left.\sum_{i=0}^\infty  a_i T^{\lambda_i}  \right|a_i\in k, \lambda_i \in \RR,  \lim_{i\to\infty} \lambda_i = \infty\right \}.
    \]
    An \emph{energy filtration} on a graded $\Lambda$-module $A^\bullet$  is a filtration $F^{\lambda_i}A^k$ so that
    \begin{itemize}
        \item
              Each $A^k$ is complete with respect to the filtration \footnote{In the sense that the topology induced by taking a basis of opens to be $F^{\lambda_i}A^k$ is complete} and has a basis with valuation zero over $\Lambda$.
        \item Multiplication by $T^\lambda$ increases the filtration by $\lambda$.
    \end{itemize}
    \label{def:novikov}
\end{definition}
Associated to an immersed Lagrangian submanifold with transverse self-intersections is a filtered graded $\Lambda$ module generated on critical points of $f$ and the $(q_0\to q_1)$.
\[CF^k(L)= \bigoplus_{\substack{x\in \Crit(f)\\\ind(x)=k}} \Lambda_x \oplus \bigoplus_{\substack{(q_0\to q_1)\in \mathcal I^{si}(L)\\ \ind(q_0\to q_1)=k}} \Lambda_{(q_0\to q_1)}.\]
We will write $\mathcal I(L)= \Crit(f)\cup \mathcal I^{si}(L)$ for the set of generators of $\CF(L)$.
The Floer cohomology of $L$ comes with filtered product operations deforming the Morse structure by counts of holomorphic polygons with boundary on $L$,
\[\langle m^k(x_1, \ldots, x_k), x_0\rangle =\sum_{\beta\in H_2(X, L)} T^{\omega(\beta)} \#\mathcal M_{\mathcal P}(L,\beta, \underline x_i) \]
where $\underline x= x_0, x_1, \ldots, x_k \in \Crit(f)\cup \mathcal I^{si}(L)$, $\mathcal P$ is perturbation datum, and $\mathcal M_{\mathcal P}(L, \beta, \underline x_i)$ is the moduli space of $\mathcal P$-perturbed pseudoholomorphic treed polygons.
We provide a short description of what this data entails.
In \cite[Section 4.2]{charest2015floer}, a treed disk $C$ is a tree whose vertices $v$ are labeled with disks with internal marked points and $\deg(v)$ boundary marked points and whose edges are labelled with (possibly semi-infinite) intervals.
From this data we obtain a space $C$ which is glued together from the disks and intervals labelling the vertices and edges. We write $C=S\cup T$, where $S$ is the \emph{surface} portion of the treed disk, and $T$ are the edges.
For a pseudoholomorphic treed polygon, this surface portion is allowed to have strip-like ends.
Given a Lagrangian $L$ with Morse function $f$ we examine maps $u: C\to X$ which sends $u|_{\partial S}, u|_T\subset L$. Without considering perturbation datum, a pseudoholomorphic treed disk is such a map where $u|_S$ is pseudoholomorphic, and $u|_T$ are flow lines of $\nabla f$. In order to achieve regularity, we need to perturb the $J$-holomorphic curve and Morse flow line equations by picking perturbations of the almost complex structure and Morse function. A perturbation datum $\mathcal P$ is a choice of perturbations for all domains $C$. These choices must be made coherently, meaning that the perturbations for treed disks which differ by disk bubbling and flow-line breaking have related perturbation datum \cite[Definition 4.12]{charest2015floer}.

We will reserve $\mathcal P^0$ to denote the trivial perturbation datum (so that $\mathcal M_{\mathcal P^0}(L, \beta, \underline x_i)$ consists of configurations of $J$-holomorphic polygons attached to Morse flow-lines).
It is expected that the $m^k$ form a filtered $A_\infty$ algebra (in the sense of \cite{fukaya2007lagrangian}). As we only use these structure coefficients to make some suggestive computations, we do not claim that here.

In good examples, we know that $\CF(L)$ is an $A_\infty$-algebra. We give such an example below.

\subsubsection{Example Computation: Whitney Sphere}
We review a computation from \cite{alston2019immersed} computing the Floer complex for the Whitney sphere $L^{n, 0}_A\subset \CC^n$. We assume that $n\geq 2$ so that $L^{n, 0}_A$ is a graded Lagrangian submanifold. 
We take a Morse function for the $L^{n, 0}_A$ given by the $x_1$ coordinate. 
Then $\Crit(f)=\{e, x\}$, where $e$ is the maximum of $f$ in degree $0$, and $x$ is a generator in degree $n$.
We take the points $q_\pm = (\pm r(A),0, \ldots, 0 )\in S^n_{r(A)}$ to be the preimages of the self-intersections of this Whitney sphere.
The computation from \cref{example:indexcomputation} shows that $(q_+\to q_-)$ has degree $n+1$, and $(q_-\to q_+)$ has degree $-1$.
\begin{SCfigure}[50]
    \begin{tikzpicture}
    \node at (-0.5,1.5) {$0$};
    \node (v2) at (1.5,1.5) {$e$};
    \node (v3) at (1.5,0) {$x$};
    \node (v4) at (1.5,-1) {$(q_+\to q_-)$};
    \node (v1) at (1.5,2.5) {$(q_-\to q_+)$};
    \node at (-0.5,0) {$n$};
    \node at (-0.5,-1) {$n+1$};
    \node at (-0.5,2.5) {$-1$};
    \node at (-0.5,1) {$\vdots$};
    \draw[->,red]  (v1) edge (v2);
    \draw[->,red]  (v3) edge (v4);
    \end{tikzpicture}     \caption{$m^1$ structure on $\CF(L^{n,0}_A)$ with $n\geq 2$. There are no Morse flow lines contributing the differential on the Floer cohomology. However, the presence of holomorphic teardrops contributes to differential (drawn in red) cancelling all homology classes.}
    \label{fig:homcomputation}
\end{SCfigure}
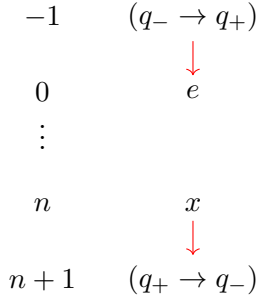
We now give a description of the moduli space of holomorphic teardrops with boundary on the Whitney sphere.
Let $p=( x_1, \ldots, x_n,0)\in S^n\subset \RR^n\times \RR$ be a point on the equatorial sphere. Then we can construct a holomorphic teardrop with boundary on the Whitney sphere $L^{n, 0}_A \subset \CC^n$ which is parameterized by 
\begin{align*}
    u_p: D_\alpha\to& \CC^n\\
    z\mapsto& (x_1z, x_2z, \ldots, x_n z)
\end{align*}
where $D_\alpha=\{a+\jmath b \;|\; a\in [0, 1], |b|\leq 2a\sqrt{1-a^2}\}$.
Let $\beta\in H_2(X, L)$ be the homology class of this teardrop. The parameter $A$ of the Whitney sphere is constructed so that 
\[\omega(\beta)=A.\]
\begin{claim}[Lemma 5.1.2 of \cite{alston2019immersed}]
    The moduli space of holomorphic teardrops with boundary on $S^n_1$ is regular for the standard almost complex structure and has one component,
    \[\mathcal M_J(L^{n, 0}_A, \beta, (q_-\to q_+)) = S^{n-1}\]
    where $\beta$ is the class of the teardrop $[u_p]$. The evaluation map 
    \[\ev: \mathcal M_J(L^{n, 0}_A, \beta, (q_-\to q_+))\times (D^2\setminus\{1\})\to L^{n, 0}_A\]
    is a homeomorphism onto $L^{n, 0}_A\setminus\{q_+, q_-\}.$
    \label{claim:whitney}
\end{claim}
\begin{proof}
    In the version 2 preprint of \cite{alston2019immersed}, it is proven that all teardrops with boundary on $L^{n, 0}_A$ for the standard almost complex structure are of the form $u_{p}$.
    
    It remains to show that these teardrops are regular. In this proof, we will write $L:=L^{n, 0}_A$.     
    We need only show that the linearized Dolbeault operator
    \[D_{u_{p}}: W^{1, k;\epsilon}(u_{p}^*T\CC^n, \lambda)\to L^{k;\epsilon}(\Omega^{0,1}\tensor u_{p}^*T\CC^n)\]
    surjects. 
    For expositional purposes we will now restrict to $p = (0, 0, \ldots, 1)$. 
    Since $T\CC^n$ trivializes across the coordinates of $\CC^n$ as $\bigoplus_{i=1}^n \mathcal L_i$, 
    the pullback $u_{p}^*T\CC^n$ has a trivialization $\bigoplus_{i=1}^n(u_{p}^*\mathcal L_i)$ as well. 
    The proof of \cref{claim:indexOfHandle} shows that the Lagrangian sub-bundle $(u_p|_{\partial D^2})^*TL$ can also be trivialized. 
    In that proof, the curve $\gamma(\theta)= (\sin(\theta), 0, \ldots, 0, \cos(\theta))\in S^n\subset \RR^{n+1}$ has the property that 
    \[\li^{n, 0}\circ \gamma(\theta)=(0, \ldots, \cos(\theta)+\jmath 2 \sin(\theta)\cdot \cos(\theta))\]
    parameterizes the boundary of $u_p$. 
    The proof of \cref{claim:indexOfHandle} therefore shows that $(u_p|_{\partial D^2})^*TL$ splits as a sum of real line bundles $\bigoplus_{i=1}^n \lambda_i$. Furthermore, this splitting respects the splitting $T\CC^n=\mathcal L_1\oplus\cdots \oplus\mathcal L_n$, in the sense that 
    \[ \lambda_i \subset (u_p|_{\partial D^2})^*\mathcal L_i.\]
    The map $D_{u_p}$ decomposes across the trivialization 
    \[\bigoplus_{i=1}^n W^{1, k;\epsilon}(u_{p}^*\mathcal L_i, \lambda_i)\to \bigoplus_{i=1}^n L^{k;\epsilon}(\Omega^{0,1}\tensor u^*_{p}\mathcal L_i)\]
    Let $D_{u_{p}}|_{\mathcal L_i}:  W^{1, k;\epsilon}(u_{p}^*\mathcal L_i, \lambda_i)\to L^{k;\epsilon}(\Omega^{0,1}\tensor u^*_{p}\mathcal L_i)$ be the restriction of $D_{u_{p}}$ to each component of the trivialization. We show that each of the $D_{u_{p}}|_{\mathcal L_i}$ surjects.
\begin{itemize}
    \item    For $1 \leq i < n$, the argument from \cref{claim:indexOfHandle} shows that the Maslov index is of the loop determined by $\lambda_i$ is 0; the Fredholm index of $D_{u_{p}}|_{\mathcal L_i}$ is 1. By automatic regularity in dimension one, this linearized operator surjects.
    \item For $i= n$, the argument from \cref{claim:indexOfHandle} shows that the Maslov index is 2; the Fredholm index of $D_{u_{p}}|_{\mathcal L_1}$ is 3. By automatic regularity in dimension one, this linearized operator surjects.
\end{itemize}
    It follows that $D_{u_{p}}$ is surjective, with Fredholm index $\ind(D_{u_{p}})=(n-1)+2=n+1$. This gives us an $\ind(D_{u_{p}})-2=n-1$-dimensional moduli space of teardrops, as expected.
\end{proof}
\Cref{fig:homcomputation} contains a summary of generators and differentials of $\CF(L^{n, 0}_A)$ which shows that the Floer cohomology of the Whitney sphere vanishes.
We now place this computation in context with the table at the end of \cref{fig:morseprofile}.
 In that table, a matching is visible on the chain level between critical points of the surgery handles and the self-intersections of the handles. 
In the Whitney sphere, the holomorphic teardrops pair the Floer cochains from self-intersections with the Floer cochains from the auxiliary Morse function in the sense that the Floer differential is exact.  
The goal of \cref{subsec:teardrops} is to show that there are holomorphic teardrops on the Lagrangian surgery trace which similarly pairs critical point of $\rho\circ \pi_\RR: K^{k, n-k+1}_{A, B}$ with a self-intersection in the double bottleneck.

These holomorphic teardrops can be used in examples of immersed Lagrangian cobordism with transverse self-intersection to construct continuation maps; a toy computation is given in \cref{subsec:initialcomputation}.

\subsubsection{Deformations and Lagrangian intersection Floer Cohomology}

Lagrangian surgery can be interpreted as a geometric deformation of a Lagrangian submanifold.
To understand the conjectured equivalent deformation on Floer cohomology, we need to discuss deformations of filtered $A_\infty$ algebras.
\begin{definition}
    Let $A$ be a filtered $A_\infty$ algebra.
    Let $\mathfrak d\in A$ be an element with $\val(\mathfrak d)>0$. 
    The $\mathfrak d$-deformed algebra $(A, \mathfrak d)$ is the filtered $A_\infty$ algebra whose chains match $A$, and whose product is given by
    \[m^k_{\mathfrak d}(a_1\tensor \cdots \tensor a_k) = \sum_{i=0}^\infty \sum_{i_0+\ldots + i_k= i} m^{k+i}(\mathfrak d^{\tensor i_0}\tensor a_1 \tensor \mathfrak d^{\tensor i_1}\tensor \cdots \tensor \mathfrak d^{\tensor i_{k-1}}\tensor a_k\tensor \mathfrak d^{\tensor i_{k}} ).\]
    We say that $\mathfrak b$ is a \emph{bounding cochain} or \emph{is a solution to the Maurer-Cartan equation} if $m^0_{\mathfrak b}=0$.
    The space of Maurer-Cartan solutions is denoted by $\mathcal MC(A)$.
    If there exists a bounding cochain for $A$, we say that $A$ is \emph{unobstructed}.
    Given $N$ an $A-B$ filtered $A_\infty$-bimodule, and deforming chains $\mathfrak d_A\in A, \mathfrak d_B\in B$, we obtain $(N, \mathfrak d_A, \mathfrak d_B)$ which is an $(A, \mathfrak d_A)-(B, \mathfrak d_B)$ bimodule. The product structure is given by
    \begin{align*}
        m^{k_1|1|k_2}_{(A, \mathfrak d_A)|N|(B, \mathfrak d_B)}&(a_1\tensor\cdots \tensor a_{k_1}\tensor n \tensor b_1\tensor \cdots b_{k_2})\\
         =& \sum_{i, j\in \NN} \sum_{\substack{i_0+\ldots + i_{k_1}= i\\j_0+\ldots + j_{k_2}= j}} m^{k_1+i|1|k_2+j}_{A|N|B}\left( 
            \begin{array}{l}
                \mathfrak d_A^{\tensor i_0}\tensor a_1 \tensor \mathfrak d^{\tensor i_1}\tensor \cdots\\
                \cdots  \tensor  a_k\tensor \mathfrak d_A^{\tensor i_{k}} \tensor n \tensor  \mathfrak d_B^{\tensor j_0}\tensor b_1 \tensor \mathfrak d^{\tensor j_1}\tensor \cdots \\
                \cdots \tensor  b_k\tensor \mathfrak d_B^{\tensor j_{k}}
            \end{array}\right).
    \end{align*}
    \label{def:boundingcochain}
\end{definition}
Observe that when $\mathfrak b$ is a bounding cochain for $A$ then the differential on $(A, \mathfrak b)$ squares to zero.
If $\mathfrak b_A, \mathfrak b_B$ are bounding cochains for $A, B$, and $N$ is an $A-B$ filtered $A_\infty$ bimodule, then $(N,\mathfrak b_A, \mathfrak b_B)$ is an $(A, \mathfrak b_A)-(B, \mathfrak b_B)$ \emph{uncurved} $A_\infty$ bimodule. In particular, the differential on $(N, \mathfrak b_A, \mathfrak b_B)$ squares to zero. 

When discussing the Lagrangian Floer cohomology, we will write $\CF(L, \mathfrak d)$ for the Lagrangian Floer cohomology deformed by the element $\mathfrak d$.
If $\CF(L)$ is unobstructed, then we say that the Lagrangian $L$ is unobstructed. 
At lowest order, the Maurer-Cartan equation for a bounding cochain for $\CF(L)$ states that 
\[m^1_{morse}(\mathfrak b)=m^0+\text{higher order terms},\]
where $m^1_{morse}$ is the Morse differential on $L$. One interpretation of this is: to first approximation, an unobstructed Lagrangian submanifold is a Lagrangian submanifold where the boundary classes of holomorphic disks with boundary on $L$ cancel out in homology.

Given two Lagrangian submanifolds $L_0, L_1$ which intersect transversely we will write 
\[\CF(L_0, L_1):=\bigoplus_{\substack{x\in L_0\cap L_1}} \Lambda_x\]
for the Lagrangian intersection Floer cochains of $L_0$ and $L_1$. It is expected that is an 
$\CF(L_0)- \CF(L_1)$ filtered $A_\infty$ bimodule, whose bimodule structure comes from counting treed-pseudoholomorphic polygons. 
Given $\mathfrak d_i\in \CF(L_i)$ deforming cochains, we write  $\CF((L_0, \mathfrak d_0), (L_1, \mathfrak d_1))$ for the $\CF(L_0, \mathfrak d_0)- \CF(L_1, \mathfrak d_1)$ filtered $A_\infty$ bimodule.
Again, as we are primarily interested in making some motivating computations, we will not use the algebraic structures of this complex and simply use the language of $A_\infty$ algebras to make some remarks about the areas of various polygons which naturally occur in Lagrangian submanifolds associated to surgeries. 
\subsubsection[Running Example: Multisection of Cotangent of Circle]{Running example: Multisection of $T^*S^1$.}
\label{subsubsec:runningexample}
We return to the Lagrangian submanifold $L_E\subset T^*S^1$ first defined in  \cref{exam:obstructedLagrangian} , parameterized by 
\begin{align*}
    S^1\to& T^*S^1\\
    \theta\mapsto& \left(2\theta, \frac{E}{8}\sin(\theta)\right)
\end{align*}
This is a Lagrangian submanifold with 1-self-intersection where $\theta=0$ and $\theta=\pi$. With respect to the standard holomorphic form on $T^*S^1=\CC^*$, the generator $(\pi \to 0)$ has degree 0 and $(0\to \pi)$ has degree $1$. There are two holomorphic strips of area $E$ from the $(\pi \to 0)$ generator to the $(0\to \pi)$ generator. 
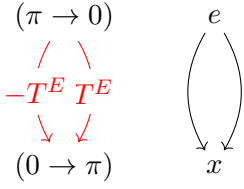
\begin{SCfigure}[50]
    \begin{tikzpicture}
    \node (v1) at (-2.5,1.5) {$(\pi \to 0)$};
    \node (v3) at (-0.5,1.5) {$e$};
    \node (v2) at (-2.5,-0.5) {$(0\to \pi)$};
    \node (v4) at (-0.5,-0.5) {$x$};
    \draw[red, ->]  (v1) edge[bend left] node[fill=white]{$T^E$}(v2);
    \draw[red, ->]  (v1) edge[bend right] node[fill=white]{$- T^E$} (v2);
    \draw  (v3) edge[->, bend left] (v4);
    \draw  (v3) edge[->, bend right] (v4);
\end{tikzpicture}
\caption{The chain complex $\CF(L_E)$. The black lines represent the Morse Flow lines, while the red edges represent contributions to the differential coming from holomorphic strips. }
\end{SCfigure}
In this case there are no holomorphic disks or teardrops with boundary on $L_E$ so we may compute $HF^\bullet(L_E)$,
which is isomorphic as a vector space to  $H^\bullet(S^1)\oplus H^\bullet(S^1)$.
We note that we do \emph{not} expect $\CF(L_E)$ to be homotopic as a differential graded algebra or filtered $A_\infty$ algebra to  $C^\bullet(S^1\cup S^1)$.

A computation which will be useful later is the Lagrangian intersection Floer cohomology between $L_E$ and a section $S^1_{E'}$ of the cotangent bundle. Here, $S^1_{E'}$ is the Lagrangian section parameterized by 
\[\theta\mapsto \left(\theta, \frac{E'}{2\pi}\right),\]
as drawn in \cref{fig:figureEight}.
Let $\mathfrak b := \sum_{i} a_iT^{D_i}(0\to \pi)\in \CF(L_E)$ be a deforming cochain for $L_E$, whose lowest order term is $a_0T^{D_0}$. 
We compute $\CF((L_E, \mathfrak b), S^1_{E'})$. 
The intersection between these two Lagrangian submanifolds consists of two points,  
\[L_E\cap S^1_{E'}=\{p, q\}.\]
There are two holomorphic strips with boundary on $L_E$ and $S^1_{E'}$ which have area $A$ and $B$. Additionally, there exists a holomorphic triangle with ends limiting to $q, p$ and $(0\to \pi)$ of area $C$. 
The constants $A, B, C, E, E'$ satisfy the relation: 
\begin{align*}
    B+E'=\frac{E}{2}+C
\end{align*}
The differential on $\CF((L_E, \mathfrak b), S^1_{E'})$ is given by 
\[m^1(p)=T^A-T^B+a_0T^{D_0}\cdot T^C+ \mathcal O(\min(A, D_0+C)).\]
For this differential to vanish we obtain the constraint at lowest order that:  
\[D_0=\frac{E}{2}-E'.\]
Furthermore, if this condition is met, there exists some extension of $T^{D_0}(0\to \pi)$ to a bounding cochain $\mathfrak b$ so that $\HF((L_E, \mathfrak b), S^1_{E'})= H^\bullet(S^1)$.
\begin{figure}
    \centering
    \begin{tikzpicture}
\begin{scope}[shift={(-10,0)}]

\fill[gray!20]  (-3,1.5) rectangle (0,-2.5);
\draw[dotted] (-3,-0.5) -- (0,-0.5);
\draw (-3,-1.3) -- (0,-1.3);
\draw (-3,-0.5);
\draw (-3,1) .. controls (-1.5,1) and (-1.5,-2) .. (0,-2);
\draw (0,1) .. controls (-1.5,1) and (-1.5,-2) .. (-3,-2);
\draw [decoration={brace},decorate] (-3,-1.3) -- (-3,-0.5) node [black,midway,xshift=-0.6cm] {\footnotesize $\frac{E'}{2\pi}$};
\node[fill=black, circle, scale=.3] at (-1.5,-0.5) {};
\node[right] at (-1.4,-0.4) {$(0\to \pi)$};
\node[fill=black, circle, scale=.3] at (-1.05,-1.3) {};
\node[below left] at (-1.05,-1.3) {$p$};

\node[fill=black, circle, scale=.3] at (-1.95,-1.3) {};
\node[below right] at (-1.95,-1.3) {$q$};

\begin{scope}[]

\clip  (-3,1.5) rectangle (0,-1.3);
\fill[pattern=north west lines, pattern color=red] (-3,-2) .. controls (-1.5,-2) and (-1.5,1) .. (0,1) .. controls (0,-0.5) and (0,-0.5) .. (0,-2);
\fill[pattern=north east lines, pattern color=red] (-3,1) .. controls (-1.5,1) and (-1.5,-2) .. (0,-2) .. controls (-1,-2) and (-1.5,-2) .. (-3,-2);

\end{scope}

\node[red] at (0.2,0) {$A$};

\end{scope}

\begin{scope}[shift={(-1,0)}]

\fill[gray!20]  (-3,1.5) rectangle (0,-2.5);
\begin{scope}[]

\clip  (-3,1.5) rectangle (0,-1.3);
\clip (-3,-2) .. controls (-1.5,-2) and (-1.5,1) .. (0,1) .. controls (0,-0.5) and (0,-0.5) .. (0,-2);
\fill[orange] (-3,1) .. controls (-1.5,1) and (-1.5,-2) .. (0,-2) .. controls (-1,-2) and (-1.5,-2) .. (-3,-2);

\end{scope}
\draw[dotted] (-3,-0.5) -- (0,-0.5);
\draw (-3,-1.3) -- (0,-1.3);
\draw (-3,-0.5);
\draw (-3,1) .. controls (-1.5,1) and (-1.5,-2) .. (0,-2);
\draw (0,1) .. controls (-1.5,1) and (-1.5,-2) .. (-3,-2);\node[fill=black, circle, scale=.3] at (-1.5,-0.5) {};
\node[right] at (-1.4,-0.4) {$(0\to \pi)$};
\node[fill=black, circle, scale=.3] at (-1.05,-1.3) {};
\node[below left] at (-1.05,-1.3) {$p$};

\node[fill=black, circle, scale=.3] at (-1.95,-1.3) {};
\node[below right] at (-1.95,-1.3) {$q$};

\node[orange] at (-0.95,-0.85) {$C$};
\end{scope}

\begin{scope}[shift={(-5.5,0)}]

\fill[gray!20]  (-3,1.5) rectangle (0,-2.5);
\draw[dotted] (-3,-0.5) -- (0,-0.5);
\draw (-3,-1.3) -- (0,-1.3);
\draw (-3,-0.5);
\draw (-3,1) .. controls (-1.5,1) and (-1.5,-2) .. (0,-2);
\draw (0,1) .. controls (-1.5,1) and (-1.5,-2) .. (-3,-2);\node[fill=black, circle, scale=.3] at (-1.5,-0.5) {};
\node[right] at (-1.4,-0.4) {$(0\to \pi)$};
\node[fill=black, circle, scale=.3] at (-1.05,-1.3) {};
\node[below left] at (-1.05,-1.3) {$p$};

\node[fill=black, circle, scale=.3] at (-1.95,-1.3) {};
\node[below right] at (-1.95,-1.3) {$q$};

\begin{scope}[]

\clip  (-3,-1.3) rectangle (0,-2);
\fill[pattern=north west lines, pattern color=blue] (-3,-2) .. controls (-1.5,-2) and (-1.5,1) .. (0,1) .. controls (-1.5,1) and (-1.5,1) .. (-3,1);
\fill[pattern=north west lines, pattern color=blue] (0,-2) .. controls (-1.5,-2) and (-1.5,1) .. (-3,1) .. controls (-1.5,1) and (-1.5,1) .. (0,1);

\end{scope}

\node[blue] at (0.2,-1.7) {$B$};

\end{scope}
\end{tikzpicture}     \caption{Areas of various holomorphic strips and triangles appearing in the computation of$\CF((L_E, \mathfrak b), S^1_{E'})$ }
    \label{fig:figureEight}
\end{figure}
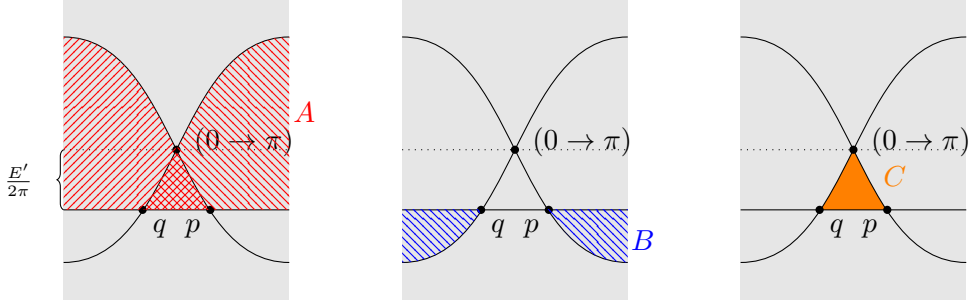
We will recover this bounding cochain from the Lagrangian surgery trace cobordism in \cref{subsec:samplecomputation}. 

 \subsection{Existence of Holomorphic Teardrops}
\label{subsec:teardrops}

\subsubsection{Teardrops on Surgery Traces}

We now examine the Floer cohomology of the bottlenecked surgery handle $K^{k, n-k+1}_{A, B}$. The Floer cochains are given in \cref{fig:morseprofile}. The Floer products, in addition to counting treed pseudoholomorphic disks, counts treed pseudoholomorphic polygons (\cite{alston2019immersed,palmer2019invariance}).
The strip-like ends of these polygons limit to the self-intersections of our Lagrangian submanifold.
When a treed pseudoholomorphic polygon has only 1 strip like end (either input or output), we call it a treed pseudoholomorphic teardrop.
\begin{theorem}
    Let $K^{k, n-k+1}_{A, B}$ be the local model of the double bottleneck Lagrangian surgery trace. Take the standard Morse function $\pi_{\RR}:  K^{k, n-k+1}_{A, B}\to \RR$, and standard choice of almost complex structure on $\CC^{n+1}$.
     Let $\beta_A$ be the class of teardrop with boundary on $K^{k, n-k+1}_{A, B}$ which has boundary on the Whitney isotropic contained in $K^{k, n-k+1}_{A, B}|_0$. The moduli space  $\mathcal M_J(K^{k, n-k+1}_{A, B}, \beta_A,((q_+,0)\to (q_-,0))),x_+)$ is comprised of a single regular treed holomorphic polygon, which is a teardrop. 
     \label{thm:teardropexistence}
\end{theorem} 
\begin{proof}
    The space of holomorphic teardrops with input on $((q_+,0)\to (q_-,0)))$ can be described by the space of holomorphic teardrops with boundary on the isotropic $k$-sphere contained in $K^{k, n-k+1}_{A, B}$.
    By arguments similar to \cref{claim:whitney}, the moduli space of such teardrops is regular and has an evaluation map which sweeps out the homology class of $y$.
\end{proof}
This leads to the appearance of a term in the differential of the Floer complex of $K^{k, n-k+1}_{A, B}$; see \cref{fig:bottleneckfloercomplex}.
\begin{figure}
    \centering
    \usetikzlibrary{matrix, arrows,  decorations.markings,decorations.pathreplacing,  patterns,  plotmarks}
\begin{tikzpicture}

\begin{scope}[shift={(6,0)}]

\fill[gray!20]  (-4.5,4) rectangle (3,2);
\draw[thick] (1,3.5) .. controls (0.5,3.5) and (0,3.5) .. (-0.5,3) .. controls (-1.5,2) and (-3.5,2.5) .. (-3.5,3) .. controls (-3.5,3.5) and (-1.5,4) .. (-0.5,3) .. controls (0,2.5) and (0.5,2.5) .. (1,2.5);

\end{scope}

\draw[pattern=north west lines, pattern color=orange] (5.5,3) .. controls (4.5,2) and (2.5,2.5) .. (2.5,3) .. controls (2.5,3.5) and (4.5,4) .. (5.5,3);
\node[above] at (5.5,3.4) {$(q_-\to q_+)$};
\node[left] at (2.5,3) {$y$};
\node (v1) at (2.5,-2) {$y$};
\node[blue] (v2) at (5.5,0) {$((q_+,0)\to (q_-,0))$};
\node at (2,3.5) {$\mathbb C$};

\draw[thick] (7,2.5) .. controls (7.5,2.5) and (8,2.5) .. (8.5,3) .. controls (9,3.5) and (9,3.5) .. (9,3.5);
\draw[thick] (7,3.5) .. controls (7.5,3.5) and (8,3.5) .. (8.5,3) .. controls (9,2.5) and (9,2.5) .. (9,2.5);
\draw[pattern=north west lines, pattern color=blue] (5.5,3) .. controls (6,2.5) and (6.5,2.5) .. (7,2.5) .. controls (7.5,2.5) and (8,2.5) .. (8.5,3) .. controls (8,3.5) and (7.5,3.5) .. (7,3.5) .. controls (6.5,3.5) and (6,3.5) .. (5.5,3);
\node at (1.5,1.5) {Index};

\node at (1.5,0) {$k+2$};
\node at (1.5,1) {$k+1$};
\node at (1.5,-1) {$n-k-1$};

\node[red] (v4) at (5.5,-1) {$ ((q_-,0)\to (q_+,0))$};
\node[red] (v3) at (8.5,1) {$((q_+,1)\to (q_-,1))$};
\draw  (v3) edge[->] node[fill=white] {$T^B$} (v2);
\node  at (1.5,-2) {$n-k$};
\draw  (v4) edge[->] node[fill=white]{$T^A$} (v1);
\node[rounded corners, fill=orange!20] (v5) at (5.5,1.5) {$m^0$ if $k=0$};
\draw[->]  (v5) edge node[fill=white]{$T^A$} (v2);
\end{tikzpicture}     \caption{Floer cohomology of the bottlenecked surgery trace $K^{k, n-k+1}_{A, B}$.
    The holomorphic teardrop is marked in orange and has area $A$, the holomorphic strip from the bottleneck is marked in blue and has area $B$.}
    \label{fig:bottleneckfloercomplex}
\end{figure}
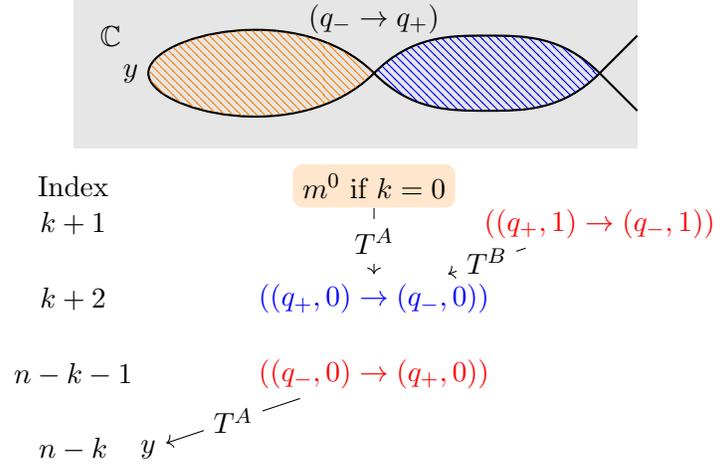
The most interesting example of this occurs when we take the trace of the standard Polterovich surgery, when $k=0$. 
In this setting, the Whitney isotropic is a 1-dimensional sphere, and so the space of holomorphic teardrops is zero-dimensional. This gives us an additional contribution which we can count. 
\begin{corollary}
    For the standard choice of $J$,  $\mathcal M_J(K^{0, n+1}_{A, B}, \beta_A,((q_-,0)\to (q_+,0)))$ consists of a single holomorphic teardrop.
    \label{claim:teardropm0}
\end{corollary}
The presence of this contribution will turn on a non-trivial curvature term $T^A((q_+,0)\to (q_-,0())_-$ in the $A_\infty$ algebra. For $\CF(K^{0, n+1}_{A, B})$ to be unobstructed, there must be a Floer cochain whose coboundary is $T^A((q_+,0)\to (q_-,0))$. A candidate cochain would be $T^{B-A}((q_+,1)\to (q_-,1))$. 
As deforming cochains are required to have positive valuation, we see that $K^{0, n+1}_{A, B}$ is unobstructed if and only if $B<A$.
 
\section{Sample Computation}
\label{sec:applications}
The previous section suggests the following algorithm for computing a continuation map associated to an unobstructed Lagrangian cobordism. First, decompose the Lagrangian cobordism $K$ using \cref{thm:cobordismsaresurgery} into pieces $K_i$ which are either a suspension of an exact homotopy or surgery trace. 
If we can show that unobstructed of $K$ implies unobstructedness of each of the $K_i$ (\cref{con:hofunobstructed}), then \cref{thm:teardropexistence} suggests that each $K_i$ provides a continuation map between its ends. The continuation map associated to $K$ would then be given by composing all of these smaller continuation maps together.

More precisely, we conjecture that there exists a version of Floer cohomology for a Lagrangian cobordism with double bottlenecks, $\CF_{bot}(K)$, which has filtered $A_\infty$ projections
    \[\begin{tikzcd}
        & \CF_{bot}(K) \arrow{dl}{\pi^+}\arrow{dr}{\pi^-}\\
        \CF(L^+)&& \CF(L^-).
    \end{tikzcd}\]
As a vector space, $\CF_{bot}(K)$ is generated on $\mathcal I^{bot}(K)$. We observe from \cref{fig:morseprofile} that 
\[\mathcal I^{bot}(K)= \mathcal I(L^-)\cup \mathcal I(L^+)\cup (\mathcal I(L^+)[1]\setminus \{(q_-\to q_+)\})\cup\{y\}.\]
Since $\ind(y)=\ind(q_-\to q_+)+1$, we identify $\CF_{bot}(K)$ as a vector space with $\CF(L^-)\oplus \CF(L^+)[1]\oplus \CF(L^+)$. Under this identification the projections onto the first and last component agree with the maps $\pi^\pm$ above, and $\CF(L^+)[1]$ is a filtered $A_\infty$ ideal of $\CF_{bot}(K)$.
Then \cref{thm:teardropexistence} suggests that the maps induced by the Floer differential on $\CF(K)$, 
\begin{align*} \CF(L^+)\to \CF(L^+)[1] && \CF(L^-)\to \CF(L^+)[1]\end{align*}
are invertible (although the inverse will have negative Novikov valuation). A similar approach was used in \cite[Claim 4.4.7]{hicks2019wall}.

From this map we can construct weakly filtered\footnote{A weakly filtered map is allowed to decrease the filtration, see \cite[Definition A.1.3]{hicks2019wall}} homotopy equivalences from $\CF(K)$ to either $\CF(L^+)$ or $\CF(L^-)$. 
A similar story holds for Lagrangian cobordisms which are the suspension of an exact homotopy.

Problematically, weakly filtered maps \emph{do not} preserve the property of unobstructedness. If a weakly filtered map decreases the filtration by at most $\lambda$, it only defines a pushforward map on bounding cochains whose valuations are at least $\lambda$. In particular, if $\CF(L^+), \CF(L^-)$ are weakly filtered homotopy equivalent, and $\CF(L^+)$ is unobstructed, then $\CF(L^-)$ is unobstructed provided that the bounding cochain for $L^+$ has sufficiently large valuation.
 Therefore, each Lagrangian cobordism coming from a Lagrangian surgery trace $K: L^+\rightsquigarrow L^-$ only identifies a subset of the Maurer-Cartan elements of $L^+$ with those of $L^-$.
To see which bounding cochains are identified: every bounding cochain $\mathfrak b$ on $K$ induces bounding cochains on $\pi_*^\pm \mathfrak b$ on $L^\pm$, which sets up a correspondence between subsets of the Maurer Cartan spaces of $L^+$ and $L^-$. The objects $(L^\pm, \mathfrak b^\pm)$ are expected to be isomorphic in the Fukaya category, with the homotopy equivalences from $\CF_{bot}(K, \mathfrak b)$ to $\CF(L^\pm, \mathfrak b^\pm)$ providing a ``continuation map''.
Note that not every bounding cochain on $L^\pm$ can be achieved in this manner. This observation was employed in \cite{hicks2019wall} to recover wall-crossing formula for Lagrangian tori via Lagrangian cobordisms related to the mutation operation. 

More generally, given a Lagrangian cobordism $K: L^+\rightsquigarrow L^-$, we expect that the decomposition from \cref{thm:cobordismsarebottleneckedsurgery} will provide us with an algorithm which:
\begin{itemize}
    \item determines if a bounding cochain exists for $K$. For simplicity, write $K=K_n\circ \cdots \circ K_1$ where each $K_i: L_{i}\rightsquigarrow L_{i-1}$ is either a Lagrangian surgery trace or suspension of an exact homotopy. Let $\mathcal MC(L)$ denote the set of bounding cochains on $L$. Then the discussion above produces for each $i$ a subset $B_i\subset \mathcal MC(L_i)\times \mathcal MC(L_{i-1})$ of bounding cochains restricted from $K_i$. The fiber product \[B_1\times_{\mathcal MC(L_1)}B_2\times_{\mathcal MC(L_2)}\cdots \times_{\mathcal MC(L_{n-1})}B_n\] then describes the bounding cochains on $K$.
    \item computes the homotopy equivalence between $\CF(L^+)$ and $\CF(L^-)$ if they are unobstructed. This homotopy equivalence is the composition of the homotopy equivalences induced by each suspension or surgery trace cobordism.
\end{itemize}

The construction of such an algorithm is beyond the scope of this paper. We therefore conclude with two sample computations demonstrating how the proposed algorithm might work in practice. In \cref{subsec:initialcomputation}, we look at which choices of surgery parameters make the Lagrangian cobordism constructed in \cref{exam:obstructedLagrangian} unobstructed. In \cref{subsec:samplecomputation}, we look at a single Lagrangian surgery trace $K^{0,2}_{A,B}$. When $A>B$, we show that this Lagrangian cobordism is unobstructed and gives a continuation map between its ends (when equipped with appropriate bounding cochains). When $A<B$, we show that the Lagrangian cobordism is obstructed (and furthermore that the ends are not isomorphic objects in the Fukaya category for any choice of bounding cochain).
\subsection{A return to \cref{exam:obstructedLagrangian}}
\label{subsec:initialcomputation}
We now return to a variation of \cref{exam:obstructedLagrangian}. In particular, we discuss the conditions for when antisurgery followed by surgery give an unobstructed Lagrangian cobordism, and why when unobstructed these Lagrangian cobordisms are mapping cocylinders.

Consider the Lagrangian cobordism $K_{A_+, A_-}$ which is obtained concatenating the following Lagrangian submanifolds:
\begin{itemize} 
    \item A Lagrangian anti-surgery trace $(K_{A_+}^{1,1})^{-1} : S^1_{E_+}\sqcup S^1_{-E_+} \rightsquigarrow L_E$, where $L_E$ is an immersed Lagrangian submanifold in $T^*S^1$.
    \item A Lagrangian surgery trace $K_{A_-}^{1,1}: L_E\rightsquigarrow S^1_{E_-}\sqcup S^1_{-E_-}$.
\end{itemize}
The slices and shadow of the Lagrangian cobordism $K_{A_+, A_-}$ are drawn in  \cref{fig:simplecomputation} (a) and (b).
The relations between $E_+, E_-, A_+, A_-$ and $E$ are 
\begin{align*}
    E-2E_+=A_+ && E-2E_-=A_-.
\end{align*}
The immersed Floer cohomology of $K_{A_+, A_-}$ is a deformation of the Morse cohomology of $K_{A_+, A_-}$.
We take the standard Morse  function for the ends of the cobordism, which is generated on critical points:
\[\CM(S^1_{E_\pm}\sqcup S^1_{-E_\pm})= \Lambda\langle e_0^\pm, e_1^\pm, x_0^\pm, x_1^\pm\rangle.\]
$K_{A_+, A_-}$ topologically is two pairs of pants sharing a common neck; we take the our Morse function to be a perturbation of the Morse-Bott function which has maximums along the ends, and a minimally graded $S^1$ along the common neck. 
The generators and 0-dimensional flow lines contributing to $m^1: \CM(K_{A_+, A_-})\to \CM(K_{A_+, A_-})$ are drawn in \cref{fig:simplecomputation} (c).
Pairs of cancelling flow lines are indicated with dotted lines. 
\begin{figure}
    \centering
    \begin{tikzpicture}

\fill[brown!20, rounded corners]  (0,-10) rectangle (-2,-7.5);
\fill[brown!20, rounded corners]  (7.5,-10) rectangle (5.5,-7.5);
\fill[brown!20, rounded corners] (7.5,-14.5) rectangle (5.5,-12);
\fill[brown!20, rounded corners]  (0,-14.5) rectangle (-2,-12);
\begin{scope}[shift={(1,-1)}]
\fill[gray!20]  (-3,2.5) rectangle (-1.6,-2.5);
\draw[thick](-3,-0.5) -- (-1.6,-0.5) ;
\draw[thick] (-3,0.5) -- (-1.6,0.5);
\node at (-3,0) {$\uparrow$};
\node at (-1.6,0) {$\uparrow$};
\draw (-3,2.5) -- (-3,-2.5) (-1.6,-2.5) -- (-1.6,2.5);
\draw[dotted] (-3,2.5) -- (-1.6,2.5) (-3,-2.5) -- (-1.6,-2.5);
\end{scope}

\begin{scope}[shift={(6,-1)}]
\fill[gray!20]  (-3,2.5) rectangle (-1.6,-2.5);
\draw[thick, brown] (-1.6,1.3) .. controls (-1.8,1.3) and (-1.8,1.9) .. (-2.2,1.9) .. controls (-2.3,1.9) and (-2.3,-1.9) .. (-2.4,-1.9) .. controls (-2.8,-1.9) and (-2.8,-1.3) .. (-3,-1.3);

\draw (-3,2.5) -- (-3,-2.5) (-1.6,-2.5) -- (-1.6,2.5);
\draw[dotted] (-3,2.5) -- (-1.6,2.5) (-3,-2.5) -- (-1.6,-2.5);
\node at (-3,0) {$\uparrow$};
\node at (-1.6,0) {$\uparrow$};

\draw[thick, brown] (-1.62,-1.3) .. controls (-1.82,-1.3) and (-1.82,-1.9) .. (-2.2,-1.9) .. controls (-2.3,-1.9) and (-2.3,1.9) .. (-2.4,1.9) .. controls (-2.8,1.9) and (-2.82,1.3) .. (-3.02,1.3);

\end{scope}

\begin{scope}[shift={(3,-1)}]
\fill[gray!20]  (-3,2.5) rectangle (-1.6,-2.5);
\node at (-3,0) {$\uparrow$};
\node at (-1.6,0) {$\uparrow$};
\draw (-3,2.5) -- (-3,-2.5) (-1.6,-2.5) -- (-1.6,2.5);
\draw[dotted] (-3,2.5) -- (-1.6,2.5) (-3,-2.5) -- (-1.6,-2.5);
\draw[thick, brown] (-3,0.5) .. controls (-2.8,0.5) and (-2.8,1.3) .. (-2.4,1.3) .. controls (-2.3,1.3) and (-2.3,-1.3) .. (-2.2,-1.3) .. controls (-1.8,-1.3) and (-1.8,-0.5) .. (-1.6,-0.5);
\draw[thick, brown] (-3,-0.5) .. controls (-2.8,-0.5) and (-2.8,-1.3) .. (-2.4,-1.3) .. controls (-2.3,-1.3) and (-2.3,1.3) .. (-2.2,1.3) .. controls (-1.8,1.3) and (-1.8,0.5) .. (-1.6,0.5);
\end{scope}

\begin{scope}[shift={(8.5,-1)}]
\fill[gray!20]  (-3,2.5) rectangle (-1.6,-2.5);
\node at (-3,0) {$\uparrow$};
\draw[thick](-3,1.5) -- (-1.6,1.5);
\draw (-3,2.5) -- (-3,-2.5) (-1.6,-2.5) -- (-1.6,2.5);
\draw[dotted] (-3,2.5) -- (-1.6,2.5) (-3,-2.5) -- (-1.6,-2.5);
\draw[thick] (-3,-1.5) -- (-1.6,-1.5) ;
\node at (-1.6,0) {$\uparrow$};\end{scope}

\fill[gray!20]  (-2,-4) rectangle (7.25,-7);

\draw[dashed, <-] (-1.25,-3.65) -- (-1.25,-7);
\draw[dashed, <-] (0.75,-3.65) -- (0.75,-7);
\draw[dashed, <-] (3.75,-3.65) -- (3.75,-7);
\draw[dashed, <-]  (6.3,-3.75) -- (6.3,-7);

\draw[<-] (-1,1.7) .. controls (-0.75,2.2) and (0.25,2.2) .. (0.5,1.7);
\draw[<-] (1,1.7) .. controls (1.25,2.2) and (2.25,2.2) .. (2.5,1.7);
\draw[<-] (3,1.7) .. controls (3.5,2.2) and (5.5,2.2) .. (6,1.7);
\draw[<-]   ;
\node[above] at (-0.25,2.2) {\tiny Surgery};
\node[above] at (1.75,2.2) {\tiny Exact Homotopy};
\node[above] at (4.5,2.2) {\tiny Anti-Surgery};
\node at (5.25,-3.9) {$\mathbb C$};

\node at (7.5,-2.5) {$-E_+/\pi$};
\node at (-3,-1.5) {$-E_-/\pi$};
\node at (-3,-0.5) {$E_-/\pi$};
\node at (7.5,0.5) {$E_+/\pi$};

\begin{scope}[shift={(1.7,-2.8)}]

\node at (6,1.75) {$T^*S^1$};

\end{scope}

\node[teal, circle, scale=.2, fill] at (0.7,-1) {};
\node[teal, circle, scale=.2, fill] at (3.75,-1) {};

\fill[brown!20] (3.75,-5.5) .. controls (3.25,-5) and (1,-5) .. (0.5,-5) .. controls (0,-5) and (-0.5,-5.5) .. (-1.25,-5.5) .. controls (-0.75,-5.5) and (0,-6) .. (0.5,-6) .. controls (1,-6) and (3.25,-6) .. (3.75,-5.5);
\fill[brown!20] (3.75,-5.5) .. controls (4.25,-6) and (4.5,-6) .. (4.75,-6) .. controls (5,-6) and (5.25,-5.5) .. (5.75,-5.5) .. controls (5.25,-5.5) and (5,-5) .. (4.75,-5) .. controls (4.5,-5) and (4.25,-5) .. (3.75,-5.5);

\draw (-2,-5.5) .. controls (-1.5,-5.5) and (-1.5,-5.5) .. (-1.25,-5.5) .. controls (-0.75,-5.5) and (0,-6) .. (0.5,-6) .. controls (1,-6) and (3.25,-6) .. (3.75,-5.5) .. controls (3.25,-5) and (1,-5) .. (0.5,-5) .. controls (0,-5) and (-0.75,-5.5) .. (-1.25,-5.5);
\draw (5.75,-5.5) .. controls (5.25,-5.5) and (5,-6) .. (4.75,-6) .. controls (4.5,-6) and (4.25,-6) .. (3.75,-5.5) .. controls (4.25,-5) and (4.5,-5) .. (4.75,-5) .. controls (5,-5) and (5.25,-5.5) .. (5.75,-5.5);
\draw[pattern=north west lines, pattern color=red] (0.5,-5) .. controls (0,-5) and (0,-6) .. (0.5,-6) .. controls (1,-6) and (3.25,-6) .. (3.75,-5.5) .. controls (3.25,-5) and (1,-5) .. (0.5,-5);
\draw[pattern=north west lines, pattern color=teal] (4.75,-6) .. controls (4.5,-6) and (4.25,-6) .. (3.75,-5.5) .. controls (4.25,-5) and (4.5,-5) .. (4.75,-5) .. controls (5,-5) and (5,-6) .. (4.75,-6);
\draw (5.75,-5.5) -- (7.25,-5.5);
\node at (0.75,-5.5) {$A_-$};
\node at (4.5,-5.5) {$A_+$};

\begin{scope}[shift={(0,-0.5)}]

\node (v1) at (-1.5,-7.5) {$e_0^-$};
\node (v2) at (-1.5,-9) {$x_0^+$};
\node (v3) at (-0.5,-7.5) {$E_+^-$};
\node (v4) at (-0.5,-9) {$x_1^-$};
\node (v9) at (0.5,-9) {$y^-$};
\node (v5) at (6,-7.5) {$e_0^+$};
\node (v6) at (6,-9) {$x_0^+$};
\node (v7) at (7,-7.5) {$E_+^+$};
\node (v8) at (7,-9) {$x_1^+$};
\node (v11) at (1.5,-9) {$e^0$};
\node (v13) at (1.5,-10.5) {$x^0$};
\node (v12) at (3.75,-7.5) {$(q_+\to q_-)$};
\node (v14) at (3.5,-10.5) {$(q_+\to q_-)$};
\node (v10) at (5,-9) {$y^+$};
\draw[dotted]  (v1) edge[->,bend left] (v2);
\draw[dotted]  (v1) edge[->,bend right] (v2);
\draw[dotted]  (v3) edge[->,bend left] (v4);
\draw[dotted]  (v3) edge[->,bend right] (v4);
\draw[dotted]  (v5) edge[->,bend left] (v6);
\draw[dotted]  (v7) edge[->,bend left] (v8);
\draw[dotted]  (v5) edge[->,bend right] (v6);
\draw[dotted]  (v7) edge[->,bend right] (v8);
\draw  (v1) edge[->] (v9);
\draw  (v3) edge[->] (v9);
\draw  (v5) edge[->] (v10);
\draw  (v7) edge[->] (v10);
\draw  (v5) edge[->] (v11);
\draw  (v3) edge[->] (v11);
\draw  (v2) edge[->] (v13);
\draw  (v4) edge[->] (v13);
\draw  (v6) edge[->] (v13);
\draw  (v8) edge[->] (v13);

\draw[->,dotted, bend left =5]   (v9) edge (v13);
\draw[->,dotted, bend right=5]  (v9) edge (v13);
\draw[->,dotted, bend left =5]   (v11) edge (v13);
\draw[->,dotted, bend right=5]  (v11) edge (v13);
\draw[->,dotted, bend left =5]   (v10) edge (v13);
\draw[->,dotted, bend right=5]  (v10) edge (v13);
\end{scope}

\begin{scope}[shift={(0,-5)}]

\node (v1) at (-1.5,-7.5) {$e_0^-$};
\node (v2) at (-1.5,-9) {$x_0^+$};
\node (v3) at (-0.5,-7.5) {$E_+^-$};
\node (v4) at (-0.5,-9) {$x_1^-$};
\node (v9) at (0.5,-9) {$y^-$};
\node (v5) at (6,-7.5) {$e_0^+$};
\node (v6) at (6,-9) {$x_0^+$};
\node (v7) at (7,-7.5) {$E_+^+$};
\node (v8) at (7,-9) {$x_1^+$};
\node (v11) at (1.5,-9) {$e^0$};
\node (v13) at (1.5,-10.5) {$x^0$};
\node (v12) at (3.75,-7.5) {$(q_+\to q_-)$};
\node (v14) at (3.5,-10.5) {$(q_+\to q_-)$};
\node (v10) at (5,-9) {$y^+$};
\draw  (v12) edge[->, red] node [fill=white] {$T^{A_-}$}(v9);
\draw  (v12) edge[->, teal] node [fill=white] {$T^{A_+}$}(v10);
\draw  (v4) edge[->, red]node [fill=white] {$T^{A_-}$} (v14);
\draw  (v6) edge[->, teal]node [fill=white] {$T^{A_+}$} (v14);
\node[rounded corners, fill=teal!20] (v16) at (6,-11.5) {$m^0$};
\node[rounded corners, fill=red!20] (v15) at (0.5,-11.5) {$m^0$};
\draw  (v15) edge[->, red] node [red][fill=white] {$T^{A_-}$} (v14);
\draw  (v16) edge[->, teal] node [fill=white] {$T^{A_+}$}(v14);

\end{scope}

\node at (-4.5,-1) {(a) Slices};
\node at (-4.5,-5.5) {(b) Shadow};
\node at (-4.5,-9) {(c) Morse Differential};
\node at (-4.5,-14.5) {(d) Floer Differential};
\end{tikzpicture}     \caption{Anti-surgery followed by surgery.}
    \label{fig:simplecomputation}
\end{figure}
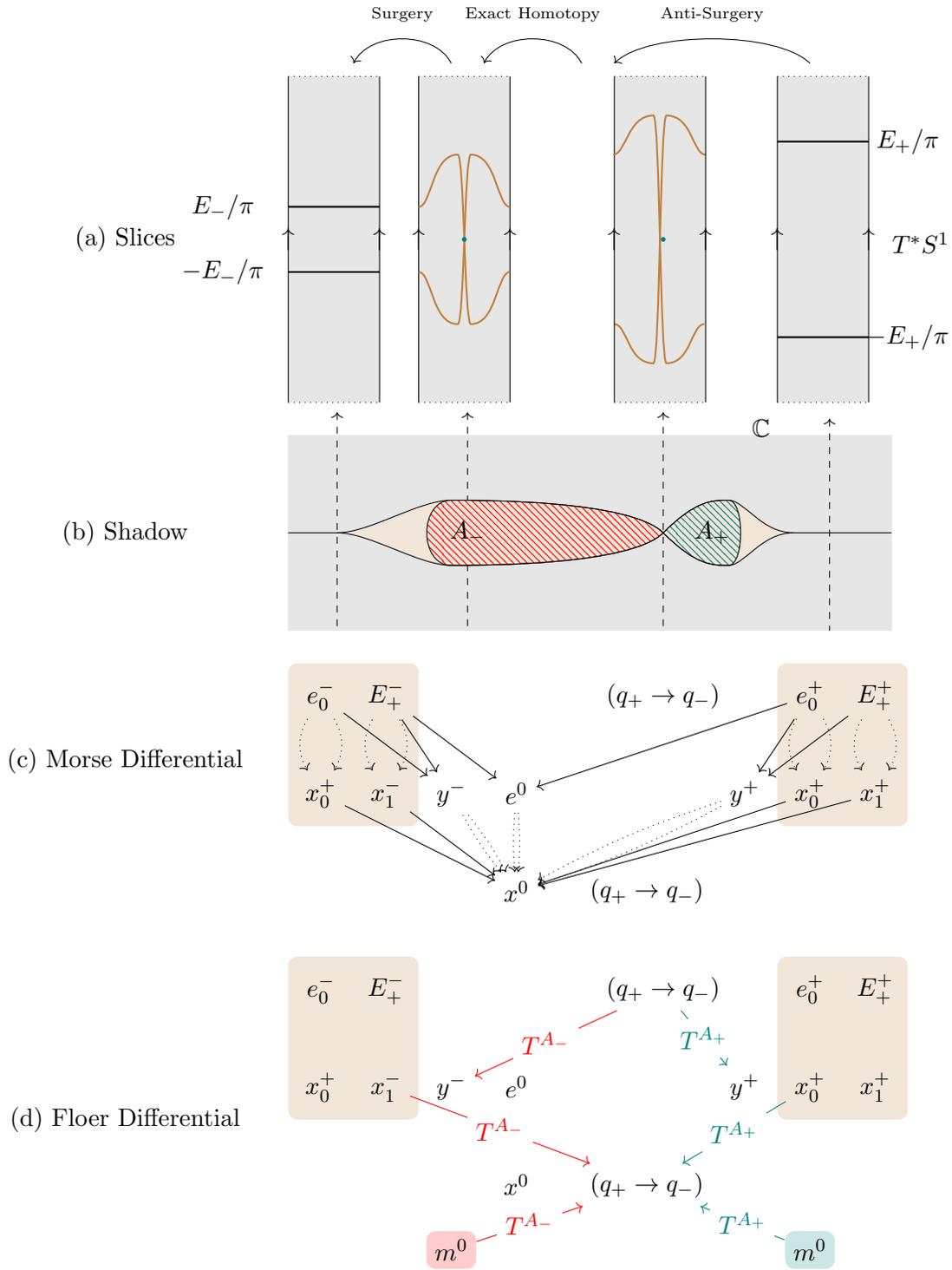
In addition to these Morse flow trees, the $A_\infty$ product structure on $\CF(K_{A_+, A_-})$ counts configurations of holomorphic teardrops.
By \cref{thm:teardropexistence} both the surgery and anti-surgery trace contribute holomorphic teardrops  with output on $(q_+\to q_-)$.
The lowest order contributions to $m^0$ and $m^1$ are listed in \cref{fig:simplecomputation} (d). 
The most important term is the curvature term,
\[m^0=(T^{A_+}-T^{A_-})(q^+\to q^-)+ \mathcal O(\min(A_+, A_-)).\]
We break into two cases. 
\subsubsection[Case 1: Obstructedness]{Case 1: $A_+\neq A_-$}
In the event where $A_+\neq A_-$, the $m^0: \Lambda \to \CF(K_{A_+, A_-})$ term is non-zero. Furthermore, since $\min(A_+, A_-)$ is the smallest area of a holomorphic teardrop, 
\[\val(\langle m^k(x_1, \ldots, x_k), (q_+\to q_-)\rangle) > \min(A_+, A_-)\]
for all $x_i$ satisfying $\val (x_i)>0$. 
It follows that there is no solution to the Maurer-Cartan equation, and $\CF(K_{A_+, A_-})$ is obstructed.
Furthermore, clearly $ S^1_{E_+}\sqcup S^1_{-E_+} $ is not isomorphic to  $S^1_{E_-}\sqcup S^1_{-E_-} $ as an object of the Fukaya category.
\subsubsection[Case 2: Continuation map]{Case 2: $A_+=A_-$}
The more interesting example is to consider is when $A_+=A_-$. 
In this setting, the curvature term vanishes, and $\CF(K_{A_+, A_-})$ is tautologically unobstructed.
A computation shows that the chain complex $\CF(K_{A_+, A_-})$ is a mapping cylinder; by using homotopy transfer theorem one can construct a $A_\infty$ morphism $\hat i^-: \CF( S^1_{E_-}\sqcup S^1_{-E_-} )\to \CF(K_{A_+, A_-})$.
The composition \[\Phi=\pi^+\circ \hat i^-:\CF( S^1_{E_-}\sqcup S^1_{-E_-} )\to \CF( S^1_{E_+}\sqcup S^1_{-E_+} )\]
is a continuation map, reflecting that $K_{A_+, A_-}$ should provide an equivalence of objects in the Fukaya category.

\subsection{Continuations, obstructions, and mapping cones}
\label{subsec:samplecomputation}
We now employ the algorithm proposed at the beginning of this section to compute a continuation map for an immersed Lagrangian cobordism.
Consider the double-section Lagrangian $L_{E_a}\subset T^*S^1$ discussed in \cref{subsubsec:runningexample}. 
Let $S_{\pm E}$ be the pair of sections  $\pm \frac{E}{2\pi}d\theta$ of $T^*S^1$.
When $0<2E<E_a$ there is a Lagrangian surgery trace cobordism $L_{E_a} \rightsquigarrow  S^1_{E}\sqcup S^1_{-E}$.
\begin{remark}
    The direction of the inequality   $0<2E<E_a$ is a bit subtle. At first glance, the inequality $0<E_a<2E$ appears natural because it is the one obtained by using the local surgery model associated with the Polterovich surgery of two Lagrangian submanifolds intersecting transversely at a single point. Because we apply \cref{prop:splittingX} to construct the Lagrangian surgery trace, the ends of the Lagrangian surgery trace (\cref{subsubsec:surgerytrace}) differ by a non-trivial amount of flux from the Polterovich surgery model. The flux swept out is opposite and greater than the area bounded by the surgery neck, leading to the inequality $0<2E<E_a$. We give a more detailed account in \cref{exam:fluxDirection}.
\end{remark}
One constructs the double bottleneck surgery trace (\cref{def:doubleBottleneckSurgeryTrace}) 
\[K_{A,B}^{0,2}:  L_{E_b}\rightsquigarrow S^1_{E}\sqcup S^1_{-E}\]
by subsequently applying a small exact homotopy, where the flux swept out on the positive end gives the relation  $0<E_b<E_a$. We will drop the superscript and from here on write $K_{A, B}$ for this Lagrangian cobordism.
Here, $A$ is the area of the small teardrop of the surgery, and $B$ is the area of the strip associated with the bottleneck; the slice above the second bottleneck is denoted $L_{E_a}$.
The relations between $A, B, E_a, E_b$ and $E$ are
\begin{align*}
    E_a-E_b=2B && E_a-2E=2A.
\end{align*}
See \cref{fig:bigdiagram} (a, b) for slice and shadow of the Lagrangian cobordism. 
We equip $K_{A, B}$ with an admissible Morse function; this has critical points $e_0, e_1, x_0, x_1$ near the negative end, $e_a, e_b, x_a, x_b$ corresponding to the Morse critical points of the slices $L_{E_a}$ and $L_{E_b}$, and an additional critical point $y$ from the surgery handle. 
The Lagrangian $K_{A, B}$ has two self-intersections; we call 3 generators of the Floer cohomology $(\pi \to 0)_a, (0\to \pi)_a$, and $(0\to \pi)_b$, and discard the fourth generator.

\begin{figure}
    \centering
    
\begin{tikzpicture}

\begin{scope}[shift={(-9,-12.5)},decoration={    markings,   , mark=at position 0.5 with {\arrow{>}}}]
\begin{scope}[decoration={    markings,   , mark=at position 0.5 with {\arrow{>}}}]

\fill[gray!20]  (-4,-3) rectangle (-2,0);
\draw[postaction={decorate}] (-4,0) -- (-4,-3);
\draw[postaction={decorate}](-2,0)--(-2,-3);
\draw[dashed] (-4,0) -- (-2,0)  (-2,-3) -- (-4,-3);

\node[circle, fill=black, scale=.25] at (-3.5,-1) {};
\node[circle, fill=black, scale=.25] at (-2.5,-1) {};
\node[circle, fill=black, scale=.25] at (-3.5,-2) {};
\node[circle, fill=black, scale=.25] at (-2.5,-2) {};
\node[above] at (-3.5,-1) {$e_0$};
\node[above] at (-2.5,-1) {$x_0$};
\node[below] at (-3.5,-2) {$e_1$};
\node[below] at (-2.5,-2) {$x_1$};
\end{scope}
\begin{scope}[shift={(5.75,0)}]

\fill[gray!20]  (-4,-3) rectangle (-2,0);
\draw[postaction={decorate}] (-4,0) -- (-4,-3);
\draw[postaction={decorate}](-2,0)--(-2,-3);
\draw[dashed] (-4,0) -- (-2,0)  (-2,-3) -- (-4,-3);

\end{scope}
\begin{scope}[shift={(2.75,0)}]

\fill[gray!20]  (-4,-3) rectangle (-2,0);
\draw[postaction={decorate}] (-4,0) -- (-4,-3);
\draw[postaction={decorate}](-2,0)--(-2,-3);
\draw[dashed] (-4,0) -- (-2,0)  (-2,-3) -- (-4,-3);

\end{scope}

\draw[thick] (-4,-2) -- (-2,-2);
\draw[thick] (-4,-1) -- (-2,-1);
\fill[gray!10]  (-4,-4) rectangle (4,-7);

\draw[fill=brown!20,thick] (-3,-5.5) .. controls (-2.5,-5.5) and (-2,-6) .. (-1.5,-6) .. controls (-1,-6) and (1,-6) .. (1.5,-5.5) .. controls (1,-5) and (-1,-5) .. (-1.5,-5) .. controls (-2,-5) and (-2.5,-5.5) .. (-3,-5.5);
\draw[fill=brown!20,thick] (4,-5.5) .. controls (3.5,-5) and (2,-5) .. (1.5,-5.5) .. controls (2,-6) and (3.5,-6) .. (4,-5.5);
\draw[pattern=north west lines, pattern color=orange] (4,-5.5) .. controls (3.5,-5) and (2,-5) .. (1.5,-5.5) .. controls (2,-6) and (3.5,-6) .. (4,-5.5);
\draw (-4,-5.5) -- (-3,-5.5);
\draw[pattern=north west lines, pattern color=orange] (1.5,-5.5) .. controls (1,-6) and (-1.5,-6) .. (-1.5,-5.5) .. controls (-1.5,-5) and (1,-5) .. (1.5,-5.5);

\node at (-4.5,-2) {$S^1_{-E}$};
\node at (-4.5,-1) {$S^1_{E}$};

\node[fill, scale=.5, circle] at (-1.5,-5.5) {};
\node at (-3,-0.5) {$\mathbb{C}^*$};
\node at (0.25,-0.25) {$\mathbb{C}^*$};
\node[left] at (-1.5,-5.5) {$y$};
\node at (0,-5.5) {$A$};
\node at (3.25,-5.5) {$B$};

\begin{scope}[shift={(-0.5,0)}]

\clip  (1.75,-5) rectangle (4.75,-6.25);
\draw[red, thick] (-1,-6) .. controls (-0.5,-6) and (1.5,-6) .. (2,-5.5) .. controls (2.5,-5) and (4,-5) .. (4.5,-5.5);
\draw[blue, thick] (4.5,-5.5) .. controls (4,-6) and (2.5,-6) .. (2,-5.5) .. controls (1.5,-5) and (-0.5,-5) .. (-1,-5);

\end{scope}

\begin{scope}[shift={(-2.25,-0.5)}]
\node[below right] at (1,-0.5) {$e$};
\node[above right] at (1,-1.5) {$x$};
\node[circle, fill=black, scale=.25] at (1,-0.5) {};
\node[circle, fill=black, scale=.25] at (1,-1.5) {};
\draw[thick] (1,-1.5) .. controls (1.5,-1.5) and (1.5,-2.25) .. (1.75,-2.25) .. controls (2,-2.25) and (2,0.25) .. (2.25,0.25) .. controls (2.5,0.25) and (2.5,-0.5) .. (3,-0.5);
\draw[thick] (1,-0.5) .. controls (1.5,-0.5) and (1.5,0.25) .. (1.75,0.25) .. controls (2,0.25) and (2,-2.25) .. (2.25,-2.25) .. controls (2.5,-2.25) and (2.5,-1.5) .. (3,-1.5);

\clip  (1.8,-0.7) rectangle (2.2,-1.3);
\draw[red, thick] (1.75,-2.25) .. controls (2,-2.25) and (2,0.25) .. (2.25,0.25);
\draw[blue, thick] (1.75,0.25) .. controls (2,0.25) and (2,-2.25) .. (2.25,-2.25);
\end{scope}
\begin{scope}[shift={(0.75,-0.5)}]
\node[below right] at (1,-0.5) {$e'$};
\node[above right] at (1,-1.5) {$x'$};
\node[circle, fill=black, scale=.25] at (1,-0.5) {};
\node[circle, fill=black, scale=.25] at (1,-1.5) {};
\draw[thick] (1,-1.5) .. controls (1.5,-1.5) and (1.5,-2) .. (1.75,-2) .. controls (2,-2) and (2,0) .. (2.25,0) .. controls (2.5,0) and (2.5,-0.5) .. (3,-0.5);
\draw[thick] (1,-0.5) .. controls (1.5,-0.5) and (1.5,0) .. (1.75,0) .. controls (2,0) and (2,-2) .. (2.25,-2) .. controls (2.5,-2) and (2.5,-1.5) .. (3,-1.5);

\clip  (1.8,-0.7) rectangle (2.2,-1.3);
\draw[red, thick] (1.75,-2) .. controls (2,-2) and (2,0) .. (2.25,0);
\draw[blue, thick] (1.75,0) .. controls (2,0) and (2,-2) .. (2.25,-2);
\end{scope}

\end{scope}

\draw[dashed, ->] (-12.5,-19.5) -- (-12.5,-16);
\draw[dashed, ->] (-7.5,-19.5) .. controls (-7.5,-19) and (-7.5,-17) .. (-7.5,-16.5) .. controls (-7.5,-16.25) and (-7.5,-16.25) .. (-7.5,-16.25) .. controls (-7.5,-16) and (-7.5,-16) .. (-7.75,-16) .. controls (-8,-16) and (-8.75,-16) .. (-9,-16) .. controls (-9.25,-16) and (-9.25,-16) .. (-9.25,-15.75);
\draw[dashed, ->] (-5,-19.5) .. controls (-5,-19) and (-5,-17) .. (-5,-16.75) .. controls (-5,-16.5) and (-5,-16) .. (-5.5,-16) .. controls (-5,-16) and (-5.75,-16) .. (-6,-16) .. controls (-6.25,-16) and (-6.25,-16) .. (-6.25,-16) .. controls (-6.5,-16) and (-6.5,-16) .. (-6.5,-15.75);

\begin{scope}[rotate=-90, shift={(21.5,-11.5)},decoration={    markings,    mark=at position 0.33 with {\arrow{>}} , mark=at position 0.66 with {\arrow{>}}}]

\draw[postaction={decorate}] (0.25,-1) .. controls (0.25,0) and (0.5,1) .. (1,1);
\draw[postaction={decorate}] (1.75,-1) .. controls (1.75,0) and (1.5,1) .. (1,1);
\draw[postaction={decorate}] (1,1) .. controls (1,2.5) and (0,3.25) .. (0,4);
\draw[postaction={decorate}] (-1,1) .. controls (-1,2) and (0,3.25) .. (0,4);
\draw [postaction={decorate}](-0.25,-1) .. controls (-0.25,0.5) and (0,3) .. (0,4);
\draw[postaction={decorate}] (2.25,-1) .. controls (2.25,1.5) and (2.5,2.25) .. (3,2.75);
\draw[postaction={decorate}] (-1,2.75) .. controls (-0.5,3.25) and (0,3.5) .. (0,4);
\draw[postaction={decorate}] (1.75,-1) .. controls (1.75,0.75) and (2,2.75) .. (2,4);
\draw[postaction={decorate}] (0.25,-1) -- (-0.25,-1);
\draw[postaction={decorate}] (1.75,-1)--(2.25,-1);
\draw[postaction={decorate}] (1.75,-1)--(1.5,-1)    (2.5,-1)--(2.25,-1) ;
\draw[postaction={decorate}] (0.25,-1) -- (0.5,-1) (-0.5,-1) -- (-0.25,-1);

\begin{scope}[shift={(0,0)}]

\draw[purple] (-0.5,-1) .. controls (-0.5,0) and (-0.5,1) .. (-1,1) (0.5,-1) .. controls (0.5,0) and (0.5,1) .. (1,1);
\end{scope}
\begin{scope}[shift={(3,0)}]

\draw[pink] (-1.5,-1) .. controls (-1.5,0) and (-1.5,1) .. (-2,1) (-0.5,-1) .. controls (-0.5,0) and (-0.5,1) .. (0,1);
\end{scope}
\draw[green] (-1,1) -- (-1,6.5) (3,1) -- (3,6.5);
\node[left] at (0.25,-1) {$e_0$};
\node[left] at (-0.25,-1) {$x_0$};
\node[left] at (2.25,-1) {$x_1$};
\node[left] at (1.75,-1) {$e_1$};
\node[left] at (1,1) {$y$};
\node[left] at (-1,1) {$y$};

\node[fill=black, circle, scale=.25] at (0.25,-1) {};
\node[fill=black, circle, scale=.25] at (-0.25,-1) {};
\node[fill=black, circle, scale=.25] at (2.25,-1) {};
\node[fill=black, circle, scale=.25] at (1.75,-1) {};
\node[fill=black, circle, scale=.25] at (1,1) {};
\node[fill=black, circle, scale=.25] at (-1,1) {};

\begin{scope}[]

\node[above right] at (2,4) {$e_a$};
\node[above right] at (0,4) {$x_a$};
\node[fill=black, circle, scale=.25] at (2,4) {};
\node[fill=black, circle, scale=.25] at (0,4) {};
\draw[postaction={decorate}] (2,4) -- (0,4);
\draw[postaction={decorate}] (2,4) -- (3,4) (-1,4) -- (0,4);
\end{scope}
\begin{scope}[shift={(0,2.5)}]

\node[right] at (2,4) {$e_b$};
\node[right] at (0,4) {$x_b$};
\node[fill=black, circle, scale=.25] at (2,4) {};
\node[fill=black, circle, scale=.25] at (0,4) {};
\draw[postaction={decorate}] (2,4) -- (0,4);
\draw[postaction={decorate}] (2,4) -- (3,4) (-1,4) -- (0,4);
\end{scope}
\draw[postaction={decorate}] (0,6.5) -- (0,4);
\draw[postaction={decorate}] (2,6.5) -- (2,4);

\draw[orange, dashed, thick] (1,4) -- (1,1) (-1,1) -- (-1,4);
\node at (1,4) {$\times$};
\node at (-1,4) {$\times$};
\node at (1,6.5) {$\times$};
\node at (-1,6.5) {$\times$};
\node[right] at (1,4) {$1_a$};
\node[right] at (-1,4) {$0_a$};
\node[right] at (1,6.5) {$1_b$};
\node[right] at (-1,6.5) {$0_b$};
\end{scope}

\node[rounded corners, fill=red!20] (v13) at (-5,-29.5) {$m^0$};
    
\fill[brown!20, rounded corners]  (-14.5,-26) rectangle (-12.5,-28.5);
\fill[brown!20, rounded corners]  (-10.5,-26) rectangle (-7.5,-30);
\fill[brown!20, rounded corners]  (-7,-26) rectangle (-4,-28.5);

    \node (v1) at (-14,-26.5) {$e_0$};
    \node (v2) at (-14,-28) {$x_0$};
    \node (v3) at (-13,-26.5) {$e_1$};
    \node (v4) at (-13,-28) {$x_1$};
    \node[fill=blue!20, rounded corners] (v6) at (-11,-28) {$y$};
    \node (v7) at (-9,-28) {$e_a$};
    \node at (-10,-29.5) {$x_a$};
    \node (v5) at (-10,-29.5) {};
    \node (v11) at (-8.5,-29.5) {$(0\to \pi)_a$};
    \node (v12) at (-9,-26.5) {$(\pi \to 0)_a$};
    \node (v10) at (-5,-28) {$(0\to \pi)_b$};
    \node (v9) at (-5,-26.5) {$e_b$};
    \node (v8) at (-6.5,-28) {$x_b$};
    \draw  (v1) edge[bend left=10, ->,dotted] (v2);
    \draw  (v1) edge[bend right=10, ->,dotted] (v2);
    \draw  (v3) edge[bend right=10, ->,dotted] (v4);
    \draw  (v3) edge[bend left=10, ->,dotted] (v4);
    \draw  (v4) edge[->] (v5);
    \draw  (v2) edge[->] (v5);
    \draw  (v6) edge[bend left=10, ->,dotted] (v5);
    \draw  (v6) edge[bend right=10, ->,dotted] (v5);
    \draw  (v7) edge[bend right=10, ->,dotted] (v5);
    \draw  (v7) edge[bend left=10, ->,dotted] (v5);
    \draw  (v3) edge[->] (v7);
    \draw  (v1) edge[->] (v6);
    \draw  (v3) edge[->] (v6);
    \draw  (v8) edge[->] (v5);
    \draw  (v9) edge[->] (v7);
    \draw  (v10) edge[->, purple] node[midway, fill=white]{\scriptsize$T^B$} (v11);
    \draw  (v12) edge[->,red]  node[midway, fill=brown!20]{\scriptsize$T^A$}(v6);
    \draw  (v4) edge[red, ->]  node[midway, fill=white]{\scriptsize$T^A$}(v11);
    \draw  (v13) edge[red,->]  node[midway, fill=white]{\scriptsize$T^A$}(v11);

\draw  (v12) edge[bend left=10,red,dotted] (v10);
\draw  (v12) edge[bend right=10,red,dotted] (v10);
\node at (-12,-12) {$S^1_E\sqcup S^1_{-E}$};
\node at (-9.25,-12) {$L_{E_a}$};
\node at (-6.25,-12) {$L_{E_b}$};
\node[red, fill=brown!20] at (-6,-27.5) {\scriptsize$\pm T^{B+E_b}$};
\draw[dashed] (-12.5,-19.5) .. controls (-12.5,-20) and (-12.5,-24.5) .. (-12.5,-25) .. controls (-12.5,-25.5) and (-12.5,-25.5) .. (-12.5,-25.5) .. controls (-12.5,-25.5) and (-12.5,-25.5) .. (-13.5,-26);

\draw[dashed] (-7.5,-19.5) .. controls (-7.5,-20) and (-7.5,-24.5) .. (-7.5,-25) .. controls (-7.5,-25) and (-7.5,-25.5) .. (-8.25,-25.5) .. controls (-9,-25.5) and (-9,-25.5) .. (-9,-26);
\draw[dashed] (-5,-19.5) .. controls (-5,-20) and (-5,-24.5) .. (-5,-25) .. controls (-5,-25.5) and (-5,-25.5) .. (-5.5,-25.5) .. controls (-6,-25.5) and (-6,-25.5) .. (-6,-26);
\node at (-15,-26) {Index};
\node at (-15,-26.5) {0};
\node at (-15,-28) {1};
\node at (-15,-29.5) {2};
\node[left] at (-14.3207,-15) {(a) Slices};
\node[left] at (-14,-19) {(b) Shadow};
\node[left] at (-14,-24.5) {(c) Morse Function};
\node[left] at (-14,-30) {(d) Floer Complex};
\end{tikzpicture}     \caption{Slices, shadow, Morse function and Floer complex of the surgery trace cobordism.}
    \label{fig:bigdiagram}
\end{figure}
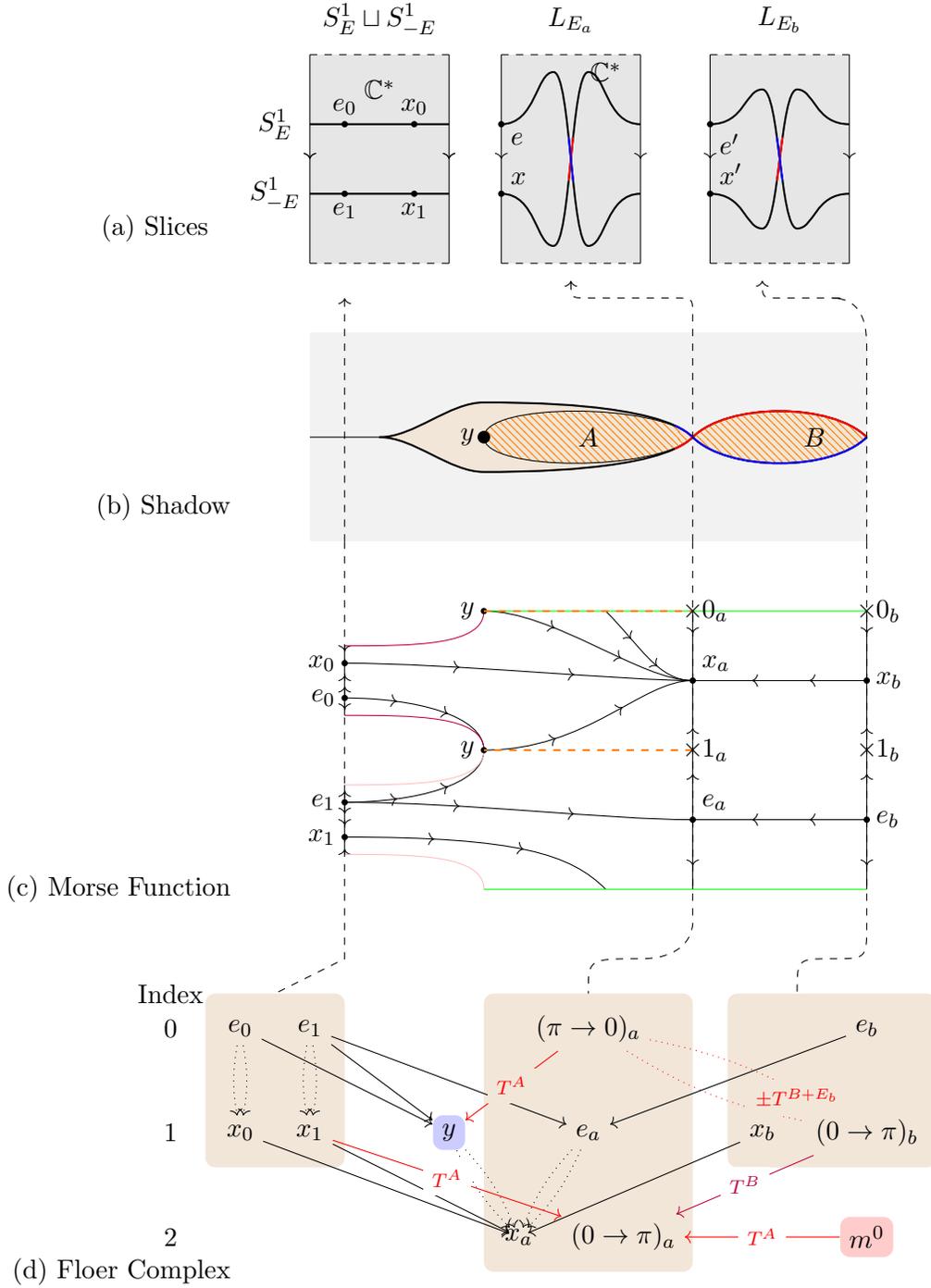

\subsubsection{Computation of low order products}
As it is difficult to compute the full product structure for Floer cochains of $K_{A, B}$, we restrict ourselves to computing smallest-order contributions to the $m^0$ and $m^1$ terms on $\CF_{bot}(K_{A, B})$. 
These lowest order computations are sufficient to prove unobstructedness of Floer groups by standard arguments using the filtration of $\CF_{bot}(K_{A, B})$.
\begin{itemize}
    \item The differential on the Morse complex of $\CM(K_{A, B})$ can be fully computed. The Morse flow-lines are drawn in \cref{fig:bigdiagram}(c). The differential is represented by black dotted and solid arrows in \cref{fig:bigdiagram}(d). The dotted arrows have canceling contributions.
    \item By \cref{claim:teardropm0}, there exists a holomorphic teardrop which can either be considered as having input on  $(0\to \pi)_a$ or output on $(\pi \to 0)_a$ with area $A$. :
    \begin{itemize}
        \item The teardrop with output on $(0\to \pi)_a$ appears in the $m^0$ term in the $A_\infty$ structure 
        \[m^0 = T^A \cdot (0\to \pi)_a+\mathcal (O_A)\]
        This is represented with the highlighted red term in \cref{fig:bigdiagram}(d). Additionally, this teardrop contributes to the differential, 
        \[\langle m^1(x_1), (0\to \pi)_a\rangle = T^A+ \mathcal O(A).\]
        \item The teardrop with input on $(\pi \to 0)_a$ gives a contribution to the differential
        \[\langle m^1(\pi \to 0)_a, y\rangle= T^A + \mathcal O(A).\] 
    \end{itemize} 
    \item The holomorphic strips contributing to the differential on $\CF_{bot}(K_{A, B})$ arise from the strips which appear on a double bottleneck.
    \[\langle m^1(\pi \to 0)_a, (0\to \pi)_b\rangle= T^{B+E_b}-T^{B+E_b}.\] 
    The double bottleneck also has a holomorphic strip pairing the generator which is ``doubled'' with maximal and minimal grading from the base:
    \[\langle m^1(0\to \pi)_b, (0\to \pi)_a\rangle= T^B + \mathcal O(B).\] 
\end{itemize}
These contributions are represented by red arrows in \cref{fig:bigdiagram}(d).
We obtain two cases which we consider separately: when $A>B$, and $A<B$. We discuss the Floer theoretic and geometric implications below.
\subsubsection[Continuation maps and Unobstructedness: Case I]{Continuation maps and Unobstructedness: $A>B$}
We first note that whenever $A>B$, the curvature term $m^0$ can be cancelled out at lowest order by the differential, as
\[m^1(T^{A-B}(0\to \pi)_b)=T^A (0\to \pi)_a+\mathcal O(A).\]
We can build a bounding cochain order by order using the filtration on $\CF_{bot}(K_{A, B})$, whose first term is 
\[\mathfrak b_{A, B}:= T^{A-B}(0\to \pi)_b+ \mathcal O(A-B) \in \mathcal MC(K_{A, B}).\]
This gives us a sufficient condition for this Lagrangian cobordism to be unobstructed: the flux swept out by the exact homotopy as defined by the double bottleneck (measured by the area $B$ of the holomorphic strip) must not be greater than the area of the surgery neck taken (as measured by the holomorphic teardrop $A$).
Since $\pi^+: \CF_{bot}(K_{A, B})\to \CF(L_{E_b})$ is an $A_\infty$ homomorphism, we obtain a bounding cochain $\pi^+_*\mathfrak b_{A, B}$ for $\CF(L_{E_b})$. We can verify this computation by checking the Lagrangian intersection Floer cohomology against a test Lagrangian $S^1_{E'}$.
Recall from \cref{subsubsec:runningexample}, the complex $\CF((L_{E_b},\pi^+_*\mathfrak b_{A, B}), S^1_{E'})$ is acyclic whenever the lowest order term of $\pi^+_*(\mathfrak b_{A, B})$ is not $T^{\frac{E_b}{2}-E'}$. In this case, the lowest order term of $\mathfrak b$ is
\[
    T^{A-B}=T^{\frac{E_a-2E}{2}-\frac{E_a-E_b}{2}}=T^{\frac{E_b}{2}-E}.
\]
So at lowest order, $\CF((L_{E_b},\pi^+_*\mathfrak b_{A, B}), S^1_{E'} )$ has nontrivial homology if and only if $E=E'$.

The computation shows that the ends of  $K_{A, B}$ cannot be distinguished at lowest order by testing them against objects $S^1_{E'}$ in the Fukaya category.
We now show that the ends have the same immersed Lagrangian Floer cohomology by showing that $(K_{A, B}, \mathfrak b)$ is a mapping cocylinder, providing a homotopy equivalence between the Floer cohomology of its ends. This will construct a continuation map between the ends of the cobordism. 
First, we note that the  projection $\pi^+: \CF_{bot}(K_{A, B}, \mathfrak b)\to \CF(L_{E_b})$ is a homotopy equivalence.
The homotopy can be described on generators as follows: let 
\begin{align*}
    \mathcal I^-=&\{e_0, x_0, e_1, x_1\}\\
    \mathcal I^0=&\{y, e_a, x_a, (0\to \pi)_a \}\\
    \mathcal I^+=&\{(\pi \to 0)_a, e_b, x_b, (0\to \pi)_b\}.
\end{align*}
 Let $\CF_{bot}(K_{A, B})|_{\mathcal I^*}$ be the subspace spanned by the appropriate set of generators. This is generally not an $A_\infty$ ideal, however
 \[\CF_{bot}(K_{A, B})|_{\mathcal I^0}\oplus \CF_{bot}(K_{A, B})|_{\mathcal I^+},\]
 is an $A_\infty$ ideal. 
We additionally consider the map $H:\CF_{bot}(K_{A, B})|_{\mathcal I^+}\to \CF_{bot}(K_{A, B})|_{\mathcal I^0}$ which on the given basis sends:
\begin{align*}
    H(e_a)=e_1-e_0 && H(y)=e_0\\
    H((0\to \pi)_a)= T^{-A}(x_1-x_0) && H(x_a)=x_0
\end{align*}
Let $i^+:\CF(L|_{E_b})\to \CF_{bot}(K_{A, B})$ be the chain-level inclusion 
\begin{align*}
    i^+(e_b)=e_b+e_1-e_0 &&  i^+(x_b)=x_b+x_0\\
    i^+((\pi \to 0)_b)= T^{-B}((\pi \to 0)_a+ T^A e_0) && i^+(0\to \pi)_b= (0\to \pi)_b+T^{B-A}(x_1-x_0)
\end{align*}
At lowest order, the map $i^+$ is a homotopy inverse to $\pi^+$  with homotopy given by $H$. 

By using the homotopy transfer theorem\footnote{There may be some concern that the valuation of $H$ is negative. This is in general not a problem as long as $H\circ m^0$ has positive valuation. Since we have canceled out $m^0$ by incorporation of a bounding cochain, we can use the homotopy transfer theorem here.} we may extend $i^+$ to an $A_\infty$ homomorphism $\hat i^+: \CF(L_{E_b},(\pi^+)_* \mathfrak b_{A, B})\to \CF_{bot}(K_{A, B},\mathfrak b_{A, B})$.
We combine the inclusion map with the projection $\pi^-: \CF_{bot}(K_{A, B}, \mathfrak b_{A, B})\to \CF(L^-, (\pi^-)_*\mathfrak b_{A, B})$ to obtain a continuation map 
\[\Phi:=\pi^-\circ \hat i^+:\CF(L_{E_b}, (\pi^+)_* \mathfrak b_{A, B})\to \CF(S^1_E\sqcup S^1_{-E},(\pi^-)_* \mathfrak b_{A, B}).\]

One can furthermore use the continuation map to study how modifications of the bounding cochain on $(L_{E_b}, \mathfrak b)$ relate to local systems on the Lagrangian $S^1_E, \sqcup S^1_{-E}$. 
As $\Phi$ is a weakly filtered $A_\infty$ homomorphism of energy loss $B-A$, there exists a pushforward map on degree 1 Maurer-Cartan solutions with valuation greater than $A-B$
\[\Phi_*:\{\mathfrak d \in \mathcal MC^1(L_{E_b}, \mathfrak b), \val(\langle \mathfrak d, (0\to \pi)_b\rangle )>A-B\}\to \mathcal MC^{1}(S^1_E\sqcup S^{-E}).\]
Given an element $\mathfrak d = c_1 (T^{A-B+\epsilon}+\mathcal O(A-B+\epsilon)) ((0\to \pi)_b)+2c_0(T^\epsilon +\mathcal O(\epsilon)) x_b$, we can compute the pushforward to first order:
\begin{align*}
    \Phi_*(\mathfrak d)= c_0 (T^\epsilon+\mathcal O(\epsilon)) x_0 + c_1 (T^\epsilon + \mathcal O(\epsilon)) x_1
\end{align*}
showing that $\Phi_*$ is a surjection.
If we interpret graded deformations in $\mathcal MC((L_{E_b}, \mathfrak b))$ as a local system, this suggests that all local systems on $S^1_E\sqcup S^1_{-E}$ are realized by deformations of $(L_{E_b}, \mathfrak b))$.

\subsubsection[Continuation maps and Unobstructedness: Case II]{Obstructedness: $B\geq A$}
The other setting of interest is when $B\geq A$, as drawn in \cref{fig:obstructedpants}. Then the Lagrangian cobordism $K_{A, B}$ is obstructed, in the sense that $\mathcal MC(K_{A, B})$ is empty. 
In this setting, $L_{E_b}$ is never isomorphic to $S^1_E\sqcup S^1_{-E}$ for any choice of deforming cochain, as $E_b< 2E$. 
This is easily observed as $L_{E_b}$ can be displaced from $S^1_E\sqcup S^1_{-E}$ by a Hamiltonian isotopy.
 
\appendix
\section{Monotone Two-Ended Lagrangian Cobordisms}
\label{subsec:widenornarrow}
\label{app:wideornarrow}
A Lagrangian submanifold $L\subset (X, \omega)$ is called monotone if there exists a $\lambda>0$ such that the homomorphism 
\begin{align*}
    \omega, \mu: \pi_2(X, L)\to \RR
\end{align*} 
are proportional so that $\omega=\lambda \mu$.
Monotone Lagrangian submanifolds provide an interesting subset of examples to work with as they are more general than exact Lagrangian submanifolds, and always have well-defined Floer cohomology. As a result, the Floer cohomology of monotone Lagrangian submanifolds has been studied extensively. The pearly-Floer cohomology of monotone Lagrangian submanifolds $\CF(L)$ is a deformation of the Morse cohomology \cite{biran2008lagrangian}; as such, the maximal rank of the Floer cohomology group is bounded by the cohomology of $L$. We say that $L$ is \emph{wide} if this bound is achieved, i.e. $\HF(L)\simeq H^\bullet(L; \Lambda)$.
Many constructions of wide monotone Lagrangian submanifolds exist (for example, every exact Lagrangian). 
In contrast, a Lagrangian submanifold is called narrow if $\HF(L)$ vanishes. The Chiang Lagrangian in dimension $n$ not a prime power is narrow over every characteristic \cite[Theorem 4.10]{smith2021monotone}. 
A dramatic example of a narrow monotone Lagrangian submanifold comes from \cite{oakley2016certain}, which produces examples of displaceable monotone Lagrangian submanifolds.
The wide-narrow dichotomy (introduced in \cite{biran2009rigidity}) observes that the relative sizes of the Floer cohomology groups and the cohomology groups of a monotone Lagrangian submanifold tends to characterize the ``rigidity'' of a Lagrangian submanifold. More precisely: quantitative measures of the size of a Lagrangian submanifold $L$ (such as the width, or Gromov width of the complement of $L$) are oftentimes small for narrow Lagrangians and large for wide Lagrangians. It was conjectured in \cite{biran2009rigidity} that all Lagrangian submanifolds are either wide or narrow. 
We include a construction (the main idea was communicated to us by Abouzaid and Auroux) of a monotone Lagrangian submanifold which is neither wide-nor-narrow.

For the same reason that monotonicity is a good property for studying Lagrangian submanifold, the theory of monotone Lagrangian cobordisms is similarly well-developed. For instance, it is known that every pair of monotone Lagrangian cobordant submanifolds are isomorphic in the Fukaya category. To our knowledge, we provide the first example of a 2-ended embedded monotone Lagrangian cobordism not given by suspension of a Hamiltonian isotopy.
\subsection{A Neither wide nor narrow monotone Lagrangian submanifold}
\label{subsec:notwidenornarrow}
The goal is to obtain a Lagrangian submanifold which is neither wide nor narrow, that is $\HF(L)\neq 0, H^\bullet(L)$.
The approach is the following: given Lagrangian submanifolds $L_1, L_2$ which fit into a monotone Lagrangian cobordism $K: (L_1, L_2)\rightsquigarrow L_3$, \cite{biran2013lagrangian} shows that $\CF(L_3)$ is a mapping cone of a morphism $\CF(L_1)\rightarrow \CF(L_2)$. 
If $L_1$ is narrow and $L_2$ is wide, then 
\[\HF(L_3)=\HF(L_2)=H^\bullet(L_2),\]
which is unlikely to be $H^\bullet(L_3)$.

The difficulty of this construction is that generally the surgery you produce is non-monotone. If $L_1$ and $L_2$ intersect transversely at distinct points $L_1\pitchfork L_2=\{x_1, \ldots, x_k\}$, then a candidate Lagrangian cobordism is the surgery trace cobordism. However, it is rarely the case that the Lagrangian submanifold obtained by surgering $L_1$ and $L_2$ at their intersection points, 
 \[L_3=L_1\#_{x_1, \ldots, x_k} L_2\]
  is a monotone Lagrangian submanifold, as any strip of index 1 with boundary on $L_1\cup L_2$ will give rise to a disk of Maslov index zero on the Lagrangian surgery. 
One might hope that $L_1$ and $L_2$ intersect at single point; problematically, the Lagrangian intersection Floer cohomology $\HF(L_1, L_2)$ would be non-vanishing, which seemingly contradicts the narrowness of $L_1$.

The proposal of Abouzaid and Auroux is to perform the above construction with Lagrangian submanifolds $L_1$ and $L_2$ which intersect cleanly along a single (non-point) connected component, and perform surgery along the component. The upshot is that (in comparison to the transverse intersection case) one would hope that there are no Maslov index 1 strips with boundary in $L_1\cup L_2$; intuitively, we should think that all of the Maslov index 1 strips have degenerated to flow-lines on the intersection locus for some choice of Morse function on $L_1\cap L_2$ corresponding to an infinitesimal Hamiltonian perturbation of $L_2$. 
In this setting, one can still define the Lagrangian surgery of $L_1$ and $L_2$ at their intersection (either using the ideas of \cite{fang2017geometric}, or surgery constructions from \cite{mak2018dehn,hicks2020tropical}); this comes with a surgery trace cobordism allowing us to apply the mapping cone construction of Biran and Cornea.

We then need to find monotone $L_1$ and $L_2$ which intersect transversely along a single connected component. Let $X$ be a symplectic manifold, and $L\subset X$ be a narrow monotone Lagrangian submanifold.
Inside of $X\times \bar X$, consider the Lagrangian submanifold $L_1=L\times L\subset X\times \bar X$.
By application of the K\"unneth formula, $L_1$ is a narrow monotone Lagrangian submanifold of $X\times \bar X$.

We also consider the diagonal Lagrangian $L_2=\{(x, x)\in X\times \bar X\;|\; x\in X\}$. This Lagrangian is monotone, and by identifying the Floer cohomology of the diagonal with the fixed-point Floer cohomology of $X\times \bar X$, we see that $L_2$ is wide.

These two Lagrangian submanifolds intersect cleanly along $\Delta_L:=\{(x, x)\in X\times \bar X\;|\; x\in L\}$, so the surgery $L_3:=L_1\#_{\Delta_L}L_2$ is a good candidate for a wide nor-narrow Lagrangian submanifold.
It remains to check if the surgery trace cobordism $K:( L_1, L_2)\rightsquigarrow L_3$ is monotone. A criterion is given by \cite[Lemma 6.3]{mak2018dehn} on the fundamental groups of the summands of the Lagrangian connect sum, namely at least one of $\pi_1(L_i)$ is torsion in $\pi_1(X\times \bar X)$.
This will clearly depend on the choice of $L$ and $X$, which we now fix.

As suggested to the author by Cheuk Yu Mak, we take the example $L=L^{p}_{0,m}\subset \CP^{m+1}=X$ from \cite[Theorem 1.4]{oakley2016certain} which is self-displaceable (and therefore narrow) and monotone. The topology of $L$ is $S^1\times S^m$. Since $X\times \bar X$ is simply connected, the image of $\pi_1(L_2)$ is torsion in $\pi_1(X\times \bar X)$. 
Therefore the surgery trace cobordism $K:( L_1, L_2)\rightsquigarrow L_3$ is monotone.

Finally, it remains to show that $L_3$ has different homology than $L_2$. A computation on the fundamental group shows that when $m>2$
\begin{align*}
    \pi_1(L_3)= &\pi_1(\CP^{m+1}\setminus \Delta_L)*_{\pi_1(\Delta_L)}\pi_1((L\times L)\setminus \Delta_L)\\
    =& \ZZ\neq \pi_1(L_2)
\end{align*}
as $L_2\cong \CP^{m+1}$.

\subsection{A 1-ended oriented embedded monotone Lagrangian cobordism}
\label{subsec:oneended}
We expand on ideas in the construction above to produce a 1-ended oriented embedded monotone Lagrangian cobordism. 
First, a short observation. Let $i: L\to X$ be a Lagrangian immersion. Let $\pi_2(X, i: L\to X)$ denote the set of disks $u: (D, \partial D)\to (X, L)$ with the property that $u|_{\partial D}$ factors through the immersion $i$.
If $i$ is an embedding, this is exactly the same as $\pi_2(X, i(L))$.
Suppose that $i^\pm: L\to X$ are exactly homotopic Lagrangian immersions. There is a canonical identification $\pi_2(X, i^+: L\to X)\cong\pi_2(X, i^-: L\to X)$, and the Maslov index and symplectic area
\begin{align*}
    \mu:& \pi_2(X, i^\pm: L\to X)\\
    \omega:& \pi_2(X, i^\pm: L\to X)
\end{align*}
are unchanged under this identification. This is not the case when we work with  $\pi_2(X, i^\pm(L))$, which cannot usually be identified. Even when they are (i.e. there are no birth/deaths of self-intersections), the symplectic area of polygons with corners of the self-intersection on $i^\pm(L)$ can change. However, from this we can conclude: 
\begin{claim}
    Suppose that  $i^+: L\to X$ is an embedding which is exactly homotopic to $i^-: L\to X$ (which may not be an embedding). If $i^-(L)\subset X$ is monotone, then $i^+(L)$ is monotone.
    \label{claim:homotopymonotone}
\end{claim}

Let $L\subset X$ be a displaceable oriented monotone embedded Lagrangian submanifold, with a given parameterization $i_t^L:L\times [0, 1]_t\to X$ of the exact isotopy  which displaces the Lagrangian submanifold from itself (for example, from \cite[Theorem 1.4]{oakley2016certain}). 
Additionally we suppose that this exact isotopy is picked so that $i_t, i_{1-t}$ are disjoint for all $t\in (0, 1/4)$. Let $H_t: L\to \RR$ be the primitive of the flux form. Pick $C$ large enough so that 
\begin{equation}
    |H_t|<C/2.
    \label{eq:HamiltonianBound}
\end{equation}

Consider now the curve $\gamma_0: S^1=\RR_t/4\ZZ\to T^*S^1=\CC/(100C\ZZ)$ parameterized in \cref{fig:gamma0}, whose parameterization over $[1, 2]$ and $[3, 4]$ is given by $t+2\jmath C$ and $(t-2)-2\jmath C$ respectively. The self-intersection of the curve occurs at $\gamma_0(0.5)=\gamma_0(2.5)$. The parameterization is picked so that those segments have normal neighborhoods of radius $C$ over those portions of the parameterization. The area $A$ highlighted in blue (notably avoiding the normal neighborhoods) is chosen to be at least $10C$.

Since $\gamma_0$ parameterizes a Whitney sphere, we can construct  a null-cobordism $j^\gamma: D^2\to T^*S^1\times \CC$,
\[j^\gamma(D^2): \emptyset \rightsquigarrow \gamma_0(S^1).\]
We will denote the cobordism parameter on $T^*S^1\times \CC$ by $s$.
We parameterize the Lagrangian null-cobordism so that it is the suspension of an exact homotopy when $s\leq 1$, as drawn in the bottom of \cref{fig:oneEnded}.
The negative end of this cobordism has a chart 
\begin{align*}
     j^\gamma|_{s\leq 1}:S^1\times(-\infty, 1]_s\to T^*S^1\times \CC\\
     (t, s)\mapsto( \gamma_s(t), s+\jmath H^\gamma_s(t))
\end{align*}
Because we picked $A>10C$, we can ensure that the following criteria are met in the above chart:
\begin{enumerate}[(A)]
    \item During this portion of the Lagrangian null-cobordism, $\gamma_s(t)=\gamma(t)$ for all $t\in [1, 2]\cup[3,4]$. This means that $H^{\gamma}_s(t)$ is constant in the $t$ parameter over $[1,2]\cup [3,4]$. \label{item:constantOverSuspension}
    \item If $\gamma_s(t_0)=\gamma_s(t_1)$, with $t_0\neq t_1$, then $t_0, t_1\in \{0.5,2.5\}$. \label{item:fixedIntersection}
    \item For all $s\in [1/8,1]$, we have $H^{\gamma}_s(0.5)<-C, H^{\gamma}_s(2.5)>C$. \label{item:LargeHamiltonian}
    \item $\gamma_s|_{t\in (0,1)\cup(2,3)}$ remains disjoint from the normal neighborhoods of radius $C$ drawn in \cref{fig:gamma0}. \label{item:avoidsNeighborhoods}
\end{enumerate}

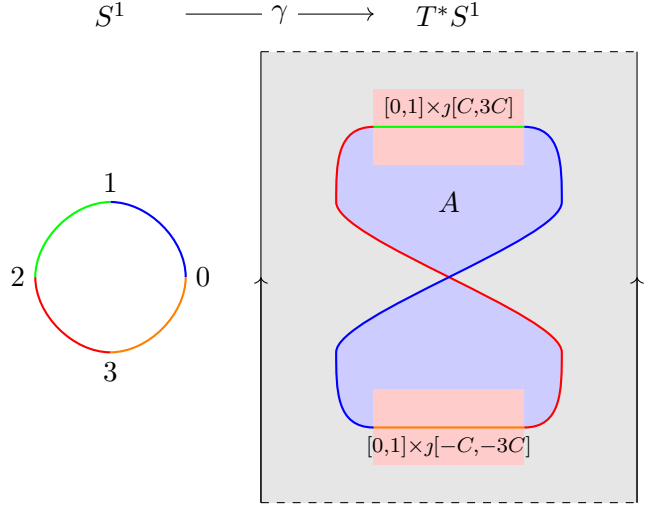
\begin{figure}
    \centering
\begin{tikzpicture}

    \fill[gray!20]  (4,-2) rectangle (-1,4);
        \fill[blue!20] (0.5,3) .. controls (0,3) and (0,2.5) .. (0,2) .. controls (0,1.5) and (3,0.5) .. (3,0) .. controls (3,-0.5) and (3,-1) .. (2.5,-1) .. controls (2,-1) and (1,-1) .. (0.5,-1) .. controls (0,-1) and (0,-0.5) .. (0,0) .. controls (0,0.5) and (3,1.5) .. (3,2) .. controls (3,2.5) and (3,3) .. (2.5,3);
    
        \fill[red!20]  (0.5,3.5) rectangle (2.5,2.5);
        \fill[red!20] (0.5,-0.5) rectangle (2.5,-1.5);
            \draw[red, thick] (2.5,-1) .. controls (3,-1) and (3,-0.5) .. (3,0) .. controls (3,0.5) and (0,1.5) .. (0,2) .. controls (0,2.5) and (0,3) .. (0.5,3);
            \draw[blue, thick] (2.5,3) .. controls (3,3) and (3,2.5) .. (3,2) .. controls (3,1.5) and (0,0.5) .. (0,0) .. controls (0,-0.5) and (0,-1) .. (0.5,-1);
            \draw[green, thick] (0.5,3) -- (2.5,3);
            \draw[orange, thick] (0.5,-1) -- (2.5,-1);
        \draw[blue, thick] (-2,1) .. controls (-2,1.5) and (-2.5,2) .. (-3,2);
        \draw[green, thick] (-3,2) .. controls (-3.5,2) and (-4,1.5) .. (-4,1);
        \draw[red, thick] (-4,1) .. controls (-4,0.5) and (-3.5,0) .. (-3,0);
        \draw[orange, thick] (-3,0) .. controls (-2.5,0) and (-2,0.5) .. (-2,1);
        \draw (-2,4.5) edge[->] node[fill=white]{$\gamma$} (0.5,4.5);
        \node at (1.5,3.25) {$\scriptstyle[0,1]\times \jmath [C, 3C]$};
        \node at (1.5,-1.25) {$\scriptstyle[0,1]\times \jmath [-C, -3C]$};
        \node at (1.5,2) {$A$};
    \draw(-1,-2) edge[->](-1,1)  -- (-1,4) (4,4) -- (4,-2)edge[->](4,1) ;
    \draw[dashed] (-1,4) -- (4,4) (-1,-2) -- (4,-2);
    
    \node at (1.5,4.5) {$T^*S^1$};
    \node at (-3,4.5) {$S^1$};
\node[above] at (-3,2) {$1$};
\node[left] at (-4,1) {$2$};
\node[below] at (-3,0) {$3$};
\node[right] at (-2,1) {$0$};
\end{tikzpicture}
     \caption{The path parameterized by $\gamma_0$.}
\label{fig:gamma0}
\end{figure}

We now describe the main idea of the construction of the null-cobordism. Consider the Lagrangian null-cobordism
\[j^\gamma\times i^L_0(D^2\times L): \emptyset \rightsquigarrow \gamma_0(S^1)\times i^L_0(L).\]
This is an immersed monotone Lagrangian null-cobordism, with self-intersection along $\{0\}\times L\times \RR_{<0}\subset \CC\times X\times \CC$. 
Since $L$ is displaceable from itself, we can hope to make this immersed Lagrangian null-cobordism an embedded Lagrangian null-cobordism by application of an exact homotopy. 

We first describe the negative end of this embedded Lagrangian null-cobordism. Consider a cutoff function $\rho_0(t):[0,1]\to [0,1]$ with the property that $\partial_t\rho<2$, $\rho|_{t<1/10}=0$ and $\rho|_{t>9/10}=1$. We look at the Lagrangian submanifold $\overline i_0^L: L\times S^1\to X\times T^*S^1$ drawn in \cref{fig:perturbedfigure8} which is obtained by concatenating:
\begin{itemize}
    \item The identity cobordism for $i_0:L\to X$ along the blue path,
    \item The suspension of the exact isotopy $i_{\rho_0(t)}:L\times[0,1]_t\to X$ (the green region)
    \item The identity cobordism for $i_1:L\to X$ along the red path,
    \item The suspension of the exact isotopy $i_{\rho_0(-t)}:L\times [0, 1]_t\to X$ (the orange region).
\end{itemize}
The blue path and red path are chosen to match the paths of $\gamma_0$ drawn in \cref{fig:gamma0}, and by our bound on $\partial_t\rho$, and the suspension cobordisms have shadow lying inside the neighborhoods surrounding $\gamma_0|_{t\in [1,2]\cup [3,4]}$.  Since $L$ is embedded and each of the four pieces described above are embedded; therefore the only self-intersections of $\overline i_0^L:L\times S^1\to X\times T^*S^1$ occur when the shadows of these four pieces overlap. By \cref{item:avoidsNeighborhoods}, the only such occurrence is when the blue and red path cross. This occurs at $t=0.5, 1.5$ by \cref{item:fixedIntersection}.
However the $X$-component of these pieces are given by $i_0:L\to X$ and $i_1:L\to X$, which were chosen to be disjoint. Therefore the Lagrangian parameterized by $\overline i_0^L$ is embedded.

We are now ready to modify the immersed Lagrangian null-cobordism $j^\gamma\times i^L_0:D^2\times L\to ( X\times T^*S^1\times \CC)$ so that the end agrees with $\overline i_0^L$. This Lagrangian cobordism is embedded over the portion where the cobordism parameter is greater than 1. 
We therefore need to replace $j^\gamma\times i^L_0|_{s\leq 1}$ with a suspension of an exact homotopy (defined for $s\leq 1)$ 
\begin{enumerate}[(i)]
    \item whose negative end is given by the Lagrangian $\overline i_0^L:S^1\times L\to T^*S^1\times X$;\label{item:negEnd}
    \item which agrees with $j^\gamma\times i^L_0$ in a neighborhood $s=1$; and \label{item:posEnd}
    \item is embedded.\label{item:embedded}
\end{enumerate}
The slices of the suspension are drawn in \cref{fig:oneEnded}. 
We now give an explicit parameterization of this exact homotopy. Let $\rho_s(t):[0, 1]_t\times[0, 1]_s\to [0, 1]$ be a smooth function drawn in \cref{fig:auxfunction} with
\begin{enumerate}[(a)]
    \item $\rho_s(0)=0$, $\rho_0(1)=1$ and $\rho_1(1)=0$;\label{item:rhoBoundary}
    \item $0\leq \partial_t\rho_s  < 2C$;\label{item:rhoSmallDerivative}
    \item When $0\leq s<1/4$ and $3/4\leq s \leq 1$, we have $\partial_s \rho_s=0$; and \label{item:rhoCobordism}
    \item When $0\leq t<1/10$ and $9/10 \leq t \leq 1$, we have $\partial_t \rho_s=0$. \label{item:gluesToCobordism}
\end{enumerate}

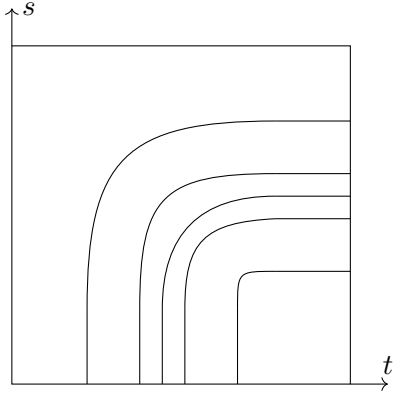
\begin{SCfigure}
    \begin{tikzpicture}

\draw  (-2,1) rectangle (2.5,-3.5)  {};
\draw (-1,-2.5) .. controls (-1,-0.5) and (-0.5,0) .. (1.5,0);
\draw (-1,-3.5) -- (-1,-2.5);
\draw (2.5,0) -- (1.5,0) (-0.3,-3.5) -- (-0.3,-2.5) (0,-3.5) -- (0,-2.5) (0.3,-3.5) -- (0.3,-2.5) (1,-3.5) -- (1,-2.5);
\draw (1.5,-0.7) -- (2.5,-0.7) (1.5,-1) -- (2.5,-1) (1.5,-1.3) -- (2.5,-1.3) (1.5,-2) -- (2.5,-2);
\draw (-0.3,-2.5) .. controls (-0.3,-1) and (0,-0.7) .. (1.5,-0.7);
\draw (0,-2.5) .. controls (0,-1.5) and (0.5,-1) .. (1.5,-1);
\draw (0.3,-2.5) .. controls (0.3,-1.6) and (0.55,-1.35) .. (1.5,-1.3);
\draw (1,-2.5) .. controls (1,-2) and (1,-2) .. (1.5,-2);
\draw[->] (2.5,-3.5)  -- (3,-3.5);
\draw[->] (-2,1) -- (-2,1.5);
\node[above] at (3,-3.5) {$t$};
\node[right] at (-2,1.5) {$s$};
\end{tikzpicture}     \caption{A contour plot of the  auxiliary function $\rho_{s}(t): [0,1]_t\times [0,1]_s\to [0,1]$. }
    \label{fig:auxfunction}
\end{SCfigure}
Consider now the exact homotopy of Lagrangian submanifolds defined for $s\in [0, 1]$
\begin{align*}
    \overline i_s^L: L\times S^1\to& X\times T^*S^1\\
    (q, t)\mapsto &\left\{\begin{array}{cc}
        \left(i_0^L(q), \gamma_s(t)\right)&\text{ if $t\in [0, 1)$}\\
        \left(i^L_{\rho_s(t-1)}(q), \gamma_s(t)-\jmath\frac{d\rho_s}{dt}H_{\rho_s(t-1)}(q)\right)& \text{ if $t\in [1, 2)$}\\
        \left(i^L_{s}(q), \gamma_s(t)\right) & \text{ if $t\in [2, 3)$}\\
        \left(i^L_{\rho_s(4-t)} ,\gamma_s(t)+\jmath\frac{d\rho_{s}}{dt}H_{\rho_s(4-t)}(q)\right) & \text{if $t\in [3, 4)$}
    \end{array}\right.
\end{align*}
The map $i_s^L$ is smooth by \cref{item:gluesToCobordism}.

Now consider the Lagrangian suspension cobordism of this exact homotopy
\begin{align*}
    j^L|_{(0,1)}: L\times S^1\times (0,1)\to& X\times T^*S^1\times \CC\\
    (q, t, s)\mapsto& \left(\overline i^L_s(q, t), s+\jmath \overline H_s(q, t)\right)
\end{align*}
where $\overline H_s(q, t)$ is the primitive for the exact homotopy $\overline i^L_s(q, t)$. 
By using \cref{item:rhoBoundary} of $\rho$, boundary of the suspension satisfies agrees with $\lj^\gamma\times i^L_0(D^2\times L)|_{s=1}$, and by \cref{item:gluesToCobordism} the domains of the suspension and $\lj^\gamma\times i^L_0(D^2\times L)|_{s\geq 1}$ can be glued together to give a smooth Lagrangian cobordism.

We now prove that $ j^L|_{(0,1)}$ is an embedding (\cref{item:embedded}).
Observe that the slices of $j^L|_{(0,1)}$ are embedded away from the locus $t\in \{0.5, 2.5\}, s>1/4$. By employing  \cref{item:LargeHamiltonian,eq:HamiltonianBound,item:rhoSmallDerivative} over the region where $s>1/4$, 
\begin{align*}
    \overline H_s(q, 0.5)&<H^{\gamma}_s(q, 0.5)+|\partial_t\rho_s||H_t|< C + 2\cdot \frac{C}{2}\\
    &\leq 0\leq C-2\cdot \frac{C}{2}<H^{\gamma}_s(q, 1.5)+|\partial_t \rho_s||H_t|< \overline H_s(q, 0.5) 
\end{align*}
so $j^L|_{(0,1)}$ is an embedding.

\begin{figure}
    \centering
    \begin{tikzpicture}

\fill[gray!20]  (4,-2) rectangle (-1,4);

        \draw[red, thick] (2.5,-1) .. controls (3,-1) and (3,-0.5) .. (3,0) .. controls (3,0.5) and (0,1.5) .. (0,2) .. controls (0,2.5) and (0,3) .. (0.5,3);
        \draw[blue, thick] (2.5,3) .. controls (3,3) and (3,2.5) .. (3,2) .. controls (3,1.5) and (0,0.5) .. (0,0) .. controls (0,-0.5) and (0,-1) .. (0.5,-1);
        \draw[green, thick] (0.5,3) -- (2.5,3);
        \draw[orange, thick] (0.5,-1) -- (2.5,-1);
    
\draw(-1,-2) edge[->](-1,1)  -- (-1,4) (4,4) -- (4,-2)edge[->](4,1) ;
\draw[dashed] (-1,4) -- (4,4) (-1,-2) -- (4,-2);

\draw[fill=green] (2.5,3) .. controls (2,3) and (2,2.25) .. (1.5,2.25) .. controls (1,2.25) and (1,3) .. (0.5,3) .. controls (1,3) and (1,3.75) .. (1.5,3.75) .. controls (2,3.75) and (2,3) .. (2.5,3);
\draw[fill=orange] (0.5,-1) .. controls (1,-1) and (1,-1.75) .. (1.5,-1.75) .. controls (2,-1.75) and (2,-1) .. (2.5,-1) .. controls (2,-1) and (2,-0.25) .. (1.5,-0.25) .. controls (1,-0.25) and (1,-1) .. (0.5,-1);

\node[blue, fill=gray!20] at (2.85,1.95) {$i_0^L(L)\times [0,1]$};
\node at (1.5,3) {$K_{i_t^L}^{-1}$};
\node[red, fill=gray!20] at (0.2,1.95) {\scriptsize$ i_1^L(L)\times [2,3]$};
\node at (1.5,-1) {$K_{i_t^L}$};
\end{tikzpicture}
     \caption{The Lagrangian submanifold $\overline i^L_0(q, t): L\times S^1\to X\times T^*S^1$, projected to the $S^1$ coordinate.}
    \label{fig:perturbedfigure8}
\end{figure}
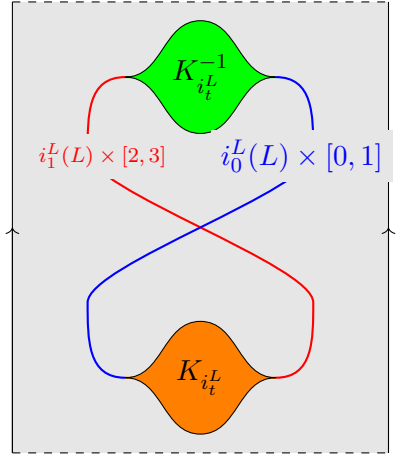
We now describe a Lagrangian null-cobordism of $\overline i^L_0(q, t)$. 

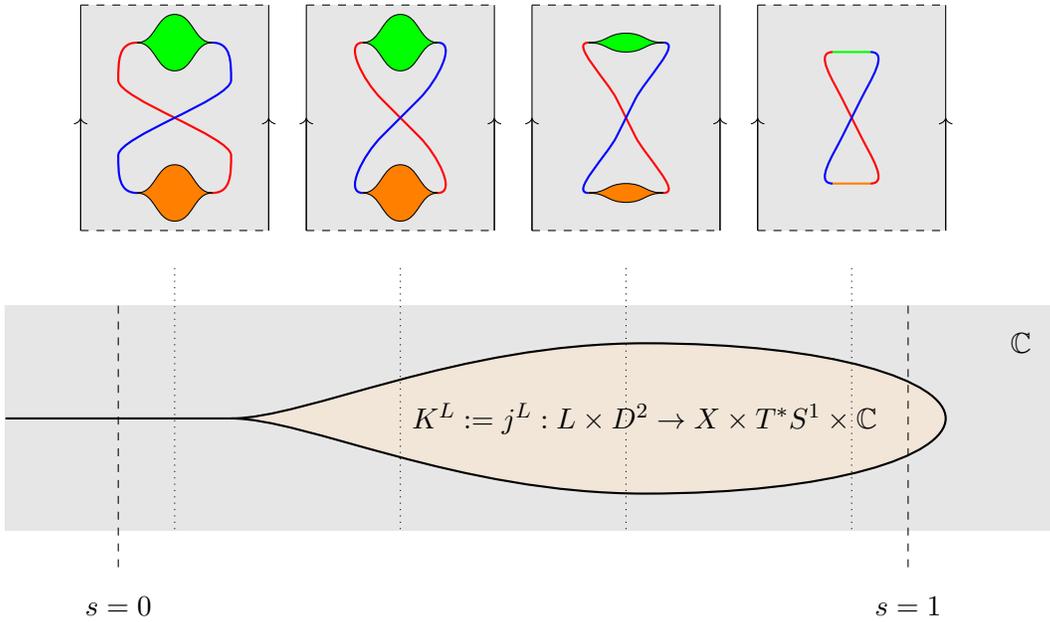
\begin{figure}
    \centering
    \begin{tikzpicture}
    \begin{scope}[shift={(3,-1)},scale=0.5]
    
    \fill[gray!20]  (4,-2) rectangle (-1,4);
    
            \draw[red, thick] (2.5,-1) .. controls (3,-1) and (3,-0.5) .. (3,0) .. controls (3,0.5) and (0,1.5) .. (0,2) .. controls (0,2.5) and (0,3) .. (0.5,3);
            \draw[blue, thick] (2.5,3) .. controls (3,3) and (3,2.5) .. (3,2) .. controls (3,1.5) and (0,0.5) .. (0,0) .. controls (0,-0.5) and (0,-1) .. (0.5,-1);
            \draw[green, thick] (0.5,3) -- (2.5,3);
            \draw[orange, thick] (0.5,-1) -- (2.5,-1);
        
    \draw(-1,-2) edge[->](-1,1)  -- (-1,4) (4,4) -- (4,-2)edge[->](4,1) ;
    \draw[dashed] (-1,4) -- (4,4) (-1,-2) -- (4,-2);

    \draw[fill=green] (2.5,3) .. controls (2,3) and (2,2.25) .. (1.5,2.25) .. controls (1,2.25) and (1,3) .. (0.5,3) .. controls (1,3) and (1,3.75) .. (1.5,3.75) .. controls (2,3.75) and (2,3) .. (2.5,3);
    \draw[fill=orange] (0.5,-1) .. controls (1,-1) and (1,-1.75) .. (1.5,-1.75) .. controls (2,-1.75) and (2,-1) .. (2.5,-1) .. controls (2,-1) and (2,-0.25) .. (1.5,-0.25) .. controls (1,-0.25) and (1,-1) .. (0.5,-1);
    
    \end{scope}

    \begin{scope}[shift={(6,-1)},scale=0.5]
    
    \fill[gray!20]  (4,-2) rectangle (-1,4);
    
            \draw[red, thick] (2.5,-1) .. controls (3,-1) and (2.5,0) .. (2,0.5) .. controls (1.5,1) and (1.5,1) .. (1,1.5) .. controls (0.5,2) and (0,3) .. (0.5,3);
            \draw[blue, thick] (2.5,3) .. controls (3,3) and (2.5,2) .. (2,1.5) .. controls (1.5,1) and (1.5,1) .. (1,0.5) .. controls (0.5,0) and (0,-1) .. (0.5,-1);
            \draw[green, thick] (0.5,3) -- (2.5,3);
            \draw[orange, thick] (0.5,-1) -- (2.5,-1);
        
    \draw(-1,-2) edge[->](-1,1)  -- (-1,4) (4,4) -- (4,-2)edge[->](4,1) ;
    \draw[dashed] (-1,4) -- (4,4) (-1,-2) -- (4,-2);

    \draw[fill=green] (2.5,3) .. controls (2,3) and (2,2.25) .. (1.5,2.25) .. controls (1,2.25) and (1,3) .. (0.5,3) .. controls (1,3) and (1,3.75) .. (1.5,3.75) .. controls (2,3.75) and (2,3) .. (2.5,3);
    \draw[fill=orange] (0.5,-1) .. controls (1,-1) and (1,-1.75) .. (1.5,-1.75) .. controls (2,-1.75) and (2,-1) .. (2.5,-1) .. controls (2,-1) and (2,-0.25) .. (1.5,-0.25) .. controls (1,-0.25) and (1,-1) .. (0.5,-1);
    
    \end{scope}

    \begin{scope}[shift={(9,-1)},scale=0.5]
    
    \fill[gray!20]  (4,-2) rectangle (-1,4);
    
            \draw[red, thick] (2.5,-1) .. controls (3,-1) and (2,0) .. (1.75,0.5) .. controls (1.5,1) and (1.5,1) .. (1.25,1.5) .. controls (1,2) and (0,3) .. (0.5,3);
            \draw[blue, thick] (2.5,3) .. controls (3,3) and (2,2) .. (1.75,1.5) .. controls (1.5,1) and (1.5,1) .. (1.25,0.5) .. controls (1,0) and (0,-1) .. (0.5,-1);
            \draw[green, thick] (0.5,3) -- (2.5,3);
            \draw[orange, thick] (0.5,-1) -- (2.5,-1);
        
    \draw(-1,-2) edge[->](-1,1)  -- (-1,4) (4,4) -- (4,-2)edge[->](4,1) ;
    \draw[dashed] (-1,4) -- (4,4) (-1,-2) -- (4,-2);

    \draw[fill=green] (2.5,3) .. controls (2,3) and (2,2.75) .. (1.5,2.75) .. controls (1,2.75) and (1,3) .. (0.5,3) .. controls (1,3) and (1,3.25) .. (1.5,3.25) .. controls (2,3.25) and (2,3) .. (2.5,3);
    \draw[fill=orange] (0.5,-1) .. controls (1,-1) and (1,-1.25) .. (1.5,-1.25) .. controls (2,-1.25) and (2,-1) .. (2.5,-1) .. controls (2,-1) and (2,-0.75) .. (1.5,-0.75) .. controls (1,-0.75) and (1,-1) .. (0.5,-1);
    
    \end{scope}

    \begin{scope}[shift={(12,-1)},scale=0.5]
    
    \fill[gray!20]  (4,-2) rectangle (-1,4);
    
            \draw[red, thick] (2,-0.75) .. controls (2.5,-0.75) and (2,0) .. (1.75,0.5) .. controls (1.5,1) and (1.5,1) .. (1.25,1.5) .. controls (1,2) and (0.5,2.75) .. (1,2.75);
            \draw[blue, thick] (2,2.75) .. controls (2.5,2.75) and (2,2) .. (1.75,1.5) .. controls (1.5,1) and (1.5,1) .. (1.25,0.5) .. controls (1,0) and (0.5,-0.75) .. (1,-0.75);
            \draw[green, thick] (1,2.75) -- (2,2.75);
            \draw[orange, thick] (1,-0.75) -- (2,-0.75);
        
    \draw(-1,-2) edge[->](-1,1)  -- (-1,4) (4,4) -- (4,-2)edge[->](4,1) ;
    \draw[dashed] (-1,4) -- (4,4) (-1,-2) -- (4,-2);

    \end{scope}

    \fill[gray!20]  (1.5,-3) rectangle (15.5,-6);
    \draw[fill=brown!20, thick] (1.5,-4.5) .. controls (2.5,-4.5) and (3.5,-4.5) .. (4.5,-4.5) .. controls (5.5,-4.5) and (7.5,-5.5) .. (10,-5.5) .. controls (12.5,-5.5) and (14,-5) .. (14,-4.5) .. controls (14,-4) and (12.5,-3.5) .. (10,-3.5) .. controls (7.5,-3.5) and (5.5,-4.5) .. (4.5,-4.5);
    \draw[dashed] (3,-3) -- (3,-6.5);
    \node at (3,-7) {$s=0$};
    
    \node at (13.5,-7) {$s=1$};
    \draw[dotted] (3.75,-2.5) -- (3.75,-6);
    \draw[dotted] (6.75,-2.5) -- (6.75,-6);
    \draw[dotted] (9.75,-2.5) -- (9.75,-6);
    \draw[dotted] (12.75,-2.5) -- (12.75,-6);
    \draw[dotted][dashed] (13.5,-3) -- (13.5,-6.5);
    \node at (10,-4.5) {$K^L:= j^L: L\times D^2\to X\times T^*S^1\times \mathbb C$};
    \node at (15,-3.5) {$\mathbb C$};
\end{tikzpicture}
         \caption{An embedded oriented monotone null-cobordism with connected ends}
    \label{fig:oneEnded}
\end{figure}

By \cref{item:negEnd,item:posEnd} we can form an embedded Lagrangian cobordism $K^L: \emptyset \to \bar i^L_0(L\times S^1)$  which is parameterized by $ j^L|_{(0, 1)}$ for slices with $0<s<1$, and parameterized by $i_1^L\times j^\gamma$ for slices $s>1$. We denote this parameterization by
\[
     j^L: L\times D^2\to X\times T^*S^1\times \CC
\]
which is an embedded Lagrangian null-cobordism.

We now prove that $ j^L: L\times D^2\to X\times T^*S^1\times \CC$ is monotone. There is an exact homotopy between $\tilde j^L$ and $i^L_1\times j^\gamma$.
Since the Maslov index and symplectic area decompose across products, and the symplectic area and Maslov index vanish on $\pi_2(X\times T^*S^1\times \CC, j :D^2\to T^*S^1\times \CC)$, $i^L_1\times j^\gamma: L\times D^2\to X\times T^*S^1\times \CC$ , the immersed Lagrangian is monotone with monotonicity constant matching $L$.
By \cref{claim:homotopymonotone}, $j^L: L\times D^2\times X\to T^*S^1\times \CC$ is monotone as well.

\subsection{A 2-ended monotone embedded Lagrangian Cobordism}
\label{subsec:2ended}
We now consider the Lagrangian $L_1=L^{p}_{k,m}\times L^{p}_{k,m}\subset X\times \bar X$. This can be displaced from itself in a way so that the resulting displacement avoids the diagonal Lagrangian $i^{L_2}:L_2\subset X\times \bar X$. Let $i^{S^1}: S^1\to T^*S^1$ be the zero section which intersects the figure eight exactly at the self-intersection of $\gamma_s$.
Consider the embedded Lagrangian submanifold $\bar i^{L_1}_0: L_1\times S^1\subset (X\times \bar X)\times T^*S^1$ from the previous section. We also take the product Lagrangian $i^{L_2}\times i^{S^1}:L_2\times S^1\to (X\times \bar X)\times T^*S^1$. 
These two Lagrangians intersect along $\Delta_L\times \{0\}\subset (X\times \bar X)\times T^*S^1$.

We take the Lagrangian surgery cobordism $K: (\bar i^{L_1}_0,i^{L_2}\times i^{S^1} )\rightsquigarrow \bar i^{L_1}_0\# (i^{L_2}\times i^{S^1})$. By \cref{subsec:notwidenornarrow}, this is a 3-ended monotone Lagrangian cobordism. We concatenate this with the Lagrangian cobordism $K^{L_1}: \emptyset \to \bar i^{L_1}_0$ constructed in \cref{subsec:oneended} to obtain a 2-ended Lagrangian cobordism $K\circ ((K^{L_1})^{-1}\cup L_2\times \RR): i^{L_2}\times i^{S^1}\to  \bar i^{L_1}_0\# (i^{L_2}\times i^{S^1})$.
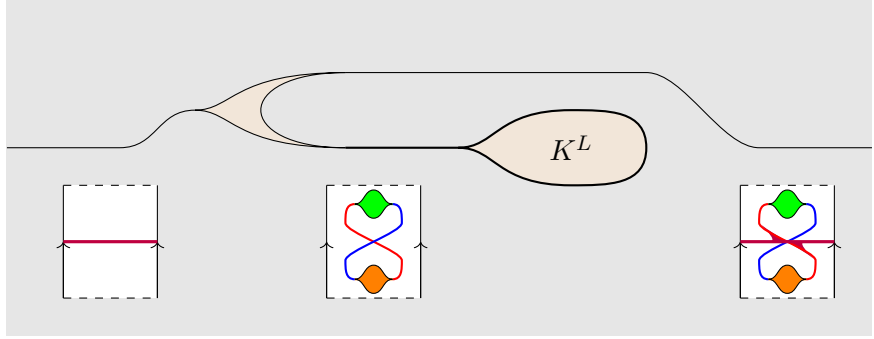
\begin{figure}
    \centering
    \begin{tikzpicture}

\fill[gray!20]  (3,-4.5) rectangle (14.5,-9);
\draw[fill=brown!20, thick] (7.5,-6.5) .. controls (8.5,-6.5) and (8.5,-6.5) .. (9,-6.5) .. controls (9.5,-6.5) and (9.5,-7) .. (10.5,-7) .. controls (11,-7) and (11.5,-7) .. (11.5,-6.5) .. controls (11.5,-6) and (11,-6) .. (10.5,-6) .. controls (9.5,-6) and (9.5,-6.5) .. (9,-6.5);

\node at (10.5,-6.5) {$K^L$};
\draw[fill=brown!20] (7.5,-5.5) .. controls (6,-5.5) and (6,-6) .. (5.5,-6) .. controls (6,-6) and (6,-6.5) .. (7.5,-6.5) .. controls (6,-6.5) and (6,-5.5) .. (7.5,-5.5);

\begin{scope}[shift={(7.5,-8)},scale=0.25]

\fill[white!20]  (4,-2) rectangle (-1,4);

        \draw[red, thick] (2.5,-1) .. controls (3,-1) and (3,-0.5) .. (3,0) .. controls (3,0.5) and (0,1.5) .. (0,2) .. controls (0,2.5) and (0,3) .. (0.5,3);
        \draw[blue, thick] (2.5,3) .. controls (3,3) and (3,2.5) .. (3,2) .. controls (3,1.5) and (0,0.5) .. (0,0) .. controls (0,-0.5) and (0,-1) .. (0.5,-1);
        \draw[green, thick] (0.5,3) -- (2.5,3);
        \draw[orange, thick] (0.5,-1) -- (2.5,-1);
    
\draw(-1,-2) edge[->](-1,1)  -- (-1,4) (4,4) -- (4,-2)edge[->](4,1) ;
\draw[dashed] (-1,4) -- (4,4) (-1,-2) -- (4,-2);

\draw[fill=green] (2.5,3) .. controls (2,3) and (2,2.25) .. (1.5,2.25) .. controls (1,2.25) and (1,3) .. (0.5,3) .. controls (1,3) and (1,3.75) .. (1.5,3.75) .. controls (2,3.75) and (2,3) .. (2.5,3);
\draw[fill=orange] (0.5,-1) .. controls (1,-1) and (1,-1.75) .. (1.5,-1.75) .. controls (2,-1.75) and (2,-1) .. (2.5,-1) .. controls (2,-1) and (2,-0.25) .. (1.5,-0.25) .. controls (1,-0.25) and (1,-1) .. (0.5,-1);

\end{scope}

\begin{scope}[shift={(13,-8)},scale=0.25]

\fill[white!20]  (4,-2) rectangle (-1,4);

\fill[fill=purple] (0.5,1) .. controls (1,1) and (0.5,1.5) .. (0,2) .. controls (0.5,1.5) and (2.5,0.5) .. (3,0) .. controls (2.5,0.5) and (2,1) .. (2.5,1) .. controls (2,1) and (1,1) .. (0.5,1);

        \draw[red, thick] (2.5,-1) .. controls (3,-1) and (3,-0.5) .. (3,0) .. controls (3,0.5) and (0,1.5) .. (0,2) .. controls (0,2.5) and (0,3) .. (0.5,3);
        \draw[blue, thick] (2.5,3) .. controls (3,3) and (3,2.5) .. (3,2) .. controls (3,1.5) and (0,0.5) .. (0,0) .. controls (0,-0.5) and (0,-1) .. (0.5,-1);
        \draw[green, thick] (0.5,3) -- (2.5,3);
        \draw[orange, thick] (0.5,-1) -- (2.5,-1);
    
\draw(-1,-2) edge[->](-1,1)  -- (-1,4) (4,4) -- (4,-2)edge[->](4,1) ;
\draw[dashed] (-1,4) -- (4,4) (-1,-2) -- (4,-2);

\draw[fill=green] (2.5,3) .. controls (2,3) and (2,2.25) .. (1.5,2.25) .. controls (1,2.25) and (1,3) .. (0.5,3) .. controls (1,3) and (1,3.75) .. (1.5,3.75) .. controls (2,3.75) and (2,3) .. (2.5,3);
\draw[fill=orange] (0.5,-1) .. controls (1,-1) and (1,-1.75) .. (1.5,-1.75) .. controls (2,-1.75) and (2,-1) .. (2.5,-1) .. controls (2,-1) and (2,-0.25) .. (1.5,-0.25) .. controls (1,-0.25) and (1,-1) .. (0.5,-1);
\draw[very thick, purple] (-1,1) -- (4,1);\end{scope}

\begin{scope}[shift={(4,-8)},scale=0.25]

\fill[white!20]  (4,-2) rectangle (-1,4);

\draw(-1,-2) edge[->](-1,1)  -- (-1,4) (4,4) -- (4,-2)edge[->](4,1) ;
\draw[dashed] (-1,4) -- (4,4) (-1,-2) -- (4,-2);

\draw[very thick, purple] (-1,1) -- (4,1);
\end{scope}

\draw (7.5,-5.5) .. controls (8,-5.5) and (11,-5.5) .. (11.5,-5.5) .. controls (12,-5.5) and (12.5,-6.5) .. (13,-6.5) .. controls (13.5,-6.5) and (14,-6.5) .. (14.5,-6.5);
\draw (5.5,-6) .. controls (5,-6) and (5,-6.5) .. (4.5,-6.5) .. controls (4,-6.5) and (3.5,-6.5) .. (3,-6.5);
\end{tikzpicture}
     \caption{A monotone 2-ended Lagrangian cobordism}
\end{figure} \printbibliography
\end{document}